\documentclass{amsart}
\usepackage[T1]{fontenc}
\usepackage[utf8]{inputenc}
\usepackage{amsmath}
\usepackage{array}
\usepackage{amsthm}
\usepackage{amssymb}
\input xy
\xyoption{all}

\usepackage[cmtip,all]{xy}

\usepackage{stmaryrd}

\usepackage{tikz}
\usetikzlibrary{matrix,arrows,decorations.pathmorphing}
\usepackage{tikz-cd}

\usepackage{enumerate}
\usepackage{mathtools}

\usepackage[french, english]{babel}

\usepackage{xcolor}

\newtheorem{thmintro}{Theorem}

\newtheorem{thm}{Theorem}[section]
\newtheorem{pr}[thm]{Proposition}
\newtheorem{cor}[thm]{Corollary}
\newtheorem{lm}[thm]{Lemma}

\theoremstyle{definition}
\newtheorem{defi}[thm]{Definition}
\newtheorem{defi-prop}[thm]{Proposition-Definition}
\newtheorem{nota}[thm]{Notation}

\newtheorem{ex}[thm]{Example}
\newtheorem*{thmnn}{Theorem}

\newtheorem{exintro}[thmintro]{Example}

\theoremstyle{remark}
\newtheorem{rk}[thm]{Remark}

\newcommand{\A}{{\mathcal{A}}}

\newcommand{\C}{{\mathcal{C}}}
\newcommand{\D}{{\mathcal{D}}}

\newcommand{\Pp}{{\mathcal{P}}}

\newcommand{\M}{{\mathcal{M}}}

\newcommand{\G}{{\mathcal{G}}}

\newcommand{\Md}{\text{-}\mathbf{Mod}}

\newcommand{\Mdd}{\mathbf{Mod}\text{-}}
\newcommand{\Si}{\mathfrak{S}}
\newcommand{\id}{\mathrm{id}}
\newcommand{\FF}{\mathbb{F}}

\newcommand{\res}{{\mathrm{res}}}
\newcommand{\gen}{{\mathrm{gen}}}

\newcommand{\op}{{\mathrm{op}}}
\newcommand{\Fp}{{\mathbb{F}_p}}
\newcommand{\Fq}{{\mathbb{F}_q}}
\newcommand{\KK}{{\mathbb{K}}}

\newcommand{\colim}{\mathop{\mathrm{colim}}}
\newcommand{\Proj}{\mathbf{P}}
\newcommand{\Vect}{\mathbf{V}}

\newcommand{\stdeg}{\deg}   
\newcommand{\GL}{\operatorname{GL}}
\newcommand{\Sp}{\operatorname{Sp}}
\newcommand{\orth}{\operatorname{O}}

\newcommand{\HH}{\underline{\mathrm{H}}}
\newcommand{\rmH}{\mathrm{H}}
\newcommand{\EExt}{\underline{\mathrm{Ext}}}
\newcommand{\Ext}{\mathrm{Ext}}
\newcommand{\Tor}{\mathrm{Tor}}
\newcommand{\Hom}{\mathrm{Hom}}
\newcommand{\End}{\mathrm{End}}

\newcommand{\diag}{\mathfrak{diag}}
\newcommand{\summ}{\mathfrak{sum}}
\newcommand{\iso}{\mathfrak{iso}}

\begin{document}

\title[Homology of strict polynomial functors over $\Fp$-categories]{Homology of strict polynomial functors over $\Fp$-linear additive categories}

\author[A. Djament]{Aur\'elien Djament}
\address{Institut Galil\'ee, Universit\'e Sorbonne Paris Nord, 99 avenue Jean-Baptiste Cl\'ement, 93430 Villetaneuse, France.}
\email{djament@math.univ-paris13.fr}
\urladdr{https://djament.perso.math.cnrs.fr/}
\thanks{This author is partly supported by the projects ChroK (ANR-16-CE40-0003), AlMaRe (ANR-19-CE40-0001-01) and Labex CEMPI (ANR-11-LABX-0007-01)}

\author[A. Touz\'e]{Antoine Touz\'e}
\address{Univ. Lille, CNRS, UMR 8524 - Laboratoire Paul Painlev\'e, F-59000 Lille, France}
\email{antoine.touze@univ-lille.fr}
\urladdr{https://pro.univ-lille.fr/antoine-touze/}
\thanks{This author is partly supported by the project ChroK (ANR-16-CE40-0003) and Labex CEMPI (ANR-11-LABX-0007-01)}



\begin{abstract}
We generalize the strong comparison theorem of Franjou Friedlander Scorichenko and Suslin to the setting of $\mathbb{F}_p$-linear additive categories. Our results have a strong impact in terms of explicit computations of functor homology, and they open the way to new applications to stable homology of groups or to $K$-theory. As an illustration, we prove comparison theorems between cohomologies of classical algebraic groups over infinite perfect fields, in the spirit of a celebrated result of Cline, Parshall, Scott et van der Kallen for finite fields.
\end{abstract}

\maketitle

\selectlanguage{french}
\renewcommand{\abstractname}{R\'esum\'e}
\begin{abstract}
Nous g\'en\'eralisons le th\'eor\`eme de comparaison forte de Franjou, Friedlander, Scorichenko et Suslin au contexte des cat\'egories additives $\Fp$-lin\'eaires. Nos r\'esultats ont un fort impact en termes de calculs explicites d'homologie des foncteurs, et ils ouvrent la voie \`a de nouvelles applications en homologie stable des groupes et en $K$-th\'eorie. Nous illustrons cela en d\'emontrant des th\'eor\`emes de comparaison entre cohomologies de groupes alg\'ebriques classiques sur des corps parfaits infinis, dans l'esprit d'un c\'el\`ebre th\'eor\`eme de Cline, Parshall, Scott et van der Kallen pour les corps finis.
\end{abstract}

\selectlanguage{english}

\section{Introduction}
Let $k$ be a perfect field of positive characteristic. The category $\Pp_k$ of strict polynomial functors over $k$ was originaly introduced by Friedlander and Suslin as a key ingredient in their study of the cohomology of finite group schemes \cite{FS,SFB,SFB2}. The objects of this category are functors from the category of finite dimensional $k$-vector spaces to the category of $k$-vector spaces, having a `base change structure'. Typical examples are the symmetric powers $S^d$, the exterior powers $\Lambda^d$, or the Schur functors of Akin Buchsbaum and Weyman \cite{ABW}. Strict polynomial functors are an avatar of classical representation theory; for example, computing $\Ext$ in $\mathcal{P}_k$ is equivalent to computing $\Ext$ in the category of modules over classical Schur algebras (as e.g. in \cite{Green}).

We denote by $\Vect_k$ the category of finitely generated $k$-vector spaces and, 
following a traditional notation for `rings with several objects', we denote by $k[\Vect_k]\Md$ the category of all functors from $\Vect_k$ to all vector spaces. Forgetting the base change structure of a strict polynomial functor provides a functor (which is an embedding if $k$ is infinite):
\[\Pp_k\to k[\Vect_k]\Md\;.\]
Given a pair of strict polynomial functors $F$ and $G$, a natural question is to compare the extensions $\Ext^*(F,G)$ computed in $\Pp_k$ with those computed in $k[\Vect_k]\Md$. The articles \cite{FFSS,Kuhn-compare} tackled this question when $k$ is a finite field. Let $F^{(r)}$ denote the composition of a strict polynomial functor $F$ with the $r$-th Frobenius twist functor $I^{(r)}$. Then there is a increasing sequence of $\Ext$-groups: 
\[ \cdots\hookrightarrow\Ext^*_{\Pp_{k}}(F^{(r)},G^{(r)})\hookrightarrow \Ext^*_{\Pp_k}(F^{(r+1)},G^{(r+1)}) \hookrightarrow\cdots\]
In analogy with the terminology used for representations of algebraic groups \cite{CPSVdK}, the colimit of this sequence is called the \emph{generic extensions between $F$ and $G$} and it is denoted by $\Ext^*_\gen(F,G)$. The following fundamental result is called the 
\emph{strong comparison theorem} in \cite{FFSS}.
\begin{thmnn}
If $k$ is a finite field of cardinal greater than the degrees of $F$ and $G$, there is a graded isomorphism:
\[\Ext^*_\gen(F,G)\simeq \Ext^*_{k[\Vect_k]}(F,G)\;.\]
\end{thmnn}
The interest of this theorem lies in the fact that $\Ext^*_\gen(F,G)$ is better understood than $\Ext^*_{k[\Vect_k]}(F,G)$. In \cite{FFSS}, the strong comparison theorem was successfully used to calculate the latter when $F$ and $G$ are symmetric, exterior or divided powers and $k$ is finite. Later works  \cite{Chalupnik,TouzeUnivNew,Touze-Survey} have given formulas which show remarkable properties of generic $\Ext$ (over an arbitrary field $k$), allowing many further $\Ext$ computations in $\Fq[\Vect_\Fq]\Md$.

The structure of the category $k[\Vect_k]\Md$ becomes significantly more complicated when $k$ is an infinite field. As Sulin wrote in \cite[Appendix p.717]{FFSS} that the $\Ext$ in $k[\Vect_k]\Md$ `do not seem to be computable unless we are dealing with finite fields'. Our first main result bridges the gap between the familiar world of finite fields and the unknown world of infinite fields. Its simplicity may come as a surprise, given the much higher complexity of $k[\Vect_k]\Md$ over infinite fields $k$.

\begin{thmintro}\label{thm-intro-1}
Let $k$ is an infinite perfect field of positive characteristic. For all strict polynomial functors $F$ and $G$, there is a graded isomorphism:
\[\Ext^*_\gen(F,G)\simeq \Ext^*_{k[\Vect_k]}(F,G)\;.\]
\end{thmintro}

Since generic $\Ext$ are well understood, theorem \ref{thm-intro-1} gives a quick access to the $\Ext$ computations in $k[\Vect_k]\Md$. As an illustration, we obtain in theorem \ref{thm-FFSS-infinite} the infinite field analogue of the $\Ext$-computations of \cite{FFSS} and we also give further new computations in theorem \ref{thm-calcul-kProjk}.

The proof of theorem \ref{thm-intro-1} is quite long and indirect. We deduce it from a more general comparison theorem, which holds when the category $\Vect_k$ of finite-dimensional $k$-vector spaces is replaced by a more general additive category $\A$.
In this general context, it is more convenient to work with $\Tor$-groups rather than with $\Ext$-groups. 
%
%
To be more specific, keep our field $k$ of positive characteristic, and assume that: 
\begin{itemize}
\item $\A$ is an additive $\FF$-linear category over a subfield $\FF\subset k$,
\item $\pi:\A^\op\to k\Md$ and $\rho:\A\to k\Md$ are $\FF$-linear functors,
\item $F,G:\Vect_k\to k\Md$ are strict polynomial functors over $k$.
\end{itemize}
We denote by $\rho^*G$ the composition of $G$ and $\rho$. If $\rho$ has infinite dimensional values then we make sense of this composition by replacing $G$ by its left Kan extension $\overline{G}$ to all vector spaces:
\[\rho^*G := \overline{G}\circ\rho\;.\] 
Functors of this form are important building blocks of the category $k[\A]\Md$ of all functors from $\A$ to $k$-vector spaces. For example, all the simple functors with finite dimensional values are tensor products of such functors, at least if $k$ is algebraically closed, see \cite[Thm~5.5]{DTV}. Similarly the composition $\pi^*F$ of $F$ and $\pi$ yields an object of the category $\Mdd k[\A]$ of all functors  from $\A^\op$ to $k$-vector spaces.
A natural question is to try to compute $\Tor_*^{k[\A]}(\pi^*F,\rho^*G)$. Let $D_{\pi,\rho}:\Vect_k^{\op}\to k\Md$ be the $k$-linear functor defined by 
\[D_{\pi,\rho}(v):= \Hom_k(v,\pi\otimes_{k[\A]}\rho)\;.\]
In the main body of the article, we construct a map for every degree $i$ 
\begin{align}\Tor_i^{k[\A]}(\pi^*F,\rho^*G)\to \Tor_i^{\gen}(D_{\pi,\rho}^*F,G)\label{eqn-comp-Tor-intro}
\end{align}
where $\Tor_*^{\gen}$ is the $\Tor$-analogue of the generic $\Ext$ of theorem \ref{thm-intro-1}. Note that the source of the comparison map \eqref{eqn-comp-Tor-intro} consists of mysterious $\Tor$-groups which are out of reach of computation, while its target consists of relatively well understood $\Tor$-groups, which can be computed in terms of $\Tor$ over classical Schur algebras.    

Our second result is a connectivity result for this comparison map, thereby giving a way to compute $\Tor_*^{k[\A]}(\pi^*F,\rho^*G)$. The $\Tor_*^{k\otimes_\mathbb{Z}\A}$ appearing in the homological vanishing condition refers to the $\Tor$-groups calculated in the category of additive functors, as opposed to $\Tor_*^{k[\A]}$ which are calculated in the category of all functors. The relation between the two is well understood by \cite{DT-add}.
\begin{thmintro}\label{thm-intro-2}
Assume that the cardinal of $\FF$ is greater than the degrees of $F$ and $G$, and that $\Tor^{k\otimes_\mathbb{Z}\A}_i(\pi,\rho)=0$ for $0<i<e$. 
Then the map \eqref{eqn-comp-Tor-intro} is an isomorphism if $0\le i<e$, and surjective if $i=e$.
\end{thmintro}

The condition on the size of $\FF$ in theorem \ref{thm-intro-2} can be removed (i.e. we can take $\FF$ smaller than the degrees of $F$ and $G$), but the price to pay is that the comparison map \eqref{eqn-comp-Tor-intro} has to be modified in order to take into account the too small size of $\FF$, in particular $D_{\pi,\rho}$ has to be replaced by a similar, but more complicated, $k$-linear functor. This leads to our \emph{generalized comparison theorem}, which is the main comparison result of the paper. We refer the reader to theorem \ref{thm-gen-comp} for the exact statement. Let us only point out here the counter-intuitive fact that the case of the very small field $\FF=\Fp$ in our generalized comparison theorem is the key to prove theorem \ref{thm-intro-2} for infinite $\FF$, hence to finally prove theorem \ref{thm-intro-1}.

\subsection*{Consequences for the stable homology of general linear groups}
Our motivation to study homological algebra in functor categories is the close relations with the stable cohomology of $\GL_n(R)$. To be more specific, let $\Proj_R$ be the category of finitely generated projective right $R$-modules over a ring $R$. If $T: \Proj_R\to k\Md$ is a functor, the vector space $T(R^n)$ is naturally endowed with an action of $\GL_n(R)$, and taking the colimit on $n$ yields a left $k$-linear representation $T_\infty$ of the infinite general linear group $\GL_\infty(R)$. Similarly every functor $T':\Proj_R^\op\to k\Md$ induces a right $k$-linear representation $T'_\infty$, which may be viewed as a left $k$-linear representation by letting each element $g$ of $\GL_\infty(R)$ act as $g^{-1}$.

It is known that when $T'=\pi^*F$ and $T=\rho^*G$ with $\pi$, $\rho$, $F$ and $G$ as above, and if the ring $R$ has a finite stable rank, the canonical morphism:
\begin{align}
\rmH_*(\GL_{n}(R);T'(R^n)\otimes_k T(R^n))\to \rmH_*(\GL_{\infty}(R);T'_\infty\otimes_k T_\infty)\label{eq-stab}
\end{align}
is an isomorphism in low degrees (with an explicit isomorphism range, increasing linearly with $n$), see \cite{Dwyer}, \cite{vdK-stab} or \cite{RWW}. The right-hand side may be therefore referred to as the stable homology of $\GL_n(R)$.

Now it follows from \cite{DjaR} that when $T'=\pi^*F$ and $T=\rho^*G$, there is an isomorphism for all rings $R$ (even when no stable rank hypothesis is satisfied):
\begin{align}
\rmH_*(\GL_{\infty}(R);T'_\infty\otimes_k T_\infty)
\simeq \rmH_*(\GL_\infty(R);k)\otimes_k \Tor_*^{k[\Proj_R]}(T',T)\;.\label{eq-iso-motiv}
\end{align}
Thus our results actually give access to the homology of general linear groups with nontrivial coefficients (modulo the knowledge of the homology with trivial coefficients). 

We illustrate the computational applicability of our results with concrete homological computations in section \ref{subsec-ex-GL}. For instance, let $R_\infty$ denote the right $R$-module with countable basis $(e_i)_{i\ge 1}$ and with left action of $\GL_\infty(R)$ given by matrix multiplication
$[a_{ij}]\cdot e_j:=[a_{ij}] e_j = \sum_i a_{ij}e_i$, and let $R_\infty'$ be its stable dual, that is,  
the same $R$-module with action of $\GL_\infty(R)$ given by multiplication by transposed inverse matrices: $A\cdot v := (A^{-1})^{\mathrm{T}}v$. If $R$ is a $k$-algebra, then $R_\infty$ and $R_\infty'$ can be regarded as $k$-linear representations of $\GL_\infty(R)$ and we can consider the $d$-th exterior power (over $k$) on a direct sum of copies of these representations: 
\[\Lambda^d_{\ell m} = \Lambda^d(\underbrace{R_\infty' \oplus\cdots\oplus R_\infty'}_{\text{$\ell$ copies}}\oplus \underbrace{R_\infty\oplus\cdots\oplus R_\infty}_{\text{$m$ copies}})\;.\]
\begin{exintro}\label{ex-intro-1}
Assume that $k$ is an infinite perfect field of positive characteristic, and that $R$ is a $k$-algebra. Then for all nonnegative integers $d$, the graded $k$-vector space $\rmH_*(\GL_\infty(R);\Lambda^d_{\ell m})$ is zero if $d$ is odd, and if $d$ is even it is isomorphic to 
\[ \rmH_*(\GL_\infty(R);k)\otimes_k S^{d/2}(T_{\ell m}) \]
where $S^{d/2}(T_{\ell m})$ is $\frac{d}{2}$-th symmetric power of the graded $k$-vector space $T_{\ell m}$ which equals $\Hom_R(R^m,R^\ell)$ in degree $2i$ for all $i\ge 0$, and which is zero in the other degrees. 
\end{exintro}

\subsection*{Rational versus discrete cohomology of classical groups}
Besides their computational power, our results also have theoretical consequences. We single out an application to the cohomology of classical groups in section \ref{sec-rat-disc}.
Namely, if $G$ is an algebraic group over $k$, we denote by $\EExt_G^*(V,W)$ its algebraic group cohomology as in \cite{Jantzen}
and by $\Ext_{G}^*(V,W)$ the cohomology of its underlying discrete group of $k$-points as in \cite{Brown}. These two cohomologies are related by a comparison map:
\[\EExt^*_{G}(V,W)\to \Ext_{G}^*(V,W)\;.\]
Assume that $G$ is defined over the prime subfield $\Fp$ of $k$. Restricting a representation $V$ along the Frobenius group morphism $\phi:G\to G$
yields the \emph{twisted representation} $V^{[r]}$. If $k$ is perfect, $\phi$ is an isomorphism in the category of discrete groups, so the comparison map becomes:
\[\EExt^*_{G}(V^{[r]},W^{[r]})\to \Ext^*_{G}(V^{[r]},W^{[r]})\simeq \Ext^*_{G}(V,W)\;.\qquad(*)\]
If $G$ is reductive (e.g. $G=\GL_n(k)$), the left-hand side does not depend on $r$ in low degrees if $r$ is big enough, and it is called the \emph{generic extensions} of $G$. 
A celebrated theorem of Cline, Parshall, Scott and van der Kallen \cite[Main Thm (6.6)]{CPSVdK} shows that if $k$ is a big finite field, $r$ is big enough and $G$ is  split over $\Fp$, the map $(*)$ is an isomorphism in low degrees (the isomorphism range grows with the size of $k$).

As an application of theorem \ref{thm-intro-1}, we prove a similar result for infinite perfect fields $k$ when $G=\GL_n(k)$ in theorem \ref{thm-comp-GL}.
\begin{thmintro}\label{thm-intro-4}
Assume that the perfect field $k$ is infinite, and that $V$ and $W$ are two polynomial representations of $G=\GL_n(k)$ of degrees less or equal to $d$. Let $r$ be a nonnegative integer such that $n\ge \max\{dp^r,4p^r+2d+1\}$. Then the map $(*)$ is an isomorphism in degrees $i<2p^r$ and it is injective in degree $i=2p^r$.

\end{thmintro}
Analogous results for symplectic and orthogonal groups are given in theorem \ref{thm-comp-OSp}. 
Let us emphasize two points regarding these comparison results. 
\begin{enumerate}[1.]
\item 
As many perfect fields do not contain big finite subfields, there is no hope to recover our results from the results of \cite{CPSVdK} by taking colimits over finite subfields. Actually, \cite{CPSVdK} uses that the polynomials on a $k$-vector space $V$ are a good approximation of discrete maps from $V$ to $k$ over big finite fields $k$. Namely the natural forgetful map:
$k[V]\to \mathrm{Map}(V,k)$
is surjective, and it is an isomorphism in polynomial degree less than the cardinality of $k$. This fact obviously fails for infinite perfect fields, which makes our results rather unexpected.
\item 
Our cohomological comparison deals with \emph{stable} cohomology, in the sense that $n$ is large. One can wonder what happens when $n$ is small. This question goes beyond the current understanding of homology of groups, even in the simplest case $W=V=k$ and $G=\GL_n(k)$. In this case, it is known that the left-hand side of $(*)$ is zero in positive degrees \cite[II 4.11]{Jantzen}. In sharp contrast, the right hand-side of $(*)$ is mysterious for small values of $n$, and still under investigation, see e.g. \cite{Mirzaii,GKRW}. 
\end{enumerate}

\subsection*{Relations with other results and future work}

\subsubsection*{Computations in $\Fp[\Vect_\Fp]\Md$ and the Steenrod algebra}
Homological algebra in $\Fp[\Vect_\Fp]\Md$ was actually considered before the relation with the homology of $\GL_\infty(\Fp)$ was established. The initial motivation was the relation with the category of unstable modules over the Steenrod algebra, see \cite{HLS} and \cite{KI,KII,KIII}. Our results might have applications in this original context. Indeed, theorems \ref{thm-interm-strong} and \ref{thm-verystrong} remove the `big field assumption' from the strong comparison theorem, hence they could be a way to obtain a better understanding of the homological algebra in $\Fp[\Vect_\Fp]\Md$ from the computations of generic $\Ext$.

\subsubsection*{Comparison without homological vanishing assumption}
Theorem \ref{thm-intro-2}, and the generalized comparison theorem \ref{thm-gen-comp} compute $\Tor_*^{k[\A]}(\pi^*F,\rho^*G)$ under a homological vanishing condition on the additive functors $\pi$ and $\rho$. 
In \cite{DT-add} we give a computation of these $\Tor$, without assumption on $\pi$ and $\rho$, but assuming instead the strong hypothesis that $F$ and $G$ are direct summands of tensor powers. It is not clear to us if there is a closed formula computing these $\Tor$ without any of these assumptions, however there should be at least a formula in low degrees.
 
\subsubsection*{Comparison without $\Fp$-linearity}
Theorem \ref{thm-intro-2}, and the generalized comparison theorem \ref{thm-gen-comp} compute $\Tor_*^{k[\A]}(\pi^*F,\rho^*G)$ when the source category $\A$ is $\Fp$-linear. In view of the applications to group homology and K-theory, it would be interesting to remove this assumption. The few computations known in this context \cite{FPmaclane,PiraZnZ} show that new homological phenomena appear when $\Fp$-linearity is removed, which makes it a challenging problem.

\section{Recollections of functor categories}

\subsection{Ext and Tor in functor categories}\label{subsec-recoll-ord} If $\C$ is a svelte (= essentially small) category and $k$ is a commutative ring, we denote by $k[\C]\Md$ the category whose objects are the functors from $\C$ to $k$-modules and whose morphisms are the natural transformations (with the usual composition of natural natural transformations). To emphasize the analogy with modules over a ring, and to better distinguish between the source category $\C$ and the functor category $k[\C]\Md$, we denote by $\C(c,d)$ the $\Hom$-sets in $\C$ while we use the notation $\Hom_{k[\C]}(F,G)$ for the morphisms between two functors $F$ and $G$.   
We refer the reader to \cite{Mi72} for a detailed study of $k[\C]\Md$; we only recall here the salient facts which will be useful to us. 

\begin{rk}
Although the main focus of the article is the situation where $k$ is a field, the recollections are given over a commutative ring $k$ because this generality does not bring additional complexity. Moreover, the results of section \ref{sec-gen-prelim-comparison}, which rely on the material recalled here, are valid for an arbitrary commutative ring $k$.
\end{rk}

\subsubsection{Abelian structure}
The category $k[\C]\Md$ is abelian, bicomplete with enough projectives and injectives. To be more specific, limits and colimits are computed  objectwise, e.g. the direct sum of two functors is defined on all objects $c$ by $(F\oplus G)(c)=F(c)\bigoplus G(c)$. The \emph{standard projectives} are the functors $P^c =k[\C(c,-)]$ which map every object $d$ to the free $k$-module on $\C(c,d)$. The Yoneda lemma yields an isomorphism, natural with respect to $F$ and $c$:
\[\Hom_{k[\C]}(P^c,F)\simeq F(c)\;.\]
Every object of $k[\C]\Md$ has a projective resolution by direct sums of standard projectives. 

\subsubsection{Tensor product and $\Tor$}\label{subsubsec-tensTor}
Let $\Mdd k[\C]$ denote the category of functors from $\C^\op$ to $k$-modules (in other words $\Mdd k[\C]=k[\C^\op]\Md$). There is a tensor product over $\C$, which generalizes the tensor product of modules over a $k$-algebra:
\[ \otimes_{k[\C]}: \Mdd k[\C]\times k[\C]\Md\to k\Md\;.\]
To be more specific, $E\otimes_{k[\C]} F$ is the coend $\int^{c\in\C} E(c)\otimes_kF(c)$, which can be concretely computed as the quotient of the $k$-module $\bigoplus _c E(c)\otimes_k F(c)$ by the relations $E(f)(y)\otimes x= y\otimes F(f)(x)$ for all $x\in F(c)$, all $y\in E(d)$ and all $f:c\to d$.

For all $E$ in $\Mdd k[\C]$ and all $k$-modules $M$, let $D_M E$ in $k[\C]\Md$ denote the functor defined by $(D_ME)(c)=\Hom_k(E(c),M)$. Then the tensor product over $\C$ is characterized by the following adjunction isomorphism, natural with respect to $E$, $F$ and $M$:
\begin{align}
\Hom_k(E\otimes_{k[\C]} F,M)\simeq \Hom_{k[\C]}(F,D_M E)\;.\label{eqn-Formule-Cartan}
\end{align}

We denote by $\Tor^{k[C]}_*(E,F)$ the derived functors of the tensor product over $\C$. The previous adjunction isomorphism may be derived, to give a natural duality isomorphism, for all injective $k$-modules $M$ and for all degrees $i$:
\[\Hom_k(\Tor_i^{k[\C]}(E,F),M)\simeq \Ext^i_{k[\C]}(F,D_M E)\;.\]


\subsubsection{Restrictions and adjunctions in the source category}\label{subsubsec-res}
We denote by $\phi^*F$ the composition of a functor $\phi:\C\to \D$ with an object $F$ of $k[\D]\Md$. This induces an exact functor
\begin{align*}
\phi^*:k[\D]\Md\to k[\C]\Md\;,
\end{align*}
hence graded $k$-linear maps on the level of $\Ext$ and $\Tor$, that we call restriction maps, and that we denote by $\res^\phi$ and $\res_\phi$: 
\begin{align*}
&\res^\phi:\Ext^*_{k[\D]}(F,G)\to \Ext^*_{k[\C]}(\phi^*F,\phi^*G)\;,\\
&\res_\phi:\Tor_*^{k[\C]}(\phi^*E,\phi^*F)\to \Tor^{k[\D]}_*(E,F)\;.
\end{align*}
The natural duality isomorphism between $\Ext$ and $\Tor$ is natural with respect to restriction maps. That is, for all injective $k$-modules $M$ and for all degrees $i$ we have a commutative square:
\begin{equation}
\label{dgm-nat-duality1}
\begin{tikzcd}
\Hom_k(\Tor_i^{k[\C]}(\phi^*E,\phi^*F),M)\ar{rr}{\simeq}&&\Ext^i_{k[\C]}(\phi^*F,\phi^*D_ME)\\
\Hom_k(\Tor_i^{k[\D]}(E,F),M)\ar{u}{\Hom_k(\res_\phi,M)}\ar{rr}{\simeq}&&\Ext^i_{k[\D]}(F,D_ME)\ar{u}{\res^\phi}
\end{tikzcd}
\end{equation}

The next proposition is an important tool for computations. The $\Ext$-version can be found e.g. in \cite[Lm 1.3 and Lm 1.5]{Pira-Pan}. The proof is a straightforward check.
\begin{pr}\label{pr-iso-Ext-explicit}
Let $\phi:\C\leftrightarrows \D:\psi$ be an adjoint pair and let $u:\id\to \psi\circ\phi$ and $e:\phi\circ\psi\to\id$ denote the unit and the counit of an adjunction. Then the following composition is an isomorphism:
\[\Ext^*_{k[\D]}(\psi^*F,G)\xrightarrow[]{\res^\phi}\Ext^*_{k[\C]}(\phi^*\psi^*F,\phi^*G)\xrightarrow[]{\Ext^*_{k[\C]}(F(u),G)} \Ext^*_{k[\C]}(F,\phi^*G)\;,\]
whose inverse is given by the composition of $\res^\psi$ and $\Ext^*_{k[\D]}(F,G(e))$. Similarly, there is a $\Tor$-isomorphism given by the following composite map:
\[\Tor_*^{k[\C]}(\phi^*E,F)\xrightarrow[]{\Tor_*^{k[\C]}(\phi^*E,F(u))}\Tor_*^{k[\C]}(\phi^*E,\phi^*\psi^*F) \xrightarrow[]{\res^\phi}\Tor_*^{k[\D]}(E,\psi^*F)\;,\]
whose inverse is given by the composition of $\Tor_*^{k[\D]}(E(e),\psi^*F)$ and $\res_\psi$.
\end{pr}


\subsection{Strict polynomial functors}\label{subsec-str-pol}
Assume that $k$ is a field, and let $\Vect_k$ denote the category of finite-dimensional $k$-vector spaces. We recall here some basic facts that we will need regarding strict polynomial functors. As a small difference with \cite{FS}, we allow our strict polynomial functors to have infinite dimensional values -- this is needed for example in the statement of the generalized comparison theorem, in which $F^\dag$ may have infinite dimensional values. This difference does not affect the results of \cite[Section 2 and 3]{FS} in an essential way.
\subsubsection{The abelian structure}
We denote by $\Pp_k$ the category of strict polynomial functors of bounded degree over $k$ (possibly with infinite dimensional values), and natural transformations of strict polynomial functors. It is an abelian category, and there is a (faithful, exact) forgetful functor: 
\[\Pp_k\to k[\Vect_k]\Md\]
which allows to think of a strict polynomial functor as a functor $F:\Vect_k\to k\Md$ equiped with some additional structure. If $k$ is infinite, this forgetful functor is fully faithful. 

The category $\Pp_k$ has enough injectives and projectives. To be more specific, the standard injective objects have the form $S^d_V(-)=S^d(V\otimes_k-)$ and the standard projective objects have the form $\Gamma^{d\, V}(-)=\Gamma^d(\Hom_k(V,-))$, where $V$ denotes a vector space, and $S^d(U)=(U^{\otimes d})_{\Si_d}$ stands for the $d$-th symmetric power of a $k$-vector space $U$, while $\Gamma^d(U)=(U^{\otimes d})^{\Si_d}$ stands for its $d$-th divided power. Every strict polynomial functor has an injective resolution by products of standard injectives, and a projective resolution by direct sums of standard projectives.

\begin{rk}\label{rk-abuse}
We will usually make no notational distinction between a strict polynomial functor $F$ and the underlying ordinary functor, i.e. the image of $F$ by the forgetful functor $\Pp_k\to k[\Vect_k]\Md$. For example, we will say that the forgetful functor induces a graded morphism: $\Ext^*_{\Pp_k}(F,G)\to \Ext^*_{k[\Vect_k]}(F,G)$.
Pushing this (slight) abuse of notation further, if $\phi:\C\to \Vect_k$ is a functor, we will also denote by $\phi^*$  the composition of the map $\phi^*$ of section \ref{subsubsec-res} with the forgetful functor:
\[\Pp_k\to k[\Vect_k]\Md\xrightarrow[]{\phi^*}k[\A]\Md\;.\]
\end{rk}

\subsubsection{Homogeneous functors}
Let $\Pp_{d,k}$ denote the full subcategory of homogeneous strict polynomial functors of degree $d$; this is the smallest abelian subcategory of $\Pp_k$ which contains the functors $\Gamma^{d\,V}$ and $S^d_V$ for all $V$. There is a direct sum decomposition:
\[\Pp_k=\bigoplus_{d\ge 0}\Pp_{d,k}\;.\]
Moreover, each category $\Pp_{d,k}$ is equivalent to the category $S(n,d)\Md$ of modules over the Schur algebra $S(n,d)=\End_{\Si_d}((k^n)^{\otimes d})$ provided $n\ge d$. This equivalence of categories sends a functor $F$ to the $k$-vector space $F(k^n)$, equipped with the canonical action of $S(n,d)$.
\begin{rk}
Each subcategory $\Pp_{d,k}$ has all direct sums and products. The forgetful functor $\Pp_{d,k}\to k[\Vect_k]\Md$ preserves arbitrary direct sums and arbitrary products. 
\end{rk}
\begin{rk}
The subcategory $\Pp_{d,k}$ may be equivalently described \cite{Kuhn-compare, Pira-Pan} as the category of $k$-linear functors $\Gamma^d\Vect_k\to k\Md$ where $\Gamma^d\Vect_k$ denotes the category whose objects are the finite dimensional vector spaces and $\Hom_{\Gamma^d\Vect_k}(V,W)=\Hom_{\Si_d}(V^{\otimes d},W^{\otimes d})$. With this description, the forgetful functor is given by precomposition with the functor $\gamma^d:\Vect_k\to \Gamma^d\Vect_k$ which is the identity on objects and which satisfies $\gamma^d(f)=f^{\otimes d}$. Moreover, $S(n,d)=\End_{\Gamma^d\Vect_k}(k^n)$ and the equivalence of categories $\Pp_k\simeq S(n,d)\Md$ is simply obtained by restricting a functor to the full subcategory of $\Gamma^d\Vect_k$ supported by $k^n$.
\end{rk}

\subsubsection{Variants of the category $\Pp_k$}
The definition of strict polynomial functors in \cite{FS} has obvious variants, such as the category $\Pp_k'$ of contravariant strict polynomial functors. Let $^\vee v=\Hom_k(v,k)$ denote the dual of a vector space $v$. Then we denote by $F^\vee$ the composition of $F$ with the duality functor $^\vee-:\Vect_k^\op\to \Vect_k$. Thus $F^\vee(v)=F({}^\vee v)$. This construction induces an equivalence of categories
\[\Pp_k\xrightarrow[]{\simeq}\Pp_k'\;,\quad F\mapsto F^\vee\;.\]

\subsubsection{Tensor functors and $\Tor$}
There is a tensor product: 
\[-\otimes_{\Pp_k}-:\Pp_k'\times \Pp_k\to k\Md\]
which is compatible with the tensor product over $k[\Vect_k]$ in the sense that the forgetful functor induces a canonical morphism
\begin{align} E\otimes_{k[\Vect_k]}F\to  E\otimes_{\Pp_k}F\;.\label{eqn-cantens}
\end{align}
If $V$ is a vector space, we extend the notation of section \ref{subsubsec-tensTor} to strict polynomial functors by letting $D_V E$ denote the strict polynomial functor such that $(D_VE)(v)=\Hom_k(E(v),V)$. Then the tensor product is characterized by the isomorphism (natural with respect to $E$, $F$, $V$)
\begin{align}\Hom_k(E\otimes_{\Pp_k}F,V)\simeq \Hom_{\Pp_k}(F,D_V E)\;.\label{eqn-cartan-strict}
\end{align}
This isomorphism shows that the canonical map \eqref{eqn-cantens} is always surjective (because the forgetful functor is faithful), and it is an isomorphism if $k$ is infinite. Since there are no nonzero morphisms between two homogeneous strict polynomial functors of different degrees, this isomorphism also shows that $E\otimes_{\Pp_k}F=0$ if $E$ and $F$ are homogeneous of different degrees.

We denote by $\Tor_*^{\Pp_k}(E,F)$ the derived functors of the tensor product. Then the isomorphism \eqref{eqn-cartan-strict} can be derived, it yields an isomorphism
\[\Hom_k(\Tor_i^{\Pp_k}(E,F),V)\simeq \Ext^i_{\Pp_k}(F,D_VE)\]

\begin{rk}
The equivalence $\Pp_{d,k}\simeq S(n,d)\Md$ is strongly monoidal, thus the tensor product is just the usual tensor product over Schur algebras in disguise:\[\Tor_*^{\Pp_k}(E,F)\simeq \Tor_*^{S(n,d)}(E(k^n),F(k^{n}))\;.\] 
In the description of $\Pp_{d,k}$ as the category of $k$-linear functors from $\Gamma^d\Vect_k$ to $k\Md$, the tensor product identifies with the tensor product over $\Gamma^d\Vect_k$.
\end{rk}

\subsubsection{Composition and restriction} The composition of a strict polynomial functor $F$ of degree $d$ with a strict polynomial functor $G$ of degree $e$ yields a strict polynomial functor $F\circ G$ of degree $de$. Precomposition by $G$ yields an exact functor $-\circ G:\Pp_k\to \Pp_k$ and $-\circ G:\Pp_k'\to \Pp_k'$, hence graded maps:
\begin{align*}
&\Ext^*_{\Pp_k}(F,F')\to \Ext^*_{\Pp_k}(F\circ G,F'\circ G)\;,\\
&\Tor_*^{\Pp_k}(E\circ G,F\circ G)\to \Tor_*^{\Pp_k}(E\circ G,F\circ G)\;.
\end{align*}
which fit into a commutative diagrams:
\begin{equation}
\label{dgm-nat-duality2}
\begin{tikzcd}
\Hom_k(\Tor_i^{\Pp_k}(E\circ G,F\circ G),V)\ar{rr}{\simeq}&&\Ext^i_{\Pp_k}(F\circ G,D_V(E\circ G))\\
\Hom_k(\Tor_i^{\Pp_k}(E,F),V)\ar{u}{}\ar{rr}{\simeq}&&\Ext^i_{\Pp_k}(F,D_VE)\ar{u}{}
\end{tikzcd}
\end{equation}

\subsection{Frobenius twists and generic homology}\label{subsec-Frob-twist}

Throughout this section, $k$ is a perfect field, of positive characteristic $p$.

\subsubsection{Frobenius twists}\label{subsubsec-Frob-discrete}
Let $\FF$ be a perfect field of positive characteristic $p$. For all integers $r$ and for all $\FF$-vector spaces $v$ we denote by $^{(r)}v$ the $\FF$-vector space which equals $v$ as an abelian group, with action of $\FF$ given by 
\[\lambda\cdot x:= \lambda^{p^{-r}}x\;.\]
We note that $^{(r)}-$ is an additive endofunctor of $\FF$-vector spaces which preserves dimension. 
Moreover $^{(0)}v= v$ and $^{(s)}({}^{(r)}v)={}^{(s+r)}v$, hence  $^{(r)}-$ is a self-equivalence of the category of $\FF$-vector spaces, with inverse $^{(-r)}-$.

\begin{nota}\label{nota-twist-ord}
For all (contravariant or covariant) functors $F:\Vect_\FF\to k\Md$, we denote by $F^{(r)}: \Vect_\FF\to k\Md$ the functor such that 
$F^{(r)}(v) := F({}^{(r)}v)$.
\end{nota}

Assume now that $\FF=k$. Then if $r\ge 0$, the functor $^{(r)}-: \Vect_k\to k\Md$ is the underlying ordinary functor of a certain $p^r$-homogeneous strict polynomial functor, called the \emph{$r$-th Frobenius twist functor} and denoted by $I^{(r)}$. One has $(I^{(r)})^{(s)}= I^{(r+s)}$ for all nonnegative integers $r$ and $s$.
The following notation is the analogue of notation \ref{nota-twist-ord} for strict polynomial functors.
\begin{nota}\label{nota-twist}
For all strict polynomial functors $F$ of degree $d$, we denote by $F^{(r)}$ the strict polynomial functor of degree $dp^r$ defined as $F^{(r)}:= F\circ I^{(r)}$.
\end{nota}

\subsubsection{Generic $\Ext$ and generic $\Tor$}
For all positive integers $r$, precomposition by the equivalence of categories ${}^{(r)}-:\Vect_k\to \Vect_k$ yields a graded isomorphism:
\[\Ext^*_{k[\Vect_k]}(F,G)\xrightarrow[]{\simeq} \Ext^*_{k[\Vect_k]}(F^{(r)},G^{(r)})\]
with inverse given by precomposition by ${}^{(-r)}-$. 
The situation is quite different in the realm of strict polynomial functors. Indeed, precomposition by the strict polynomial functor $I^{(r)}$ induces a graded map  
\begin{align*}\Ext^*_{\Pp_k}(F,G)\xrightarrow[]{} \Ext^*_{\Pp_k}(F^{(r)},G^{(r)})
\end{align*}
but since $I^{(r)}$ has no `strict polynomial inverse' (i.e. there is no strict polynomial functor $I^{(-r)}$!), there is no reason why this graded map should be an isomorphism. 
The next result was first established in \cite[Cor 1.3 and Cor 4.6]{FFSS}. 
\begin{defi-prop}\label{pdef-gen-Ext}
Let $F$ and $G$ be two strict polynomial functors. The maps given by precomposition by $I^{(1)}$:
\[\Ext^i_{\Pp_k}(F^{(r)},G^{(r)})\to \Ext^i_{\Pp_k}(F^{(r+1)},G^{(r+1)})\]
are always injective, and they are isomorphisms if $i<2p^r$. The stable value is called the \emph{generic extensions of degree $i$} and denoted by $\Ext^i_\gen(F,G)$:
\begin{align*}
\Ext^i_\gen(F,G)&:=\colim_{r}\Ext^i_{\Pp_k}(F^{(r)},G^{(r)})= \Ext^i_{\Pp_k}(F^{(r)},G^{(r)})\text{ if $r\gg 0$.}
\end{align*}
\end{defi-prop}

We refer the reader to \cite{Touze-Survey} for a survey of generic extensions and formulas computing them (which simplify and generalize the computations of \cite{FFSS}).
Using the commutative diagram \eqref{dgm-nat-duality2}, we can dualize the $\Ext$ situation to define generic $\Tor$.

\begin{defi-prop}\label{pdef-gen-Tor}
Let $E$ and $G$ be two strict polynomial functors, with $E$ contravariant. The maps given by precomposition by $I^{(1)}$
\[\Tor_i^{\Pp_k}(E^{(r+1)},G^{(r+1)})\to \Tor_i^{\Pp_k}(E^{(r)},G^{(r)})\]
are always surjective, and they are isomorphisms if $i<2p^r$. The stable value is called the \emph{generic torsion of degree $i$}
and denoted by $\Tor^\gen_i(E,G)$:
\begin{align*}\Tor_i^\gen(E,G)&:=\lim_{r}\Tor_i^{\Pp_k}(E^{(r)},G^{(r)})= \Tor_i^{\Pp_k}(E^{(r)},G^{(r)})\text{ for $r\gg 0$.}
\end{align*}
\end{defi-prop}

\section{Homology of strict polynomial functors over $\Vect_\Fq$.}
Throughout the section, $k$ is a perfect field of positive characteristic $p$, and $\Fq$ is a finite field of cardinal $q=p^r$.
We will elaborate on the results of \cite{FFSS} to prove the generalized comparison theorem over $\A=\Vect_\Fq$ (the category of finite-dimensional $\Fq$-vector spaces). The proof is done in two steps: we first establish the theorem when the field $\Fq$ is big enough in section \ref{subsec-strong-comparison-Fq}, and then we extend the theorem to arbitrary finite fields in section \ref{subsec-generalized-comparison-Fq}.

\subsection{The strong comparison map}\label{subsec-strong-comparison-Fq}
We assume that $k$ contains a perfect subfield $\FF$, and we let $t:\Vect_\FF\to \Vect_k$ denote the extension of scalars: $t(v)=k\otimes_\FF v$. As explained in remark \ref{rk-abuse}, we slightly abuse notations and we denote by $t^*$ the exact functor defined as the composition:
\[t^*: \Pp_k\to k[\Vect_k]\Md\to k[\Vect_\FF]\Md\]
where the first functor is the forgetful functor and the second functor is precomposition by $t$.
\begin{defi}\label{defi-strong-comp-map}
The \emph{strong comparison map} (associated to the extension of perfect fields $\FF\subset k$) is the graded $k$-linear map $\Phi_\FF$ defined for $n\gg 0$ by restriction along $t^*$ and along the equivalence of categories $^{(-n)}-:\Vect_\FF\to \Vect_\FF$ given by the Frobenius twist:
\begin{align*}
\Phi_\FF:\begin{array}[t]{c}
\Ext^i_{\Pp_k}(F^{(n)},G^{(n)})\\
= \Ext^i_{\gen}(F,G)
\end{array}\xrightarrow[]{} \Ext^i_{k[\Vect_\FF]}(t^*F^{(n)},t^*G^{(n)})\xrightarrow[]{\simeq} \Ext^i_{k[\Vect_\FF]}(t^*F,t^*G)\;. 
\end{align*}
\end{defi}
\begin{rk}
The strong comparison map is independent of $n$ (provided $n$ is big enough). This follows from the commutative squares (in which the horizontal maps are induced by $t^*$ and the vertical ones by restriction along $I^{(1)}$ and $^{(1)}-$, together with the natural isomorphisms $k\otimes_\FF {}^{(1)}v \simeq {}^{(1)}(k\otimes_\FF v)$):
\[
\begin{tikzcd}
\Ext^i_{\Pp_k}(F^{(n+1)},G^{(n+1)})\ar{r}&\Ext^i_{k[\Vect_\FF]}(t^*F^{(n+1)},t^*G^{(n+1)})\\
\Ext^i_{\Pp_k}(F^{(n)},G^{(n)})\ar{r}\ar{u}{\simeq}&\Ext^i_{k[\Vect_\FF]}(t^*F^{(n)},t^*G^{(n)})\ar{u}{\simeq}
\end{tikzcd}\;.
\]
\end{rk}

The next result follows from the strong comparison theorem \cite[Thm 3.10]{FFSS}. It will be extended to all perfect infinite fields $\FF$ in section \ref{sec-special-cases}.
\begin{thm}\label{thm-strong-comparison}
If $\FF$ is a finite field with $q$ elements, if $F$ and $G$ are two strict polynomial functors of degrees less than $q$, the strong comparison map $\Phi_\FF$ is an isomorphism in all degrees $i$.
\end{thm}
\begin{proof}
Theorem \ref{thm-strong-comparison} slightly generalizes the strong comparison theorem of \cite{FFSS} in two ways. Firstly, contrarily to \cite{FFSS}, we do not assume that $k=\Fq$. Secondly, as we recall it in section \ref{subsec-str-pol}, we allow our strict polynomial functors to have infinite-dimensional values.

We overcome these two technical points as follows. The standard projective objects of $\Pp_k$ are the divided power functors $\Gamma^{d,s}=\Gamma^d(\Hom_k(k^s,-))$ and the standard injectives are the symmetric power functors $S^{e,s}=S^e(k^s\otimes -)$. These two kinds of functors commute with base change in the following sense: there are canonical isomorphisms
\[t^*\Gamma^{d,s}(v)\simeq\Gamma^{d,s}_\FF(v)\otimes_\FF k \text{ and } t^*S^{e,s}(v)\simeq S^{e,s}_\FF(v)\otimes_\FF k\]
where the indices $\FF$ indicate their counterparts in the category $\Pp_\FF$ of strict polynomial functors over $\FF$. There is a commutative square (compare \cite[Prop 3.8]{FFSS})
\[
\begin{tikzcd}
\Ext^i_{\Pp_\FF}(\Gamma^{d,s\,(n)}_\FF, S^{e,s\,(n)}_\FF )\otimes_\FF k\ar{r}{\simeq}\ar{d}& \Ext^i_{\Pp_k}(\Gamma^{d,s\,(n)}, S^{e,s\,(n)})\ar{d}{\Phi_\FF}\\
\Ext^i_{\FF[\Vect_\FF]}(\Gamma^{d,s}_\FF, S^{e,s}_\FF )\otimes_\FF k\ar{r}{\simeq}& \Ext^i_{k[\Vect_\FF]}(t^*\Gamma^{d,s}, t^*S^{e,s})
\end{tikzcd}
\]
in which the horizontal morphisms are the base change isomorphisms and the vertical morphism on the left is induced by the forgetful functor $\Pp_\FF\to \FF[\Vect_\FF]\Md$ and restriction along $^{(-n)}-$. The vertical map on the left is an isomorphism if $n\gg 0$ by \cite[Thm 3.10]{FFSS} if $d=e$ and by \cite[Lm 3.11]{FFSS} if $d\ne e$. Thus $\Phi_\FF$ is an isomorphism when $F$ is a standard projective and $G$ is a standard injective.

For arbitrary $F$ and $G$, we consider a projective resolution of $F$ and an injective resolution of $G$, and a standard spectral sequence argument shows that $\Phi_\FF$ is an isomorphism.
\end{proof}

Theorem \ref{thm-strong-comparison} can be dualized. 
If $E$ and $F$ are two strict polynomial functors of bounded degrees, with $E$ contravariant, there is a \emph{strong comparison map for $\Tor$}  (still denoted by $\Phi_\FF$) defined for $n$ big enough as the composition:
\begin{align*}
\Phi_\FF: \Tor_i^{k[\Vect_\FF]}(t^*E,t^*F)\simeq \Tor_i^{k[\Vect_\FF]}(t^*E^{(n)},t^*F^{(n)})\xrightarrow[]{} \begin{array}[t]{c}
\Tor_i^{\Pp_k}(E^{(n)},F^{(n)})\\
= \Tor_i^\gen(E,F)
\end{array}
\label{eqn-forget-comparison-Tor}
\end{align*}
induced by restriction along the Frobenius twist $^{(-n)}-:\Vect_\FF\to \Vect_\FF$ and along $t^*$.
The commutativity of diagrams \eqref{dgm-nat-duality1} and \eqref{dgm-nat-duality2} involving the $\Ext$-$\Tor$ duality isomorphisms yields a commutative square:
\begin{equation}\label{cd-phiphi}
\begin{tikzcd}
\Hom_k(\Tor_i^\gen(E,F),k)\ar{rr}{\simeq}\ar{d}{\Hom_k(\Phi_\FF,k)}&& \Ext^i_\gen(F,D_kE)\ar{d}{\Phi_\FF}\\
\Hom_k(\Tor_i^{k[\Vect_\FF]}(t^*E,t^*F),k)\ar{rr}{\simeq}&& \Ext^i_{k[\Vect_\FF]}(t^*F,t^*(D_kE))
\end{tikzcd}\;.
\end{equation}
Thus, theorem \ref{thm-strong-comparison} has the following consequence.
\begin{cor}\label{cor-strong-comparison-Tor}
If $E$ and $F$ be two strict polynomial functors of degree $d$, with $E$ contravariant, and if $\FF$ a finite field with $q>d$ elements, then the strong comparison map $\Phi_\FF$ is an isomorphism of $\Tor$ groups in all degrees $i$.
\end{cor}

\subsection{The generalized comparison theorem over $\Vect_\Fq$}\label{subsec-generalized-comparison-Fq} 
Theorem \ref{thm-strong-comparison} and corollary \ref{cor-strong-comparison-Tor}
require that $\FF$ is a big finite field. This hypothesis is necessary: it is not hard to see that $\Hom_{\Pp_k}(S^1,S^q)$ has dimension zero, while $\Hom_{k[\Vect_\Fq]}(t^*S^1,t^*S^q)$ has dimension one.
The purpose of this section is to show that this assumption on the size of the finite field $\FF$ can nonetheless be removed, provided the strong comparison map is replaced by a slightly more complicated map, which takes into account the size of $\FF$ in its definition. 

Our approach, in particular lemma \ref{lm-chgt-base-finitefield} and theorem \ref{thm-interm-strong}, is inspired by the proof of \cite[Thm 6.1]{FFSS}. The key idea is that when the field $\FF$ is too small, the extensions $\Ext^*_{k[\Vect_\FF]}(F,G)$ can be computed in terms of extensions in $k[\Vect_L]\Md$ where $L$ is a bigger finite field, provided $F$ and $G$ are restrictions to $\Vect_\FF$ of functors over $\Vect_L$. This is proved in lemma \ref{lm-chgt-base-finitefield}, and, combined with theorem \ref{thm-strong-comparison}, this leads to the generalized comparison theorem over finite fields, for which we give two formulations, namely theorems \ref{thm-interm-strong} and \ref{thm-verystrong}.

\subsubsection{Skew diagonal maps and skew sum maps}\label{subsubsec-skew}
We introduce here some notations which will be used throughout the article.
Assume that $L$ is a perfect field. Then for all integers $a\ge 0$ and $s\ge 1$ and for all $L$-vector spaces $v$, we set 
\begin{align*}^{(a|s)}v:={}^{(0)}v\oplus {}^{(a)}v\oplus\cdots \oplus {}^{(\,(s-1)a\,)}v\;.
\end{align*} 
Assume now that $L$ contains a finite field $\Fq$, with $q=p^r$ elements, and let $\tau:\Fq\Md\to L\Md$ be the associated extension of scalars: $\tau(u)=L\otimes_\Fq u$. Then for all $\Fq$-vector spaces $u$, there is a canonical isomorphism of $L$-vector spaces:
\[
\begin{array}[t]{cccc}
\iso: &\tau(u) & \xrightarrow[]{\simeq} & {}^{(nr)}\tau(u)\\
&\lambda\otimes x & \mapsto & \lambda^{p^{-nr}}\otimes x
\end{array}\;.
\]
The \emph{skew diagonal map} is the $L$-linear map, natural with respect to $u$:
\[
\begin{array}[t]{cccc}
\diag:& \tau(u) & \to & {}^{(ar|s)}\tau(u)\\
&  \lambda \otimes x & \mapsto & (\lambda\otimes x, \lambda^{p^{-ar}}\otimes x,\dots,\lambda^{p^{-ar(s-1)}}\otimes x)
\end{array}\;.
\]
\begin{rk}
If $a=0$ then $^{(a|s)}\tau(u) = \tau(u)^{\oplus s}$ and in that case, the skew diagonal map equals the usual diagonal map: $\lambda\otimes x \mapsto (\lambda\otimes x,\cdots,\lambda\otimes x)$. In general, the skew-diagonal map can be written as the composition of the diagonal map $\tau(u)\to \tau(u)^{\oplus s}$ with the isomorphisms $\iso:\tau(u)\simeq {}^{(nar)}\tau(u)$, $0\le n <s$.
\end{rk}
Similarly, the \emph{skew sum map} is the $L$-linear map, natural with respect to $u$:
\[
\begin{array}[t]{cccc}
\summ:& {}^{(ar|s)}\tau(u) & \to & \tau(u)\\
& (\lambda_0\otimes x_0, \dots,\lambda_{s-1}\otimes x_{s-1})  & \mapsto & \displaystyle \sum_{0\le n<s} \lambda_n^{p^{nar}}\otimes x_n
\end{array}\;.
\]
\subsubsection{A base change isomorphism} We keep our perfect field $L$. We extend notation \ref{nota-twist-ord} for precomposition by Frobenius twists to the case of multiple twists. Namely, for all (covariant or contravariant) functors $F:\Vect_L\to k\Md$ and for all integers $a\ge 0$ and $s\ge 1$, we denote by $F^{(a|s)}:\Vect_L\to k\Md$ the functor defined by
\begin{align}F^{(a|s)}(v):= F\left({}^{(0)}v\oplus {}^{(a)}v\oplus\cdots \oplus {}^{(\,(s-1)a\,)}v\right)\;.\label{eqn-nota-precomp-multiple-twists-discrete}
\end{align}

Assume now that $q=p^r$ and that $\Fq\subset L$ is an extension of fields of degree $s^2$, and let $\tau:\Vect_\Fq\to \Vect_L$ denote the extension of scalars. We consider the composition, in which the second map is induced by $F(\diag)$ and $G(\summ)$ as in section \ref{subsubsec-skew}:
\begin{equation}
\Ext^*_{k[\Vect_L]}(F^{(r|s)},G^{(rs|s)})\xrightarrow[]{\res^\tau} \Ext^*_{k[\Vect_\Fq]}(\tau^*F^{(r|s)},\tau^*G^{(rs|s)})\to \Ext^*_{k[\Vect_\Fq]}(\tau^*F,\tau^*G)\;.\label{eqn-interm-iso}
\end{equation}
\begin{lm}\label{lm-chgt-base-finitefield}
If $\Fq\subset L$ is an extension of fields of degree $s^2$ and $q=p^r$, then for all $F$ and $G$ in $k[\Vect_L]\Md$, the map \eqref{eqn-interm-iso} is an isomorphism.
\end{lm}
\begin{proof}
We can find an intermediate field $K$ such that $\Fq\subset K\subset L$ is a sequence of extensions of fields of degree $s$. We are going to convert $\Ext$ over $k[\Vect_L]$ into $\Ext$ over $k[\Vect_\Fq]$ in two steps, by using the restrictions of scalars $\Vect_L\to \Vect_K$ and $\Vect_K\to \Vect_\Fq$ and their adjoints, and we are going to check that the two-steps isomorphism obtained coincides with the explicit map \eqref{eqn-interm-iso}.

In the first step, we express extensions over $k[\Vect_L]$ in terms of extensions over $k[\Vect_K]$. For this purpose, we consider the adjoint pair \cite[Prop 3.1]{FFSS}
$$\rho':\Vect_L\leftrightarrows \Vect_K:\tau'$$
where $\tau'$ is the extension of scalars and $\rho'$ the restriction of scalars associated to the extension $K\subset L$. Note that $\tau'=L\otimes_K-$ is considered here as the \emph{right} adjoint of $\rho'$, which is possible because the extension $K\subset L$ has finite degree. By proposition \ref{pr-iso-Ext-explicit} the adjoint pair $(\rho',\tau')$ induces an isomorphism for all $H$ in $k[\Vect_L]\Md$
\begin{align}
\Ext^*_{k[\Vect_L]}(H,{\rho'}^*{\tau'}^*G)\simeq \Ext^*_{k[\Vect_K]}({\tau'}^*H,{\tau'}^*G)\;.\label{eqn-isoproof2}
\end{align}
Moreover, there is an isomorphism of $L$-vector spaces $\phi:L\otimes_K v\simeq \bigoplus_{0\le i<s} {}^{(irs)}v$ natural with respect to the $L$-vector space $v$. This isomorphism is given by sending $\lambda\otimes x$ to $\sum_{0\le i<s}\lambda^{p^{-rsi}}x$. Therefore we have an isomorphism in $k[\Vect_L]\Md$:
\begin{align}
G(\phi^{-1}): G^{(rs|s)}\simeq {\rho'}^*{\tau'}^*G\;.
\label{eqn-isoproof1}
\end{align}
By combining the isomorphisms \eqref{eqn-isoproof2} and \eqref{eqn-isoproof1} we obtain an isomorphism
\begin{align}
\Ext^*_{k[\Vect_L]}(H,G^{(rs|s)})\xrightarrow[]{\simeq}
\Ext^*_{k[\Vect_K]}({\tau'}^*H,{\tau'}^*G)\;.
\label{eqn-isoproof12}
\end{align}
To finish this first step, we give a more explicit expression of the isomorphism \eqref{eqn-isoproof12}. Recall from proposition \ref{pr-iso-Ext-explicit} that the isomorphism \eqref{eqn-isoproof2} is induced by restriction along $\tau'$ and by the map $({\tau'}^*G)(\epsilon)$ where $\epsilon$ is the counit of the adjunction $\rho'\dashv\tau'$. By \cite[Prop 3.1]{FFSS}, this counit of adjunction $\epsilon_u: L\otimes_K u\to u$ is given by $\epsilon_u(\lambda\otimes x)= T(\lambda)x$, where $T(\lambda)=\sum_{0\le i<s}\lambda^{p^{rsi}}$ is the trace of $\lambda$. Thus for all $K$-vector spaces $u$ we have a commutative square of $L$-vector spaces, in which the upper horizontal arrow is induced by the canonical isomorphism $\iso$ from section \ref{subsubsec-skew}:
\[
\begin{tikzcd}
\bigoplus_{0\le i<s}{}^{(-rsi)}\tau'(u)\ar{r}{\simeq} &
\bigoplus_{0\le i<s}\tau'(u)\ar{d}{\summ}\\
\tau'(\rho'(\tau'(u)))\ar{u}{\phi_{\tau'(u)}}\ar{r}{\tau'(\epsilon_u)}&
\tau'(u)
\end{tikzcd}\;.
\]
It follows that the isomorphism \eqref{eqn-isoproof12} equals the following composition:
\[\Ext^*_{k[\Vect_L]}(H,G^{(rs|s)})\xrightarrow[]{\res^{\tau'}} \Ext^*_{k[\Vect_K]}(\tau'{}^*H,\tau'{}^*(G^{(rs|s)}))\to
\Ext^*_{k[\Vect_K]}(\tau'{}^*H,\tau'{}^*G)\]
where the last map is induced by the morphism ${\tau'}^*(G^{(rs|s)})\xrightarrow{G(\summ)}{\tau'}^*G$.

In the second step, we express extensions over $k[\Vect_K]$ in terms of extensions over $k[\Vect_\Fq]$. For this purpose, we consider a pair of adjoints, in which $\tau''$ and $\rho''$ are the extension of scalars and the restriction of scalars associated to the extension $\Fq\subset K$, and $\tau''$ is this time seen as a \textit{left} adjoint:
$$\tau'':\Vect_\Fq\leftrightarrows \Vect_K:\rho''\;.$$
Since $\rho''$ is right adjoint to $\tau''$, proposition \ref{pr-iso-Ext-explicit} yields an isomorphism:
\begin{align}
\Ext^*_{k[\Vect_K]}({\rho''}^* {\tau}^*F,E)\simeq \Ext^*_{k[\Vect_\Fq]}(\tau^*F,{\tau''}^*E)\;.
\label{eqn-isoproof4}
\end{align}
Moreover, the isomorphism of $L$-vector spaces $\psi:L\otimes_\Fq v\simeq \bigoplus_{0\le i<s} {}^{(ri)}(L\otimes_K v)$ given by $\psi(\lambda \otimes x)=\sum_{0\le i<s}\lambda^{p^{-ri}}x$ is natural with respect to the $K$-vector space $v$, hence it induces an isomorphism in $k[\Vect_K]\Md$:
\begin{align}
F(\psi): {\rho''}^* {\tau}^*F \simeq {\tau'}^*(F^{(r|s)}) \;.
\label{eqn-isoproof3}
\end{align}
Combining the isomorphisms \eqref{eqn-isoproof4} and \eqref{eqn-isoproof3}, we obtain an isomorphism:
\begin{align}
\Ext^*_{k[\Vect_K]}({\tau'}^*(F^{(r|s)}),E)\xrightarrow[]{\simeq}\Ext^*_{k[\Vect_\Fq]}(\tau^*F,{\tau''}^*E)\;.
\label{eqn-isoproof34}
\end{align}
Following the same reasoning as in the first step, one sees that 
the isomorphism \eqref{eqn-isoproof34} is equal to the composition
\begin{align*}
\Ext^*_{k[\Vect_K]}({\tau'}^*(F^{(r|s)}),E)\xrightarrow[]{\res^{\tau''}} \Ext^*_{k[\Vect_\Fq]}({\tau}^*(F^{(r|s)}),\tau''{}^*E)\to \Ext^*_{k[\Vect_\Fq]}(\tau^*F,{\tau''}^*E)
\end{align*}
with last map induced by the morphism $\tau^*F\xrightarrow[]{F(\diag)}{\tau''}^*{\tau'}^*(F^{(r|s)})$.

Now, the graded morphism \eqref{eqn-interm-iso} is the composition of the maps \eqref{eqn-isoproof12} (with $H=F^{(r|s)}$) and  \eqref{eqn-isoproof34} (with $E={\tau'}^*G$), hence it is an isomorphism.
\end{proof}

\subsubsection{The generalized comparison theorem - first form}
We extend notation \ref{nota-twist} for precomposition by Frobenius twists to the case of multiple twists. Namely, for all (covariant or contravariant) strict polynomial functors $F$ over $k$, and for all integers $a\ge 0$ and $s\ge 1$, we set 
\begin{align*}
I^{(a|s)}:=I^{(0)}\oplus I^{(r)}\oplus \cdots\oplus I^{(\,(s-1)a\,)}.
\end{align*}
and we let $F^{(a|s)}=F\circ I^{(a|s)}$. If $\stdeg F=d$ then $\stdeg(F^{(a|s)})=p^{(s-1)a}d$.
%

Fix a positive integer $s$. The strong comparison map of definition \ref{defi-strong-comp-map} and the morphisms $F(\diag)$ and $G(\summ)$ (where $\diag$ and $\summ$ are the skew maps from section \ref{subsubsec-skew}) yield a composite map:
\begin{align}
\Ext^i_{\gen}(F^{(r|s)},G^{(rs|s)})\xrightarrow[]{\Phi_\Fq} 
\Ext^i_{k[\Vect_\Fq]}(t^*F^{(r|s)},t^*G^{(rs|s)})\to \Ext^i_{k[\Vect_\Fq]}(t^*F,t^*G)\;. 
\label{eq-comparison-q-petit}
\end{align}
The next result extends theorem \ref{thm-strong-comparison} to the case of an arbitrary finite field $\Fq$, see remark \ref{rk-extend-strong}.

\begin{thm}\label{thm-interm-strong}
Assume that $k$ contains a finite field $\Fq$ of cardinal $q=p^r$, and let $s$ be a positive integer. Then for all strict polynomial functors $F$ and $G$ of degrees less than $q^s$, the map \eqref{eq-comparison-q-petit} is an isomorphism in all degrees $i$.
\end{thm}
\begin{proof}
We first claim that it suffices to prove the result when $k$ contains a subfield $L$ with $q^{s^2}$ elements. Indeed, let $k\to K$ be a finite extension of fields and let $\tau:k\Md\to K\Md$ be the extension of scalars. By \cite[Section 2]{SFB} there is an exact $k$-linear base change functor
\[-_K:\Pp_k \to \Pp_K\]
such that for all strict polynomial functors $F'$ over $k$ there are canonical isomorphisms of functors $\tau^*F'_K\simeq \tau\circ F'$. Moreover this base change functor induces an isomorphism on the level of $\Ext$. (See \cite[cor 2.7]{SFB} for the case of functors with finite-dimensional values. The proof extends to arbitrary functors when $k\to K$ is a finite extension of fields.)
There is also an exact base change functor:
\[K\otimes_k-:k[\Vect_\Fq]\Md\to K[\Vect_\Fq]\Md\]
which sends $F$ to the functor $K\otimes_k F$ such that $(K\otimes_k F)(c)=K\otimes_k F(c)$. Since $k\to K$ is a finite extension of fields, this functor also induces an isomorphism on the level of $\Ext$.
Let us denote by $\Xi_k$ the map \eqref{eq-comparison-q-petit} and by $\Xi_K$ its counterpart when $k$ is replaced by the bigger field $K$.
There is a commutative square
\[
\begin{tikzcd}[column sep = large]
\Ext^*_\gen(F^{(r|s)}_K,G^{(rs|s)}_K)\ar{r}{\Xi_K}&\Ext^*_{K[\Vect_\Fq]}(t^*\tau^*F_K,t^*\tau^*G_K)\\
K\otimes_k\Ext^*_\gen(F^{(r|s)},G^{(rs|s)})\ar{r}{K\otimes\Xi_k}\ar{u}{\simeq}&K\otimes_k \Ext^*_{K[\Vect_\Fq]}(t^*F,t^*G)\ar{u}{\simeq}
\end{tikzcd}
\]
in which the vertical isomorphisms are induced by the base change functors $-_K$ and $K\otimes_k-$ (and by the isomorphisms $H_K\circ \tau\circ t\simeq \tau\circ H\circ t$, for $H=F$ or $G$). Therefore $\Xi_k$ is an isomorphism if and only if $\Xi_K$ is an isomorphism. Thus, up to replacing $k$ by a suitable finite extension $K$, we may assume that our field $k$ contains a subfield $L$ with $q^{s^2}$ elements.

We denote by $t':\Vect_L\to \Vect_k$ the extension of scalars associated to the extension of fields $L\subset k$. Then $\Xi_k$ is an isomorphism because we can rewrite it as the composition of three isomorphisms:
\[
\begin{tikzcd}[column sep = large]
\Ext^*_\gen(F^{(r|s)},G^{(rs|s)})\ar{r}{\Xi_k}\ar{d}{\simeq}[swap]{\Phi_L}& \Ext^*_{k[\Vect_\Fq]}(t^*F,t^*G)\\
\Ext^*_{k[\Vect_L]}({t'}^*(F^{(r|s)}),{t'}^*(G^{(rs|s)}))\ar{r}{\simeq}&\Ext^*_{k[\Vect_L]}(({t'}^*F)^{(r|s)},({t'}^*G)^{(rs|s)})\ar{u}{\simeq}
\end{tikzcd}\;.
\]
To be more specific, our assumptions on $s$ imply that the degrees of $F^{(r|s)}$ and $G^{(rs|s)}$ are less than the cardinal of $L$, hence $\Phi_L$ is an isomorphism by theorem \ref{thm-interm-strong}. To define the lower horizontal map, we first observe that for all integers $i$ there is a canonical isomorphism $^{(ir)}t'(v)\simeq t'({}^{(ir)}v)$ which sends an element $\lambda\otimes x \in {}^{(ir)}(k\otimes_L v)$ to the element $\lambda^{p^{ir}}\otimes x\in k\otimes_L {}^{(ir)}v$. These canonical isomorphisms induce isomorphisms of functors ${t'}^*(F^{(r|s)})\simeq ({t'}^*F)^{(r|s)}$ and ${t'}^*(G^{(rs|s)})\simeq ({t'}^*G)^{(rs|s)}$, and the lower horizontal map is induced by these isomorphisms. Finally, the vertical map on the right hand side is the isomorphism provided by lemma \ref{lm-chgt-base-finitefield}.
\end{proof}

\begin{rk}\label{rk-extend-strong}
If $s=1$, then $F^{(r|s)}=F$, $G^{(rs|s)}=G$ and the maps $F(\diag)$ and $G(\summ)$ are equal to the identity. Thus, for $s=1$ the map \eqref{eq-comparison-q-petit} is equal to $\Phi_\FF$, and the statement of theorem \ref{thm-interm-strong} is exactly the statement of theorem \ref{thm-strong-comparison}.
\end{rk}

\subsubsection{The generalized comparison theorem - second form}
Functors precomposed by multiple twists, i.e. functors of the form $F^{(a|s)}$ are quite complicated. We now prove a reformulation of the generalized comparison theorem, in which only one of the two functors appearing in the $\Ext$ is precomposed by multiple twists. 
This second form of the theorem will be better adapted to generalization to more general additive categories $\A$.
So we now consider the following composite map, where the second map is induced by the morphisms $F(\iso)$ and $G(\summ)$, where $\iso$ and $\summ$ are defined in section \ref{subsubsec-skew}:
\begin{align}
\Ext^i_{\gen}(F^{(rs-r)},G^{(r|s^2)})\xrightarrow[]{\Phi_\Fq} 
\Ext^i_{k[\Vect_\Fq]}(t^*F^{(rs-r)},t^*G^{(r|s^2)})\to \Ext^i_{k[\Vect_\Fq]}(t^*F,t^*G)\;. 
\label{eqn-verystrong}
\end{align}

\begin{thm}\label{thm-verystrong}
Let $k$ be a perfect field containing a finite subfield with $q=p^r$ elements, and let $s$ be a positive integer. Assume that $F$ and $G$ are two strict polynomial functors with degrees less than $q^s$. Then the map \eqref{eqn-verystrong} is a graded isomorphism.
\end{thm}
\begin{proof}
The idea it to show that the map \eqref{eq-comparison-q-petit} of theorem \ref{thm-interm-strong} and the map \eqref{eqn-verystrong} coincide up to an isomorphism. For this purpose, we use strict polynomial multifunctors, as in \cite[Section 3]{SFB} or \cite{FFSS}. 

Let $\Pp_k(n)$ denote the category of strict polynomial multifunctors of $n$ variables. Precomposition by Frobenius twist extends to the multivariable setting, namely given a strict polynomial multifunctor $F$ and an $n$-tuple of non-negative integers $\underline{r}=(r_1,\dots,r_n)$ we denote by $F^{(\underline{r})}$ the strict polynomial multifunctor 
$$F^{(\underline{r})}(v_1,\dots,v_n)=F({}^{(r_1)}v_1,\dots,{}^{(r_n)}v_n)\;.$$
Precomposition by Frobenius twists yields a morphism on $\Ext$:
$$-\circ I^{(\underline{m})}:\Ext^i_{\Pp_k(n)}(F^{(\underline{r})}, G^{(\underline{r})})\to
\Ext^i_{\Pp_k(n)}(F^{(\underline{r}+\underline{m})}, G^{(\underline{r}+\underline{m})})\;.$$

This morphism is an isomorphism if that all the integers $r_i$ are big enough (with respect to $i$, $F$ and $G$). Indeed, by a spectral sequence argument, it suffices to check this when $F$ is a standard projective of $\Pp_k(n)$ and $G$ is a standard injective of $\Pp_k(n)$. In this case,
$F(v_1,\dots,v_n)=F_1(v_1)\otimes_k\cdots\otimes_k F_n(v_n)$ for some standard projective strict polynomial functors $F_i$, and $G(v_1,\dots,v_n)=G_1(v_1)\otimes_k\cdots\otimes_k G_n(v_n)$ for some standard injective strict polynomial functors $G_i$, hence the isomorphism follows from the $\Ext$-isomorphism for functors with one variable from  proposition-definition \ref{pdef-gen-Ext} and the K\"unneth formula \cite[(1.7.2) p. 673]{FFSS}.

The following functors are adjoint to each other on both sides:
\[ 
\begin{array}[t]{cccc}
\Sigma:&\Vect_k^{\times n}&\to &\Vect_k\\
& (v_1,\dots,v_{n})&\mapsto & \bigoplus_{1< i\le n}v_i
\end{array}
\;,\quad 
\begin{array}[t]{cccc}
\Delta:&\Vect_k&\to &\Vect_k^{\times n}\\
& v & \mapsto & (v,\dots,v)
\end{array}\;.
\]
The $\Ext$-isomorphism induced by adjoint functors, stated in proposition \ref{pr-Eadjoint}, is valid in the context of strict polynomial functors.

We are now ready to prove theorem \ref{thm-verystrong}. We observe that we may choose strict polynomial multifunctors $F'$, $G'$, $F''$ and $G''$ such that there is a commutative diagram, with $n\gg 0$:
\[
\begin{tikzcd}
\Ext^{*}_{\Pp_{k}(s^2)}(F',G')\ar{r}{\simeq}[swap]{(*)}& \Ext^*_{\Pp_k}\Big((F^{(rs-r)})^{(nr)},(G^{(r|s^2)})^{(nr)}\Big)\ar{dd}{\eqref{eqn-verystrong}}\\
\Ext^{*}_{\Pp_k(s^2)} (F'',G'')\ar{d}{\simeq}[swap]{(**)}\ar{u}[swap]{\simeq}{-\circ I^{(\underline{m})}}&\\
\Ext^*_{\Pp_k}\Big((F^{(r|s)})^{(nr)},(G^{(rs|s)})^{(nr)}\Big)\ar{r}{\eqref{eq-comparison-q-petit}} & \Ext^*_{k[\Vect_\Fq]}(t^*F,t^*G)
\end{tikzcd}
\]
To be more specific, the strict polynomial multifunctors $F'$, $G'$, $F''$ and $G''$ of the $s^2$ variables $v_{ij}$, $0\le i,j< s$, are respectively given by
\begin{align*}
&F'(\dots,v_{ij},\dots)=F\Big(\bigoplus_{0\le i,j< s} {}^{(nr+rs-r)}v_{ij}\Big),\\
&G'(\dots,v_{ij},\dots)=G\Big(\bigoplus_{0\le i,j< s} {}^{(nr+r(s-1-i)+rsj)}v_{ij}\Big),\\
&F''(\dots,v_{ij},\dots)=F\Big(\bigoplus_{0\le i,j< s} {}^{(nr+ri)}v_{ij}\Big),\\
&G''(\dots,v_{ij},\dots)=G\Big(\bigoplus_{0\le i,j< s} {}^{(nr+rsj)}v_{ij}\Big).
\end{align*}
The $s^2$-tuple $\underline{m}$ is given by $m_{ij}=rs-ri-r$ and $-\circ I^{(\underline{m})}$ is an isomorphism because $n$ is big enough. The maps $(*)$ and $(**)$ are induced by the adjoints $\Sigma$ and $\Delta$. To be more explicit, the map $(*)$ is given by setting $v_{ij}=v$ for all $i$ and $j$, and by pulling back the resulting extensions of strict polynomial functors of the variable $v$ by $F(\diag)$, where $\diag: {}^{(rs-r+nr)}v^{\oplus s^2}\to {}^{(rs-r+nr)}v$ denotes the map $\diag(v)=(v,\dots,v)$. Similarly, the map $(**)$ is given by setting $v_{ij}=v$ for all $i$ and $j$, and by pulling back and pushing out the resulting extensions of strict polynomial functors of the variable $v$ by $F(\diag')$ and $G(\summ')$, where the maps 
\begin{align*}
&\diag':\bigoplus_{0\le i<s}{}^{(ri+nr)}v\to\bigoplus_{0\le i<s}{}^{(ri+nr)}v^{\oplus s}   \\
&\summ':\bigoplus_{0\le j<s}{}^{(rsj+nr)}v^{\oplus s}\to  \bigoplus_{0\le j<s}{}^{(rsj+nr)}v
\end{align*}
restricts to identity morphisms between any two summands with the same number of Frobenius twists.

Under the hypotheses of theorem \ref{thm-verystrong} the map \eqref{eq-comparison-q-petit} is an isomorphism by theorem \ref{thm-interm-strong}, hence the map \eqref{eqn-verystrong} is an isomorphism by commutativity of the above diagram.
\end{proof}

Theorem \ref{thm-verystrong} can be dualized. Let $E$ and $F$ be two strict polynomial functors over $k$, with $E$ contravariant. The morphisms $E(\summ): t^*E\to t^*E^{(r|s^2)}$ and $F(\iso):t^*F\simeq t^*(F^{(rs-r)})$ together with the strong comparison map $\Phi_{\Fq}$ for $\Tor$
induce a graded $k$-linear map:
\begin{align}
\Tor_*^{k[\Vect_\Fq]}(t^*E,t^*F)\to \Tor_*^{k[\Vect_\Fq]}(t^*E^{(r|s^2)},t^*F^{(rs-r)}) \xrightarrow[]{\Phi_\Fq} \Tor^\gen_*(E^{(r|s^2)},F^{(rs-r)})\;.\label{eqn-verystrong-Tor}
\end{align}
The next corollary follows from theorem \ref{thm-verystrong} together with together with the commutativity of diagrams \eqref{dgm-nat-duality1} and \eqref{dgm-nat-duality2} involving the $\Ext$-$\Tor$ duality isomorphisms.
\begin{cor}\label{cor-verystrong}Let $k$ be a perfect field containing a finite subfield with $q=p^r$ elements, and let $s$ be a positive integer. Assume that $E$ and $F$ are two strict polynomial functors (respectively contravariant and covariant) with degrees less than $q^s$. Then the map \eqref{eqn-verystrong-Tor} is a graded isomorphism.
\end{cor}

\section{Recollections of linear functors and Kan extensions}
In order to extend the results of the previous section to more general additive source categories $\A$, we need to recall additional facts and notations relative to linear functors and Kan extensions. 

\subsection{Categories of linear functors}
Let $\FF$ be a commutative ring, let $\A$ be a svelte additive $\FF$-linear category, and let $k$ be a commutative $\FF$-algebra. We denote by $k\otimes_\mathbb{\FF}\A\Md$ the full subcategory of $k[\A]\Md$ on the $\FF$-linear functors. 
It is an abelian subcategory of $k[\A]\Md$, stable under arbitrary limits and colimits, with enough projectives and injectives. The \emph{standard projectives} of $k\otimes_\FF\A\Md$ are the functors of the form $k\otimes_\FF\A(a,-)$. The Yoneda lemma yields isomorphisms:
\[\Hom_{k\otimes_{\FF}\A}(k\otimes_\FF\A(a,-),\pi)\simeq \pi(a)\]
natural with respect to $a$ and $\pi$,  
and every $\FF$-linear functor admits a resolution by direct sums of standard projectives. The notation $k\otimes_\FF\A\Md$ for the category of $\FF$-linear functors recalls the form of the standard projectives.

Similarly, we may consider the full subcategory $\Mdd k\otimes_\FF\A = k\otimes_\FF\A^\op\Md$ on the $\FF$-linear functors of  $\Mdd k[\A]=k[\A^\op]\Md$.
By restriction, the tensor product over $\A$ yields a functor
\[-\otimes_{k[\A]}-:(\Mdd k\otimes_\FF\A)\times (k\otimes_\FF\A\Md)\to k\Md\]
which can be derived using projective resolutions in the categories of additive functors. The resulting derived functors are denoted by $\Tor^{k\otimes_\FF\A}_*(\pi,\rho)$.

\begin{rk}
Since the inclusion $k\otimes_{\FF}\A\Md\hookrightarrow k[\A]\Md$ is exact, it induces graded $k$-linear morphisms:
\begin{align*}
&\Ext^*_{k\otimes_\FF\A}(\pi',\rho)\to \Ext^*_{k[\A]}(\pi',\rho)\;,
& \Tor_*^{k[\A]}(\pi,\rho)\to \Tor_*^{k\otimes_{\FF}\A}(\pi,\rho)\;.
\end{align*}
Although these graded maps are isomorphisms in degree zero, they are usually very far from being isomorphisms in positive degrees, see \cite{DT-add}.
\end{rk}

\subsection{Left Kan extensions and generalized compositions.}\label{subsec-gen-comp}
We now assume that $\FF$ is a field. We denote by $\overline{F}:\FF\Md\to k\Md$ the left Kan extension to all vector spaces of a functor $F:\Vect_\FF\to k\Md$. That is, $\overline{F}(v)$ is the colimit of the vector spaces $F(u)$ taken over the filtered poset of finite-dimensional subspaces $u\subset v$ ordered by inclusion. 
We use left Kan extensions to extend the composition $\phi^*F=F\circ\phi$ to the cases when $\phi$ has infinite dimensional values.
\begin{defi}\label{defi-comp}
For all functors $F$ in $k[\Vect_\FF]\Md$ and for all functors $\phi:\C\to \FF\Md$ we define $\phi^*F$ as the composition: 
$\phi^*F:=\overline{F}\circ \phi$.
\end{defi}

\begin{rk}\label{rk-ambiguous}
If $\psi:k\Md\to k\Md$ is an additive functor, we usually denote by the same letter its restriction  $\psi:\Vect_k\to k\Md$. As a consequence, the formula $\phi^*(\psi^* F)$ is ambiguous: it may be interpreted as $\overline{F}\circ \psi\circ \phi$ or as $\overline{\overline{F}\circ \psi}\circ \phi$. For this reason, we shall cautiously avoid iterating the notation of definition \ref{defi-comp} and we turn back to notations with compositions whenever there is a risk of ambiguity.
\end{rk}

If $F$ and $G$ are strict polynomial functors, $\overline{F}\circ G$ can be viewed as a strict polynomial functor in a canonical way. Indeed, $G$ is the filtered colimit of its subfunctors with finite-dimensional values (this local finiteness comes from the fact that the category $\Pp_{d,k}$ of $d$-homogenous functors is equivalent to modules over a Schur algebra, which is a finite-dimensional algebra). Thus we can make the following definition.
%
\begin{defi}\label{defi-comp-strict}
If $F$ and $G$ are two strict polynomial functors, the strict polynomial functor $\overline{F}\circ G$ is defined as the colimit of the strict polynomial functors $F\circ G'$ taken over the poset of subfunctors $G'\subset G$ having finite dimensional values. If $\deg F=d$ and $\deg G=e$ then $\deg (\overline{F}\circ G)=de$. 
\end{defi}

\subsection{Homological properties of left Kan extensions}
Let $\A$ be a svelte additive category and let $\aleph$ be a regular cardinal. There is a smallest svelte additive category $\A^\aleph$ which contains $\A$ as a full subcategory and which has all direct sums   
of cardinality less than $\aleph$. 

For all functors $F:\A\to k\Md$, we denote by $F^\aleph:\A^\aleph\to k\Md$ the left Kan extension of $F$ along the inclusion $\iota^\aleph:\A\hookrightarrow \A^\aleph$. If $F$ is additive then $F^\aleph$ is additive and preserves direct sums of cardinality less than $\aleph$. 
One can compute $F^\aleph(a)$ as the filtered colimit of $F(b)$ where $b$ runs over the poset of direct summands of $a$ which belong to $\A$. We refer the reader to \cite[section 3]{DT-add} for additional details. We will need the following three examples.
\begin{ex}
\begin{enumerate}[(1)]
\item If $\A=\Vect_\FF$, then $\A^\aleph = \Vect_\FF^\aleph$ is the category of $\FF$-vector spaces of dimension less than $\aleph$. The functor $F^\aleph$ is the restriction of the functor $\overline{F}:\FF\Md\to \FF\Md$ defined in section \ref{subsec-gen-comp} to the category $\Vect_\FF^\aleph$.
\item For all objects $a$ of $\A$, the left Kan extension of $\A(a,-):\A\to \FF\Md$ is the functor $\A^\aleph(a,-):\A^\aleph\to \FF\Md$.
\item The left Kan extension of a generalized composition $\pi^*F=\overline{F}\circ \pi$ of a functor $F:\Vect_\FF\to k\Md$ with a functor $\pi:\A\to \FF\Md$, is the composition $\overline{F}\circ \pi^\aleph:\A^\aleph\to k\Md$.
\end{enumerate}
\end{ex}
An advantage of working with $\A^\aleph$ rather than $\A$  is the existence of certain adjoints, which will play a key role in the computations of proposition \ref{pr-iso-rep-Theta}.
\begin{pr}\label{pr-Eadjoint}
Assume that $\A(a,b)$ belongs to $\Vect_\FF^\aleph$ for all $a$ and $b$. Then for all objects $a$ in $\A$ the functor $\A^\aleph(a,-):\A^\aleph\to \Vect_\FF^\aleph$ has a left adjoint $-\otimes a:\Vect_\FF^\aleph\to \A^\aleph$. 
\end{pr}

Being computed with filtered colimits of $k$-modules, the left Kan extension induces an \textit{exact} functor $-^\aleph : k[\A]\Md\to k[\A^\aleph]\Md$, which is left adjoint to the exact functor given by restriction along $\iota^\aleph$. Standard homological algebras gives:
\begin{pr}\label{pr-Kexact}Restriction along $\iota^\aleph:\A\hookrightarrow \A^\aleph$ yields isomorphisms
\begin{align*}
& \Ext^*_{k[\A]}(F,\iota^{\aleph\, *} G)\simeq \Ext^*_{k[\A^\aleph]}(F^\aleph,G)\;,
& \Tor_*^{k[\A]}(\iota^{\aleph\, *} E,F)\simeq \Tor_*^{k[\A^\aleph]}(E,F^\aleph)\;.
\end{align*}
\end{pr}

\section{An auxiliary comparison map}\label{sec-gen-prelim-comparison}
Throughout this section $k$ is a commutative ring, $\FF$ is a field, $\A$ is a small additive category, and we consider functors 
\[F,G:\Vect_\FF\to k\Md \;,\quad  \pi:\A^{\op}\to \FF\Md \;,\quad \rho:\A\to \FF\Md\;,\]
with $\rho$ and $\pi$ additive.  We define a `dual' vector space $D_{\pi,\rho}(v)$ of an $\FF$-vector space $v$ by the formula
\begin{align*}
D_{\pi,\rho}(v):= \Hom_\FF(v, \pi\otimes_{\FF[\A]} \rho)\;.
\end{align*}
The purpose of this section is to construct a certain comparison map:
\[\Theta_\FF: \Tor_*^{k[\A]}(\pi^*F,\rho^*G) \to \Tor_*^{k[\Vect_\FF]}(D_{\pi,\rho}^*F,G)\;.\]
and to establish its main properties. This comparison map will be a key ingredient in the proof of our generalized comparison theorem, but it may be interesting in its own right and useful in other contexts. In particular, we prove in theorem \ref{thm-Theta} that this comparison map is an isomorphism under a certain homological condition on $\pi$ and $\rho$, a result which may also be useful in the `cross characteristic context' (i.e. when $\FF\otimes_\mathbb{Z}k=0$) that is not the focus of this article.

\begin{rk} We will use the results of this section in the proof of proposition \ref{prop-base-case}. This requires that $\pi$, $\rho$ or $D_{\pi,\rho}$ are allowed to have infinite dimensional values. Thus, a notation like $\pi^*F$ is defined using the left Kan extension of $F$ if needed: $\pi^*F=\overline{F}\circ\pi$, as explained in section \ref{subsec-gen-comp}.
\end{rk}

\subsection{Construction and first properties of $\Theta_\FF$.}\label{subsec-defTheta}
It follows from the isomorphism \eqref{eqn-Formule-Cartan} of section \ref{subsec-recoll-ord} that we have an adjoint pair of functors
$-\otimes_{\FF[\A]}\rho:\Mdd \FF[\A]\leftrightarrows \FF\Md: \Hom_\FF(\rho,-).$
We denote by $\theta_\FF$ the unit of adjunction: 
\begin{equation}
\theta_\FF:\pi\to \Hom_\FF(\rho,\pi\otimes_{\FF[A]}\rho)=D_{\pi,\rho}\circ\rho\;.
\label{eqn-def-theta}
\end{equation}
Thus $\theta_\FF$ is a morphism in $\Mdd \FF[\A]$, whose component at $a$ is the $\FF$-linear map 
\[(\theta_\FF)_a: \pi(a)\to \Hom_\FF(\rho(a), \pi\otimes_{\FF[A]}\rho)\]
which sends $x\in\pi(a)$ to the $\FF$-linear map $y\mapsto \llbracket x\otimes y\rrbracket$ where the brackets denote the class of $x\otimes y$ in $\rho\otimes_{\FF[A]}\pi$.
Next, we choose a regular cardinal $\aleph$ such that the images of $\rho$ and $\pi$ are contained in the category $\Vect_\FF^{\aleph}$ of vector spaces of dimension less than $\aleph$, and we let $\iota^\aleph:\Vect_\FF^{\aleph}\hookrightarrow \FF\Md$ be the inclusion of categories. We define $\Theta_\FF$ as the unique graded $k$-linear map fitting into the commutative square (note that $\res_{\iota^\aleph}$ is an isomorphism by proposition \ref{pr-Kexact}):
\begin{equation}
\begin{tikzcd}[column sep=large]
\Tor_*^{k[\A]}(\overline{F}\circ\pi,\overline{G}\circ\rho)\ar[dashed]{r}{\Theta_\FF}
\ar{d}[swap]{\Tor_*^{k[\A]}(\overline{F}(\theta_\FF),\overline{G}\circ\rho)}
 & \Tor_*^{k[\Vect_\FF]}(\overline{F}\circ D_{\pi,\rho},G)\ar{d}{\res_{\iota^\aleph}}[swap]{\simeq}\\
\Tor_*^{k[\A]}(\overline{F}\circ D_{\pi,\rho}\circ\rho,\overline{G}\circ\rho)\ar{r}{\res_\rho} & \Tor_*^{k[\Vect_\FF^\aleph]}(\overline{F}\circ D_{\pi,\rho},\overline{G})
\end{tikzcd}.\label{eqn-def-Theta}
\end{equation}

\begin{lm}\label{lm-independant-card}
The map $\Theta_\FF$ does not depend on the choice of $\aleph$. 
\end{lm}
\begin{proof}
This is a consequence of the fact that for a cardinal $\beth$ greater than $\aleph$ we have a commutative diagram (where $\iota^{\aleph,\beth}$ is the inclusion of $\Vect_\FF^\aleph$ into $\Vect_\FF^\beth$):
\begin{equation*}
\begin{tikzcd}[column sep=large]
&  
 \Tor_*^{k[\Vect_\FF]}(\overline{F}\circ D_{\pi,\rho}F,G)\ar{d}[swap]{\simeq}{\res_{\iota^\aleph}}\ar[bend left, shift left=10ex]{dd}[swap, near start]{\simeq}[near start]{\res_{\iota^\beth}}\\
\Tor_*^{k[\A]}(\overline{F}\circ D_{\pi,\rho}\circ \pi,\overline{G}\circ \rho)\ar{r}{\res_\rho}\ar{rd}{\res_\rho} & \Tor_*^{k[\Vect_\FF^\aleph]}(\overline{F}\circ D_{\pi,\rho},\overline{G})\ar{d}{\res_{\iota^{\aleph,\beth}}}\\
& \Tor_*^{k[\Vect_\FF^\beth]}(\overline{F}\circ D_{\pi,\rho},\overline{G})
\end{tikzcd}.
\end{equation*}
\end{proof}

\begin{lm}\label{lm-nat-Theta}
The map $\Theta_\FF$ is natural with respect to $F$, $G$, $\pi$ and $\rho$.
\end{lm}
\begin{proof}
It is equivalent to prove the naturality of $\res_{\iota^\aleph}\circ\Theta_k$ with respect to $F$, $G$, $\pi$ and $\rho$.
Naturality with respect to $F$, $G$ and $\pi$ is a straightforward verification. We check naturality with respect to $\rho$, which is less straightforward since $\theta_\FF$ is not natural with respect to $\rho$. Let $f:\rho\to \rho'$ be a natural transformation, and let $D_f:D:=D_{\pi,\rho}\to D':=D_{\pi,\rho'}$ be the natural transformation induced by $f$. We consider the following diagram of graded $k$-modules, in which the composition operator for functors is omitted, e.g. `$\overline{F}\pi$' means $\overline{F}\circ \pi$, and the arrows are labelled by the natural transformations which induce them.
\[
\begin{tikzcd}
\Tor_*^{k[\A]}(\overline{F}\pi,\overline{G}\rho)\ar{rr}{\overline{G}f}\ar{d}{\overline{F}\theta_\FF}\ar{dr}{\overline{F}\theta_\FF}&&\Tor_*^{k[\A]}(\overline{F}\pi,G\rho')\ar{dd}{\overline{F}\theta_\FF}\\
\Tor_*^{k[\A]}(\overline{F}D\rho,\overline{G}\rho)\ar{dd}{\res_\rho}\ar{dr}[swap]{\overline{F}D_f\rho}&\Tor_*^{k[\A]}(\overline{F}D'\rho',\overline{G}\rho)\ar{dr}{\overline{G}f}\ar{d}{\overline{F}D'f}&\\
& \Tor_*^{k[\A]}(\overline{F}D'\rho,\overline{G}\rho) \ar{dr}[swap]{\res_\rho}& \Tor_*^{k[\A]}(\overline{F}D'\rho',\overline{G}\rho')\ar{d}{\res_{\rho'}}\\
\Tor_*^{k[\Vect_\FF^\aleph]}(\overline{F}  D,\overline{G})\ar{rr}{\overline{F}D_f}&& \Tor_*^{k[\Vect_\FF^\aleph]}(\overline{F} D',\overline{G})
\end{tikzcd}
\]
The upper right triangle and the lower left triangle of the diagram are obviously commutative. The upper left parallelogram is commutative because of the dinaturality of $\theta_\FF$, i.e. because the following square commutes:
\[
\begin{tikzcd}[column sep=large]
\pi \ar{rr}{\theta_\FF}\ar{d}{\theta_\FF} &&D'\rho'=\Hom_\FF(\rho', \pi\otimes_{\FF[\A]}\rho')\ar{d}{\Hom_\FF(f, \pi\otimes_{\FF[\A]}\rho')}\\
D\rho=\Hom_\FF(\rho, \pi\otimes_{\FF[\A]}\rho)\ar{rr}{\Hom_\FF(\rho, \pi\otimes_{\FF[\A]}f)} && D'\rho=\Hom_\FF(\rho, \pi\otimes_{\FF[\A]}\rho')
\end{tikzcd}\;.
\] 
Finally, the lower right parallelogram commutes by dinaturality of restriction maps between $\Tor$-modules. Thus the outer square is commutative, which shows that $\res_{\iota^\aleph}\circ\Theta_\FF$, hence $\Theta_\FF$, is natural with respect to $\pi$.
\end{proof}

\subsection{Base change}\label{subsec-base-change-theta}
Let $\FF\to \KK$ be a morphism of fields. We now establish a relation between the maps $\Theta_\FF$ and $\Theta_\KK$, which will be needed in the proof of the generalized comparison theorem in section \ref{subsec-proof-gencomp}.  
Let $t:\FF\Md\to \KK\Md$ denote the extension of scalars: $t(v)=\KK\otimes_\FF v$. There is a canonical isomorphism
\[t(\pi\otimes_{\FF[\A]}\rho)\simeq (t\circ\pi)\otimes_{\KK[\A]}(t\circ\rho)\]
which sends $\lambda\otimes\llbracket x\otimes y\rrbracket$ to $\llbracket (\lambda\otimes x) \otimes (1\otimes y)\rrbracket$. Together with the canonical map base change map 
$t(\Hom_\FF(v,w))\to \Hom_\KK(t(v), t(w))$, they induce a canonical $\KK$-linear map, natural with respect to the $\FF$-vector space $v$ and the additive functors $\pi$ and $\rho$:
\begin{align}t(\Hom_\FF(v,\pi\otimes_{\FF[A]}\rho)) \xrightarrow[]{}
&\Hom_\KK(t(v), (t\circ \pi)\otimes_{\KK[A]} (t\circ \rho))\;.\label{eqn-can}
\end{align}
which is an isomorphism when $v$ has finite dimension.  If we let $D_{\pi,\rho}$ and $D_{t\circ\pi,t\circ\rho}$ be the duality functors respectively defined by: 
\begin{align*}
&D_{\pi,\rho}(v)=\Hom_{\FF}(v,\pi\otimes_{\FF[A]}\rho)\\
& D_{t\circ\pi,t\circ\rho}(w)=\Hom_\KK(w,(t\circ\pi)\otimes_{\KK[A]}(t\circ\rho))
\end{align*}
then the morphism \eqref{eqn-can} can be written as a morphism of functors
\begin{equation}t\circ D_{\pi,\rho}\xrightarrow[]{\mathrm{can}} D_{t\circ\pi,t\circ\rho}\circ t
\label{eqn-can-fct}
\end{equation}
whose component at every finite-dimensional vector space $v$ is an isomorphism. 

\begin{pr}\label{prop-chgt-base}
Let $\FF\to \KK$ be a field morphism. 
For all additive functors $\pi:\A^\op\to \FF\Md$ and $\rho:\A\to \FF\Md$, and for all objects $F$ and $G$ in $k[\Vect_\KK]\Md$, we have a commutative diagram in which the lower horizontal isomorphism is induced by the isomorphism $\overline{F}(\mathrm{can})$:
\[
\begin{tikzcd}[column sep=large]
\Tor_*^{k[\A]}(\overline{F}\circ t\circ \pi,\overline{G}\circ t\circ \rho)\ar{r}{\Theta_\KK}\ar{d}{\Theta_\FF}& 
\Tor_*^{k[\Vect_\KK]}(\overline{F}\circ D_{t\circ \pi,t\circ \rho},G)\\
\Tor_*^{k[\Vect_\FF]}(\overline{F}\circ t\circ D_{\pi,\rho},G\circ t)\ar{r}[swap]{\simeq} &
\Tor_*^{k[\Vect_\FF]}(\overline{F}\circ D_{t\circ \pi,t\circ \rho}\circ t,G\circ t) \ar{u}{\res_t}
\end{tikzcd}.
\]
\end{pr}
\begin{proof}
Let us denote $D=D_{t\circ\pi,t\circ\rho}$ and $D'=D_{\pi,\rho}$ for short and let $\aleph$ be a big enough regular cardinal. 
We have a diagram of graded $k$-modules, in which the composition symbol for functors has been omitted and the arrows are labelled by the name of the morphisms which induce them. 
\[
\begin{tikzcd}
\Tor_*^{k[\A]}(\overline{F}t\pi,\overline{G}t\rho) \ar{r}{\overline{F}\theta_\KK}\ar{d}{\overline{F}t\theta_\FF}
& \Tor_*^{k[\A]}(\overline{F}Dt\rho,\overline{G}t\rho)\ar{d}{\res_\rho}\ar{r}{\res_{t\rho}}
& \Tor_*^{k[\Vect_\KK^\aleph]}(\overline{F}D,\overline{G}) \\
\Tor_*^{k[\A]}(\overline{F}tD'\rho,\overline{G}t\rho) \ar{ru}{\overline{F}\mathrm{can}\rho}\ar{d}{\res_\rho}
& \Tor_*^{k[\Vect_\FF^\aleph]}(\overline{F}Dt,\overline{G}t) \ar{ru}{\res_t}
& \Tor_*^{k[\Vect_\KK]}(\overline{F}D,\overline{G}) \ar{u}[swap]{\res_{\iota^\aleph}}{\simeq}\\
\Tor_*^{k[\Vect_\FF^\aleph]}(\overline{F}tD',\overline{G}t) \ar{ru}{\overline{F}\mathrm{can}}
& \Tor_*^{k[\Vect_\FF]}(\overline{F}tD',\overline{G}t) \ar{l}{\res_{\iota^\aleph}}[swap]{\simeq}\ar{r}{\simeq}[swap]{\overline{F}\mathrm{can}}
 & \Tor_*^{k[\Vect_\FF]}(\overline{F}Dt,\overline{G}t) \ar{u}{\res_t}\ar{ul}[swap]{\res_{\iota^\aleph}}{\simeq}\\
\end{tikzcd}
\]
One readily checks from the explicit expressions of $\theta_\KK$, $\theta_\FF$ and of the canonical morphism $\mathrm{can}:tD'\to Dt$ that $\theta_\KK=(\mathrm{can}\,\rho)\circ (t\theta_\FF)$, hence the upper left triangle of the diagram commutes. The other cells of the diagram obviously commute. The commutativity of the outer square proves proposition \ref{prop-chgt-base}.
\end{proof}

\subsection{A first isomorphism result} The next proposition will be used in the proof of the generalized comparison theorem in section \ref{subsec-proof-gencomp}.

\begin{pr}\label{pr-iso-rep-Theta}
If $\A$ is $\FF$-linear and if $\pi=\A(-,a)$ and $\rho=\A(b,-)$, then $\Theta_\FF$ is an isomorphism.
\end{pr}
\begin{proof}
By lemma \ref{lm-independant-card}, we may assume $\aleph$  as big as we want in the definition of $\Theta_\FF$, in particular we may assume that the functor $\rho^\aleph=\A^\aleph(b,-):\A^\aleph\to \Vect_\FF^\aleph$ has a left adjoint $\tau:=b\otimes_{\FF}-$ by proposition \ref{pr-Eadjoint}. We also let $\pi^\aleph=\A^\aleph(-,a):\A^\aleph\to \FF\Md$.

We first reinterpret $\theta_\FF$ in the situation of proposition \ref{pr-iso-rep-Theta}.
We have an isomorphism
\[\phi:\pi\otimes_{\FF[\A]}\rho \xrightarrow[]{\simeq}\A(b,a)\]
which sends the class of $f\otimes g\in \A(x,a)\otimes_\FF \A(b,x)$ to  $f\circ g\in \A(b,a)$. (The inverse of $\phi$ sends an element $f\in\A(a,b)$ to the class of $\id_a\otimes f\in\A(a,a)\otimes_\FF \A(b,a)$.) From the explicit expressions of $\theta_\FF$ and $\phi$, one sees that the lower left triangle of the following diagram commutes.
\begin{equation}
\begin{tikzcd}[column sep=large]
\A(x,a)\ar{rr}{\A(\epsilon_x,a)}\ar{d}[swap]{\theta_{\FF}}\ar{rrd}[swap]{\A(b,-)} && \A^\aleph(b\otimes_\FF \A(b,x),a)\ar{d}{\alpha}[swap]{\simeq}\\
\Hom_\FF(\A(b,x), \pi\otimes_{\FF[\A]}\rho)\ar{rr}[swap]{\Hom_\FF(\A(b,x),\phi)}{\simeq}&& \Hom_\FF(\A(b,x),\A(b,a))
\end{tikzcd}
\label{eqn-diagr-com}
\end{equation}
The upper right triangle of diagram \eqref{eqn-diagr-com} also commutes: here $\alpha$ is an adjunction isomorphism for the adjunction between $\tau$ and $\rho^\aleph$, and $\epsilon_x$ is the associated counit of adjunction. Diagram \eqref{eqn-diagr-com} is our new interpretation of $\theta_\FF$.

Next we prove that $\res^{\iota^\aleph}\circ\Theta_\FF$ is an isomorphism. We let $D:=D_{\pi,\rho}$ for short, and we let $\chi:D\simeq \pi^\aleph\circ \tau$ be the isomorphism whose component at $v$ is given by the composition:
\[\Hom_\FF(v,\pi\otimes_{\FF[\A]}\rho)\xrightarrow[\simeq]{\Hom_\FF(v,\phi)} \Hom_\FF(v, \A(a,b))\xrightarrow[\simeq]{\alpha^{-1}} \A^\aleph(b\otimes_\FF v,a)\;.\]
We consider the following diagram of graded $k$-modules, in which the composition operator for functors is omitted and the arrows are labelled by the natural transformations which induce them. The vertical maps $\res_{\iota^\aleph}$ are induced by restriction along the inclusion $\iota^\aleph:\A\hookrightarrow \A^\aleph$, and they are isomorphisms by proposition \ref{pr-Kexact}.
\[
\begin{tikzcd}
\Tor_*^{k[\A^\aleph]}(\overline{F}\pi^\aleph,\overline{G}\rho^\aleph)\ar{r}{\overline{F}\pi^\aleph\epsilon}&
\Tor_*^{k[\A^\aleph]}(\overline{F}\pi^\aleph \tau \rho^\aleph,\overline{G}\rho^\aleph)\ar{r}{\res_{\rho}} &
\Tor_*^{k[\Vect_\FF^\aleph]}(\overline{F}\pi^\aleph \tau,\overline{G})\ar[equal]{d}
\\
\Tor_*^{k[\A]}(\overline{F}\pi^\aleph,\overline{G}\rho^\aleph)
\ar{u}[swap]{\res_{\iota^\aleph}}{\simeq}
\ar{r}{\overline{F}\pi^\aleph\epsilon}\ar[equal]{d}&
\Tor_*^{k[\A]}(\overline{F}\pi^\aleph \tau \rho,\overline{G}\rho)
\ar{u}[swap]{\res{\iota^\aleph}}{\simeq}
\ar{r}{\res_{\rho}} &
\Tor_*^{k[\Vect_\FF^\aleph]}(\overline{F}\pi^\aleph \tau,\overline{G})\\
\Tor_*^{k[\A]}(\overline{F}\pi,\overline{G}\rho)
\ar{r}{\overline{F}\theta_\FF}
&\Tor_*^{k[\A]}(\overline{F}D\rho,\overline{G}\rho)
\ar{r}{\res_{\rho}}
\ar{u}[swap]{\overline{F}\chi \rho}{\simeq}&
\Tor_*^{k[\Vect_\FF^\aleph]}(\overline{F}D,\overline{G})\ar{u}[swap]{\overline{F}\chi}{\simeq}
\end{tikzcd}
\]
All the squares of the diagram are obviously commutative, except the lower left square which commutes by commutativity of diagram \eqref{eqn-diagr-com}. The composite in the top row is the $\Tor$-map induced by the adjunction between $\tau$ and $\rho^\aleph$, hence it is an isomorphism by proposition \ref{pr-iso-Ext-explicit}. We deduce that the composite in the bottom row, which is nothing but $\res^{\iota^\aleph}\circ\Theta_\FF$, is an isomorphism. Hence $\Theta_\FF$ is an isomorphism.
\end{proof}

\begin{cor}\label{cor-direct-sum-rep-Theta}
If $\A$ is $\FF$-linear, and if 
$$\pi=\bigoplus_{i\in I}\A(-,a_i)\;,\qquad \rho=\bigoplus_{j\in J}\A(b_j,-)$$
for some possibly infinite indexing sets $I$ and $J$, then $\Theta_\FF$ is an isomorphism.
\end{cor}
\begin{proof}
If $I$ and $J$ are finite, then $\pi\simeq \A(-,a)$ and $\rho\simeq \A(b,-)$ for $a=\bigoplus a_i$ and $b=\bigoplus b_j$, hence $\Theta_\FF$ is an isomorphism by proposition \ref{pr-iso-rep-Theta}. For arbitrary $I$ and $J$, the functors $\pi$ and $\rho$ are filtered colimits of monomorphisms of functors of the form $\A(-,a)$ and $\A(b,-)$. So the result follows from the fact that the target and the source of $\Theta_\FF$ both preserve filtered colimits of monomorphisms of functors, when viewed as functors of the variables $\pi$ and $\rho$. (Indeed $\Tor_*$, $\overline{F}\circ\pi$, $\overline{G}\circ\rho$ and $\overline{F}\circ D_{\pi,\rho}$ preserve filtered colimits of monomorphisms -- for $\overline{F}\circ\pi$, $\overline{G}\circ\rho$, this follows from the fact that $\overline{F}$ and $\overline{G}$ are left Kan extensions of $F$ and $G$, and for $\overline{F}\circ D_{\pi,\rho}$, one uses in addition the isomorphism $D_{\pi,\rho}(v)\simeq \Hom_\FF(v,\FF)\otimes_\FF (\pi\otimes_{\FF[\A]}\rho)$, which holds because we view $D_{\pi,\rho}$ as a functor from $\Vect_\FF$ to $\FF\Md$.)
\end{proof}

\subsection{Simplicial resolutions}\label{subsec-simplicial}
In order to generalize the result of corollary \ref{cor-direct-sum-rep-Theta} to more general functors $\pi$ and $\rho$, we will rely on some simplicial techniques that we explain now. We refer the reader to \cite{Weibel,GoerssJardine} for further details on simplicial objects. 

Recall that the homotopy groups of a simplicial object in an abelian category $\M$ may be defined as the homology groups of the associated normalized chain complex; if $\M=R\Md$, this coincides with the usual homotopy groups of the underying simplicial set. 
A morphism of $f:X\to Y$ of simplicial objects is called \emph{$e$-connected} if the map $\pi_j(f):\pi_jX\to \pi_j(Y)$ induced on the level of homotopy groups is an isomorphism if $0\le j<e$ and an epimorphism if $j=e$. A morphism $f:X\to Y$ is called a \emph{simplicial resolution of $Y$} if $\pi_i(f)$ is an isomorphism for all $i$. This terminology also applies to objects $Y$ of $\M$ by considering them as constant simplicial objects. If $\M$ has enough projectives, then it follows from the Dold-Kan correspondence that every simplicial object $Y$ of $\M$ admits a \emph{projective simplicial resolution}, i.e. a simplicial resolution $X\to Y$ in which $X$ is degreewise projective.

If $\KK$ is a field, $F$ is a simplicial object in $k[\Vect_\KK]\Md$, and $\mu$ is a simplicial object in the category of additive functors $\A\to \KK\Md$, we let $\mu^*F$ be the diagonal simplicial object $\overline{F}_n\circ \mu_n$. Thus $\mu^*F$ is a simplicial object in $k[\A]\Md$ natural with respect to $\mu$ and $X$. The next two lemmas are our main tools to contruct convenient simplicial resolutions. The first one is a variant of the Whitehead theorem (and ultimately relies on it):
\begin{lm}\label{lm-Whitehead-thm}
If $f:F\to G$ and  $g:\mu\to \nu$ are $e$-connected morphisms, the induced morphism $\mu^*F\to \nu^*G$ is $e$-connected. 
\end{lm}
\begin{proof}
It suffices to show that for all object $a$ of $\A$, the following two morphisms of simplicial $k$-modules are $e$-connected:
\begin{align*}
\overline{f}(\mu(a)):\overline{F}(\mu(a))\to \overline{G}(\mu(a))\;, &&
\overline{G}(g):\overline{G}(\mu(a))\to \overline{G}(\nu(a))\;.
\end{align*}
We first observe that for all $\KK$-vector spaces $v$, the simplicial map $\overline{f}:\overline{F}(v)\to \overline{G}(v)$ is $e$-connected because $f$ is $e$-connected and homotopy groups preserve filtered colimits. Hence the $e$-connectedness of $\overline{f}(\mu(a))$ follows from the spectral sequence \cite[IV, section 2.2]{GoerssJardine}, which is natural with respect to $\overline{F}$:
$$E_{m,n}^1(\overline{F})= \pi_m (\overline{F}(\mu_n(a)))\Rightarrow \pi_{m+n}(\overline{F}(\mu(a)))\;.$$

Let us prove the $e$-connectedness of $\overline{G}(g)$. Let $P\to G$ be a simplicial resolution of $G$ in $k[\Vect_\KK]\Md$ by direct sums of standard projectives. Since taking left Kan extensions is exact, $\overline{P}\to \overline{G}$ is also a simplicial resolution. Therefore we have a commutative diagram of simplicial $k$-modules
$$\begin{tikzcd}
\overline{P}(\mu(a))\ar{r}{\overline{P}(g)}\ar{d}{}& \overline{P}(\nu(a))\ar{d}{}\\
\overline{G}(\mu(a))\ar{r}{\overline{G}(g)}& \overline{G}(\nu(a))
\end{tikzcd}$$
whose vertical arrows induce isomorphisms on homotopy groups by the preceding paragraph. Thus, it suffices to check that $\overline{P}(g)$ is $e$-connected.
By using the spectral sequence natural with respect to the simplicial vector space $v$:  
$$'E_{m,n}^1(v)= \pi_n(\overline{P_m}(v))\Rightarrow \pi_{m+n}(\overline{P}(v))\;,$$
the proof reduces further to showing that if $Q$ is a standard projective of $k[\Vect_\KK]\Md$, the morphism of simplicial $k$-modules $\overline{Q}(g)$ is $e$-connected. Now the left Kan extension of a standard projective $Q$ has the form:
\[k[\Hom_\KK(\KK^n,-)]\simeq k[-^{\oplus n}]:\KK\Md\to\KK\Md\;.\]
Hence $\overline{Q}(g)\simeq k[g^{\oplus n}]$, and $g^{\oplus n}$ is $e$-connected. Thus the $e$-connectedness of $\overline{Q}(g)$ follows from the Whitehead theorem, namely if $h:X\to Y$ is an $e$-connected morphism of cofibrant (= Kan) simplicial sets, then $k[h]:k[X]\to k[Y]$ is an $e$-connected morphism of simplicial $k$-modules.
\end{proof}

\begin{lm}\label{lm-flat}Let $\FF\to \KK$ be a field extension, and let $\A$ be a svelte additive $\FF$-linear category. If $\varrho\to \rho$ and $P\to F$ are projective simplicial  resolutions in $\KK\otimes_\FF\A\Md$ and $k[\Vect_\KK]\Md$ respectively, then the induced morphism $\varrho^*P\to \rho^*F$ is a projective simplicial resolution in $k[\A]\Md$.
\end{lm}
\begin{proof}
We already know that $\varrho^*P\to \rho^*F$ is a simplicial resolution by lemma \ref{lm-Whitehead-thm}. 
It remains to prove that $\varrho^*P$ is degreewise projective. By considering standard projectives and using that $\KK\otimes_\FF\A(a,-)\simeq \A(a,-)^{\oplus |\KK:\FF|}$ as $\FF$-modules (hence as sets), this reduces to proving that $k[\bigoplus_{i\in E}\A(a_i,-)]$ is projective in $k[\A]\Md$ for every family $(a_i)_{i\in E}$ of objects of $\A$.

For all $a_i$ there is a canonical direct sum decomposition $k[\A(a_i,-)]\simeq K_{a_i}\oplus k$
where the constant functor $k=k[0]$ is seen as a subfunctor of $k[\A(a_i,-)]$ via the inclusion induced by the unique map $0\to \A(a_i,-)$, and $K_{a_i}$ is the kernel of the map $k[\A(a_i,-)]\to k[0]=k$ induced by the unique map $\A(a_i,-)\to 0$. With this notation, we have an isomorphism 
\[\bigoplus_{I\in\mathcal{P}_f(E)}\bigotimes_{i\in I}K_{a_i}\xrightarrow[]{\simeq}k[\bigoplus_{i\in E}\A(a_i,-)]\] 
where $\Pp_f(E)$ is the set of finite subsets of $E$. Indeed, such an isomorphism is easily seen to hold if $E$ is finite, and if $E$ is infinite it may be obtained as the filtered colimit of the isomorphisms on the finite subsets $E'\subset E$. Every tensor product $\bigotimes_{i\in I}K_{a_i}$ is projective since it is a direct summand of $\bigotimes_{i\in I}k[\A(a_i,-)]\simeq k[\A(\bigoplus_{i\in I}a_i,-)]$. Therefore the functor $k[\bigoplus_{i\in E}\A(a_i,-)]$ is projective.
\end{proof}

\subsection{A general isomorphism result}

\begin{thm}\label{thm-Theta}
Let $k$ be a commutative ring and let $\A$ be a svelte additive $\FF$-linear category over a field $\FF$. Assume that $\pi$ and $\rho$ are $\FF$-linear, and that  $e$ is a positive integer such that $\Tor_i^{\FF\otimes_\FF\A}(\pi,\rho)=0$ for $0<i<e$.
Then for all objects $F$, $G$ of $k[\Vect_\FF]\Md$, the map
$$\Theta_\FF:\Tor^{k[\A]}_j(\pi^*F,\rho^*G)\to \Tor_j^{k[\Vect_\FF]}(D^*_{\pi,\rho}F,G)$$
is an isomorphism if $0\le j<e$ and is surjective if $j=e$.
\end{thm}
\begin{proof}
Let $\mathcal{G}\to G$, $\varpi\to \pi$ and $\varrho\to \rho$ be simplicial resolutions by direct sums of standard projectives in the categories $k[\Vect_\FF]\Md$, $\Mdd\FF\otimes_\FF\A$ and $\FF\otimes_\FF\A\Md$ respectively. 
We have a commutative diagram of simplicial $k$-modules, in which the maps $(\dag)$ are induced by the morphisms $\varpi\to \pi$ and $\varrho\to \rho$, the maps $\widetilde{\Theta_\FF}$ are degreewise equal to the degree zero component of $\Theta_\FF$, and composition operators for functors are omitted (e.g. $\overline{F}\varpi$ stands for $\overline{F}\circ\varpi$):
\[
\begin{tikzcd}[column sep=large]
\overline{F}\varpi\otimes_{k[\A]}\overline{\G}\varrho\ar{dd}[swap]{\widetilde{\Theta_\FF}}\ar{rd}[swap]{(\dag)}\ar{r}{F(\theta_\FF)\otimes\id}
&\overline{F}D_{\varpi,\rho}\rho\otimes_{k[\A]}\overline{\G}\varrho\ar{r}{(\dag)} & 
\overline{F}D_{\pi,\rho}\rho\otimes_{k[\A]}\overline{\G}\rho
\ar{d}{\res_\rho} \\
& \overline{F}\pi\otimes_{k[\A]}\overline{\G}\rho
\ar{rd}{\widetilde{\Theta_\FF}}\ar{ru}{F(\theta_\FF)\otimes\id} &
\overline{F}D_{\pi,\rho}\otimes_{k[\Vect_\FF^\aleph]}\overline{\G}
\\
\overline{F}D_{\varpi,\varrho}\otimes_{k[\Vect_\FF]}\G\ar{rr}{(\dag)} &&  \overline{F}D_{\pi,\rho}\otimes_{k[\Vect_\FF]}\G\ar{u}[swap]{\res_{\iota^\aleph}}{\simeq}
\end{tikzcd}\;.
\]
The homotopy groups of $\overline{F}D_{\pi,\rho}\otimes_{k[\Vect_\FF]}\overline{\G}$ and $\overline{F}\varpi\otimes_{k[\A]}\overline{\G}\varrho$ are respectively equal to 
 $\Tor_*^{k[\Vect_\FF]}(D_{\pi,\rho}^*F,G)$ and $\Tor_*^{k[\A]}(\pi^*F,\rho^*G)$ because $\overline{G}\to G$ and $\overline{\G}\varrho\to \overline{G}\rho$ are simplicial projective resolutions (see lemma \ref{lm-flat}). And the map induced on the level of homotopy groups by the top right corner of the diagram is $\Theta_\FF$. 
Moreover, by corollary \ref{cor-direct-sum-rep-Theta}, the vertical map $\widetilde{\Theta_\FF}$ is an isomorphism. Thus, to prove the theorem, it remains to prove that the bottom horizontal map $(\dag)$ is $e$-connected.

The $\Tor$-condition in the theorem ensures that $\varpi\otimes_{k[\A]}\varrho\to \pi\otimes_{k[\A]}\rho$ is $e$-connected. Thus $D_{\varpi,\varrho}\to D_{\pi,\rho}$ is $e$-connected, hence $\overline{F}D_{\varpi,\varrho}\to \overline{F}D_{\pi,\rho}$ is $e$-connected by lemma \ref{lm-Whitehead-thm}. This implies that the bottom horizontal map is $e$-connected by a standard spectral sequence argument (use the spectral sequence of a bisimplicial $k$-module as in \cite[IV section 2.2]{GoerssJardine}).
\end{proof}

\begin{rk} Although $\Theta_\FF$ can be constructed without requiring that $\A$, $\pi$ and $\rho$ are $\FF$-linear, one cannot hope that it is an isomorphism without $\FF$-linearity. Indeed, assume that $\FF$ is a field of prime characteristic $p$ of cardinal greater than $p$, contained in a field $k$. Consider $\pi(u)=\Hom_\Fp(u,\FF)$ and $\rho(u)=\FF\otimes_\Fp u$, then $D_{\pi,\rho}(v)=\Hom_\FF(v,\FF)$.  
If $G(v)=k\otimes_\FF v$ and $F(v)=S^p(k\otimes_\FF v)$ then one has: 
\begin{align*}
\pi^*F\otimes_{k[\A]}\rho^*G\simeq k &&&\text{  and  } & D_{\pi,\rho}^*F\otimes_{k[\Vect_\FF]}G =0\;.
\end{align*} 
\end{rk}

\section{The generalized comparison theorem}\label{sec-generalized}
The purpose of this section is to prove an analogue of corollary \ref{cor-verystrong} in which the category $\Vect_\Fq$ is replaced by an $\Fq$-linear category $\A$. Throughout the section:
\begin{itemize}
\item $k$ is a perfect field of characteristic $p$ containing a subfield $\Fq$ with $q=p^r$ elements, 
\item $\A$ is a svelte additive $\Fq$-linear category,
\item $\pi:\A^\op\to k\Md$ and $\rho:\A\to k\Md$ are $\Fq$-linear functors,
\item $F$ and $G$ are two strict polynomial functors over $k$.
\end{itemize}
We will prove that under a certain homological condition on $\pi$ and $\rho$, the graded vector space $\Tor^{k[\A]}_*(\pi^*F,\rho^*G)$ is isomorphic to $\Tor^\gen_*(F^\dag,G)$ for some strict polynomial functor $F^\dag$ constructed from $F$, $\pi$ and $\rho$. This result is the content of the generalized comparison theorem \ref{thm-gen-comp}, which is proved in section \ref{subsec-proof-gencomp}.


\subsection{Statement of the generalized comparison theorem}
Fix a positive integer $s$. For all integers $i$, we define additive functors:
\begin{align*}
\pi_i:={}^{(-ri)}\pi&&& \sigma:={}^{(r-rs)}\rho
\end{align*}
where a notation such as $^{(-ri)}\pi$ refers to the additive functor obtained as the composition of $\pi$ with the Frobenius twist $^{(-ri)}- :k\Md\to k\Md$. Then every $k$-vector space $v$ has a dual $k$-vector space $D_{\pi_i,\sigma}(v)$ defined by:
\[D_{\pi_i,\sigma}(v):= \Hom_k(v,\pi_i\otimes_{k[\A]}\sigma)\;.\]
We consider each $D_{\pi_i,\sigma}$ as a contravariant strict polynomial functor of degree $1$, hence the composition $^{(ri)}D_{\pi_i,\sigma}$ as a contravariant strict polynomial functor of degree $q^i$.
If $F$ is a strict polynomial functor of degree $d$ over $k$, we define a contravariant strict polynomial functor $F^\dag$ of degree $dq^{s^2-1}$ by (see definition \ref{defi-comp-strict}):
\begin{align}
F^\dag:= \overline{F}\circ \Big(\bigoplus_{0\le i<s^2} {}^{(ri)}D_{\pi_i,\sigma}\Big)\;.\label{eq-Fdag}
\end{align}

In section \ref{subsec-gen-comp-map}, we will construct an explicit morphism of graded vector spaces, natural with respect to $F$, $G$, $\pi$ and $\rho$:
\begin{align}
\Tor_j^{k[\A]}(\pi^*F,\rho^*G)\to \Tor_j^{\gen}(F^\dag,G^{(rs-r)})\;.\label{eqn-gencompmap}
\end{align}
The additive functors are allowed to have infinite-dimensional values, hence $\pi^*F$ and $\rho^*G$ refer to generalized compositions as in definition \ref{defi-comp}.
In the remainder of the article, we will refer to morphism \eqref{eqn-gencompmap} as the \emph{generalized comparison map}, and we will refer to the following theorem as the \emph{generalized comparison theorem}.

\begin{thm}\label{thm-gen-comp}
Let $k$ be a perfect field of prime characteristic $p$ containing a subfield $\Fq$ with $q=p^r$ elements, let $\A$ be a svelte additive $\Fq$-linear category, let $\pi:\A^\op\to k\Md$ and $\rho:\A\to k\Md$ be two $\Fq$-linear functors. Assume that $s$ and $e$ are positive integers such that 
\[\Tor_j^{k\otimes_\mathbb{Z} \A}\Big(\bigoplus_{0\le i<s^2}{}^{(-ri)}\pi,{}^{(r-rs)}\rho\Big)=0\]
for $0< j<e$. 
Then for all strict polynomial functors $F$ and $G$ of degrees less than $q^s$, the generalized comparison map \eqref{eqn-gencompmap} is an isomorphism for $0\le j<e$, and surjective for $j=e$.
\end{thm}

%
%
%
%


\subsection{Construction of the generalized comparison map}\label{subsec-gen-comp-map}

We first define a map $\Theta_k^\dag$ which is a variant of the auxiliary map $\Theta_k$ from section \ref{sec-gen-prelim-comparison}. Namely, we denote by $\theta_k^\dag$ the natural transformation:
\begin{align}
\theta_k^\dag:\pi^{\oplus s^2}=\bigoplus_{0\le i<s^2} {}^{(ri)}\pi_i \xrightarrow[]{\bigoplus {}^{(ri)}\theta_k} \bigoplus_{0\le i<s^2} {}^{(ri)}D_{\pi_i,\sigma}\circ\sigma\;.\label{eqn-def-thetaprime}
\end{align}
where $\theta_k$ refers to the natural transformation introduced in section \ref{subsec-defTheta}. Then we define $\Theta_k^\dag$ as the unique map making the following diagram commute:
\begin{equation}
\begin{tikzcd}[column sep=large]
\Tor_*^{k[\A]}(\overline{F}\circ (\pi^{\oplus s^2}),\overline{G}\circ\rho)\ar[dashed]{r}{\Theta^\dag_k}
\ar{d}[swap]{\Tor_*^{k[\A]}(\overline{F}(\theta_k^\dag),\overline{G}\circ\rho )}
 & \Tor_*^{k[\Vect_k]}(F^\dag,G^{(rs-r)})\ar{d}{\res_{\iota^\aleph}}[swap]{\simeq}\\
\Tor_*^{k[\A]}(F^\dag\circ \sigma,\overline{G}^{(rs-r)}\circ\sigma)\ar{r}{\res_\sigma} & \Tor_*^{k[\Vect_k^\aleph]}(F^\dag,\overline{G}^{(rs-r)})
\end{tikzcd}.\label{eqn-def-Thetaprime}
\end{equation}
The next lemma is proved exactly in the same way as lemmas \ref{lm-independant-card} and \ref{lm-nat-Theta}.
\begin{lm}
The map $\Theta^\dag_k$ does not depend on the cardinal $\aleph$. Moreover, it is natural with respect to $F$, $G$, $\pi$ and $\rho$.
\end{lm}
The map $\Theta_k^\dag$ is closely related to the map $\Theta_k$ from section \ref{sec-gen-prelim-comparison}.
Indeed, assume that we are given isomorphisms of functors $\pi\simeq{}^{(-ri)}\pi$ and $\rho\simeq{}^{(r-rs)}\rho$. Then these isomorphisms induce isomorphisms (for the notation $F^{(r|s^2)}$, see \eqref{eqn-nota-precomp-multiple-twists-discrete}, page \pageref{eqn-nota-precomp-multiple-twists-discrete}): 
\begin{align*}
&(\pi^{\oplus s^2})^*F\simeq\pi^*(F^{(r|s^2)})\;,&&\rho^*G\simeq \rho^*(G^{(rs-r)})\;,
&&F^\dag\simeq D_{\pi,\rho}^*(F^{(r|s^2)})\;,
\end{align*} 
and the next lemma is a straightforward verification.
\begin{lm}\label{lm-ThetaprimeTheta}
There is a commutative square, whose vertical isomorphisms are induced by the above isomorphisms of functors:
\[
\begin{tikzcd}
\Tor_*^{k[\A]}((\pi^{\oplus s^2})^*F,\rho^* G)\ar{r}{\Theta^\dag_k}\ar{d}{\simeq}
 & \Tor_*^{k[\Vect_k]}(F^\dag,G^{(rs-r)})\ar{d}{\simeq}\\
\Tor_*^{k[\A]}(\pi^*(F^{(r|s^2)}),\rho^*G^{(rs-r)})\ar{r}{\Theta_k}
 & \Tor_*^{k[\Vect_k]}(D_{\pi,\rho}^*(F^{(r|s^2)}),G^{(rs-r)})
\end{tikzcd}\;.
\]
\end{lm}

\begin{defi}
The \emph{generalized comparison map} \eqref{eqn-gencompmap} is the following composition of three maps: 
\[
\begin{tikzcd}
\Tor^{k[\A]}_*(\pi^*F,\rho^*G)\ar[dashed]{d}{\text{\eqref{eqn-gencompmap}}}\ar{rr}{\Lambda} && \Tor_*^{k[\A]}((\pi^{\oplus s^2})^*F,\rho^*G)\ar{d}{\Theta^\dag_k}\\
\Tor_*^{\gen}(F^\dag,G^{(rs-r)})&&\ar{ll}{\Phi_k}\Tor_*^{k[\Vect_k]}(F^\dag,G^{(rs-r)})
\end{tikzcd}\;.
\]
where $\Lambda$ is induced by $\overline{F}(\diag)$, where $\diag:\pi\to \pi^{\oplus s^2}$ is the diagonal map, and $\Phi_k$ is the strong comparison map from section \ref{subsec-strong-comparison-Fq}.
\end{defi}

\subsection{Proof of the generalized comparison theorem \ref{thm-gen-comp}}\label{subsec-proof-gencomp}

The strategy of the proof can be summarized as follows. In a first step, we prove theorem \ref{thm-gen-comp} when $\pi$ and $\rho$ are projective (as additive functors). This step relies on the fact that such projective functors factor through $\Fp$-vector spaces, and because of this factorization, the generalized comparison map can be rewritten as the composition of the auxiliary comparison map map $\Theta_\Fp$ of section \ref{sec-gen-prelim-comparison} and the generalized comparison map of section \ref{subsec-generalized-comparison-Fq}, both of which are isomorphisms. Then in a second step, one extends the isomorphism to all additive functors $\pi$ and $\rho$ by taking simplicial resolutions. The proof is similar in spirit (though slightly more complicated) to the proof of theorem \ref{thm-Theta}, in particular we rely on the simplicial results of section \ref{subsec-simplicial}. The homological condition of theorem \ref{thm-gen-comp} is used in this second step, to ensure that the tensor product of a projective simplicial resolution of $^{(-ri)}\pi$ with a projective simplicial resolution of $^{(r-rs)}\rho$ is, up to degree $e$, a projective simplicial resolution of $^{(-ri)}\pi\otimes_{k[\A]}{}^{(r-rs)}\rho$. 

\begin{pr}\label{prop-base-case}
Assume that
$$\pi=\bigoplus_{i\in I}k\otimes_\Fq\A(-,a_i)\;,\qquad \rho=\bigoplus_{j\in J}k\otimes_\Fq\A(b_j,-)$$
for some possibly infinite indexing sets $I$ and $J$. Then the generalized comparison map \eqref{eqn-gencompmap} is an isomorphism in all degrees $i$. 
\end{pr}
\begin{proof}
If $t(v)=k\otimes_\Fq v$ denotes the extension of scalars from $\Fq$ to $k$, then we have $\pi\simeq t\circ\mu$ and $\rho\simeq t\circ \nu$ for some additive functors $\mu:\A^\op\to \Fq\Md$ and $\nu:\A\to \Fq\Md$. 
The canonical isomorphisms of functors $t\simeq {}^{(ri)}t$ induce isomorphisms $\pi\simeq {}^{(ri)}\pi$ and $\rho\simeq {}^{(r-rs)}\rho$. So lemma \ref{lm-ThetaprimeTheta} and naturality of $\Phi_k$ with respect to the isomorphism $F^\dag\xrightarrow[]{\simeq} \overline{F}{}^{(r|s^2)}\circ D_{\pi,\rho}$, yield a commutative square (in which the composition operator of functors is omitted in the arguments of $\Tor$):
\[
\begin{tikzcd}[column sep=large]
\Tor_*^{k[\A]}(\overline{F}\pi,\overline{G}\rho)\ar{dd}{\text{\eqref{eqn-gencompmap}}}
\ar{r}{\Lambda'}
& 
\Tor_*^{k[\A]}(\overline{F}^{(r|s^2)}\pi,\overline{G}^{(rs-r)}\rho)\ar{d}{\Theta_k}
\\
&
\Tor_*^{k[\Vect_k]}(\overline{F}^{(r|s^2)}D_{\pi,\rho}, G^{(rs-r)})\ar{d}{\Phi_k}
\\
\Tor_*^{\gen}(F^\dag,G^{(rs-r)})\ar{r}{\simeq}
&
\Tor_*^{\gen}(\overline{F}^{(r|s^2)}D_{\pi,\rho},G^{(rs-r)})
\end{tikzcd}
\]
in which the map $\Lambda'$ is induced by the canonical isomorphism $t\simeq {}^{(rs-r)}t$ and by the skew diagonal $\diag: t\to \bigoplus_{0\le i<s^2}{}^{(ri)}t$.
Therefore in order to prove proposition \ref{prop-base-case} it suffices to prove that the composition $\Phi_k\circ\Theta_k\circ\Lambda'$ in top right corner of the diagram is an isomorphism. 

Next, since $\pi=t\circ\mu$ and $\rho=t\circ\nu$, the base change property of proposition \ref{prop-chgt-base} gives a commutative square:
\[
\begin{tikzcd}
\Tor_*^{k[\A]}(\overline{F}^{(r|s^2)}\pi,\overline{G}^{(rs-r)}\rho)\ar{d}{\Theta_k}\ar{r}{\Theta_\Fq}
&
\Tor_*^{k[\Vect_\Fq]}(\overline{F}^{(r|s^2)}t D_{\mu,\nu},\overline{G}^{(rs-r)} t)
\ar{d}{\simeq}
\\
\Tor_*^{k[\Vect_k]}(\overline{F}^{(r|s^2)}D_{\pi,\rho}, G^{(rs-r)})
&
\Tor_*^{k[\Vect_\Fq]}(\overline{F}^{(r|s^2)}D_{\pi,\rho}t,\overline{G}^{(rs-r)} t)
\ar{l}{\res_t}
\end{tikzcd}\;.
\]
Hence, by naturality of $\Theta_{\Fq}$ with respect to the isomorphism $\overline{G}(\iso):\overline{G}t\simeq \overline{G}{}^{(rs-r)}t$ and the morphism $F(\diag): \overline{F}t\to \overline{F}{}^{(r|s^2)}$, we obtain a commutative square:
\[
\begin{tikzcd}
\Tor_*^{k[\A]}(\overline{F}\pi,\overline{G}\rho)\ar{d}{\Theta_k\circ\Lambda'}
\ar{r}{\Theta_\Fq}
&
\Tor_*^{k[\Vect_\Fp]}(\overline{F}t D_{\mu,\nu},\overline{G}t)\ar{d}{\Lambda''}
\\
\Tor_*^{k[\Vect_k]}(\overline{F}^{(r|s^2)}D_{\pi,\rho}, G^{(rs-r)})
&
\Tor_*^{k[\Vect_\Fp]}(\overline{F}^{(r|s^2)}D_{\pi,\rho}t,\overline{G}^{(rs-r)} t)
\ar{l}{\res_t}
\end{tikzcd}
\]
where $\Lambda''$ is induced by the skew diagonal $\diag: t\to {}^{(r|s^2)}t$, by the isomorphism $\iso:t\simeq {}^{(rs-r)}t$ and by the canonical isomorphism $\mathrm{can}: t\circ D_{\mu,\nu}\simeq D_{\pi,\rho}\circ t$.
The map $\Theta_\Fq$ on the top row of this square is an isomorphism by corollary \ref{cor-direct-sum-rep-Theta}. Hence, to prove proposition \ref{prop-base-case} it remains to prove that the composition $\Phi_k\circ\res_t\circ \Lambda''$ is an isomorphism.

For this purpose, we are going to rewrite the composition $\Phi_k\circ\res_t\circ \Lambda''$ into yet another form.
We claim that there is a $k$-linear isomorphism, natural with respect to $v$, $\mu$ and $\nu$:
\[\psi: {}^{(r|s^2)}D_{t\mu,t\nu}(v)\to{D_{t\mu,t\nu}}({}^{(r|s^2)}v)\;, \]
Indeed, we have isomorphisms of vector spaces, natural with respect to $\mu$ and $\nu$:
\[\phi_i:{}^{(ri)}\left(t\mu\otimes_{k[\A]}t\nu\right)\xrightarrow[]{\simeq} t\mu\otimes_{k[\A]}t\nu\]
which send the class  $\llbracket(\alpha\otimes x)\otimes (\beta\otimes y)\rrbracket$ where $\alpha,\beta\in k$, $x\in \mu(a)$ and $y\in \nu(a)$ to the class $\llbracket(\alpha^{q^{i}}\otimes x)\otimes (\beta^{q^{i}}\otimes y)\rrbracket$. We define $\psi$ as the following composition, where $T$ stands for $t\mu\otimes_{k[\A]}t\nu$, the first and last isomorphisms are the canonical ones and the second morphism is induced by the $\phi_i$:
\begin{align*}
{}^{(r|s^2)}D_{t\mu,t\nu}(v)=\bigoplus_{0\le i<s^2} {}^{(ri)}\Hom_k(v,T) &\xrightarrow[]{\simeq} \bigoplus_{0\le i<s^2} \Hom_k({}^{(ri)}v,{}^{(ri)}T)\\
&\xrightarrow[]{\simeq} \bigoplus_{0\le i<s^2} \Hom_k({}^{(ri)}v,T)\\
&\xrightarrow[]{\simeq} \Hom_k(\bigoplus_{0\le i<s^2} {}^{(ri)}v,T)=D_{t\mu,t\nu}({}^{(r|s^2)}v)\;.
\end{align*}
Moreover, one readily checks that $\psi$ fits into a commutative diagram in $\Mdd k[\Fp]$:
\begin{equation}
\begin{tikzcd}[column sep=large]
tD_{\mu,\nu}\ar{d}{\mathrm{can}}[swap]{\simeq}\ar{r}{\diag (D_{\mu,\nu})}& {}^{(r|s^2)}tD_{\mu,\nu}\ar{r}{{}^{(r|s^2)}\mathrm{can}}[swap]{\simeq}&{}^{(r|s^2)}D_{\pi,\rho}t\ar{d}{\psi(t)}[swap]{\simeq}\\
D_{\pi,\rho}t\ar{rr}{D_{\pi,\rho}(\summ)}&& D_{\pi,\rho}{}^{(r|s^2)}t
\end{tikzcd}\label{eqn-diagram-fin}
\end{equation}
where $\summ:{}^{(r|s^2)}t\to t$ is the skew sum map. 
Diagram \eqref{eqn-diagram-fin} and naturality of $\Phi_k$ and $\res_t$ with respect to the map $\overline{F}(\psi)$ yield a commutative square, in which the horizontal isomorphisms are induced by the isomorphisms $\overline{F}(\mathrm{can})$ and $\overline{F}(\psi)$ and the map $\Lambda''$ is induced by $\overline{F}(D_{\pi,\rho}(\summ))$:
\[
\begin{tikzcd}
\Tor_*^{k[\Vect_\Fq]}(\overline{F}t D_{\mu,\nu},\overline{G}t)\ar{r}{\simeq}\ar{d}{\Phi_k\circ\res_t\circ \Lambda''}& \Tor_*^{k[\Vect_\Fq]}(\overline{F}D_{\pi,\rho}t,\overline{G}t)\ar{d}{\Phi_k\circ\res_t\circ \Lambda'''}\\
\Tor_*^{\gen}(\overline{F}{}^{(r|s^2)}D_{\pi,\rho}, G^{(rs-r)})\ar{r}{\simeq}& \Tor_*^{\gen}(\overline{F}{D_{\pi,\rho}}^{(r|s^2)}, G^{(rs-r)})
\end{tikzcd}\;.
\]
Thus, to prove proposition \ref{prop-base-case}, it suffices to prove that $\Phi_k\circ\res_t\circ \Lambda'''$ is an isomorphism. But $\Phi_k\circ\res_t\circ \Lambda'''$ equals the comparison map of equation \eqref{eqn-verystrong-Tor}, hence it is an isomorphism by corollary \ref{cor-verystrong}. This concludes the proof.
\end{proof}
%

\begin{proof}[Proof of theorem  \ref{thm-gen-comp}]
For concision, we set $\sigma={}^{(r-rs)}\rho$ and $H=G^{(rs-r)}$. Thus, $\rho^*G=\sigma^*H$. To emphasize the dependence of $F^\dag$ on $\pi$ and $\sigma$ we set:
\[F^\dag_{\pi,\sigma}(v):=\overline{F}\Big(\bigoplus_{0\le i< s^2}{}^{(ri)}\Hom_k\big(v,({}^{(-ri)}\pi)\otimes_{k[\A]}\sigma\big)\,\Big)\;.\]
We fix an integer $n\gg 0$ ($n> \log_p(e/2)$ suffices) such that the canonical map
\begin{align}\Tor_*^{\gen}(F^\dag_{\pi,\sigma},H)\to \Tor_*^{\Pp_k}((F^\dag_{\pi,\sigma})^{(n)},H^{(n)})
\label{eqn-cancomp}
\end{align}
is $e$-connected. Let 
\[\Psi_k:  \Tor_*^{k[\Vect_k]}(F^\dag_{\pi,\sigma},H)\to    \Tor_*^{\Pp_k}((F^\dag_{\pi,\sigma})^{(n)},H^{(n)})\]
be the morphism given by restriction along $^{(-n)}-$ and by the restriction from ordinary functors to strict polynomial functors. Then $\Psi_k$ equals the composition of the canonical map \eqref{eqn-cancomp} with the strong comparison map $\Phi_k$. Thus, in order to prove theorem \ref{thm-gen-comp}, it suffices to prove that the composition $\Psi_k\circ\Theta^\dag_k\circ\Lambda$ is $e$-connected.

For this purpose, we are going to realize the $\Tor$ vector spaces in play as homotopy groups of some simplicial vector spaces, and the maps $\Psi_k$, $\Theta^\dag_k$ and $\Lambda$ as the morphisms induced on the level of homotopy groups by some morphisms of simplicial vector spaces. Let
\begin{align*}
\mathcal{H}\to H\;, && \mathcal{H}'\to H^{(n)}\;, && \varpi\to \pi\;, && \varsigma\to \sigma\;,
\end{align*}
be projective simplicial resolutions in the categories 
\begin{align*}
k[\Vect_k]\Md\;, && \Pp_k\;, && \Mdd (k\otimes_\mathbb{Z}\A)\;, && (k\otimes_\mathbb{Z}\A)\Md
\end{align*}
respectively.
It follows from lemmas \ref{lm-Whitehead-thm} and  \ref{lm-flat} that $\varsigma^*\mathcal{H}$ is a simplicial projective resolution of $\sigma^*H$ in $k[\A]\Md$, and that $(\varpi^{\oplus i})^*F$ is a simplicial (not projective) resolution of $(\pi^{\oplus i})^*F$ in $\Mdd k[\A]$ for all positive integers $i$. Therefore we have identifications of homotopy groups:
\begin{align*}
&\pi_*\left( \varpi^*F\otimes_{k[\A]}\varsigma^*\mathcal{H}\right)=\Tor_*^{k[\A]}(\pi^*F,\sigma^*H)\;,\\
&\pi_*\left( (\varpi^{\oplus s^2})^*F\otimes_{k[\A]}\varsigma^*\mathcal{H}\right)=\Tor_*^{k[\A]}((\pi^{\oplus s^2})^*F,\sigma^*H)\;,\\
&\pi_*\left( F_{\pi,\sigma}^\dag\otimes_{k[\A]}\mathcal{H}\right)=\Tor_*^{k[\Vect_k]}(F_{\pi,\sigma}^\dag,H)\;,\\
&\pi_*\left( (F_{\pi,\sigma}^\dag)^{(n)}\otimes_{\Pp_k}\mathcal{H}'\right)=\Tor_*^{\Pp_k}((F_{\pi,\sigma}^\dag)^{(n)},H^{(n)})\;.
\end{align*}
Moreover, let $f:\mathcal{H}\to {\mathcal{H}'}^{(-n)}$ be a simplicial morphism in $k[\Vect_k]\Md$ lifting the identity morphism of $H$. Then the maps $\Lambda$, $\Theta_k^\dag$ and $\Psi_k$ are respectively induced by the morphisms of simplicial $k$-vector spaces:
\begin{align*}
&\varpi^*F\otimes_{k[\A]}\varsigma^*\mathcal{H}\xrightarrow[]{\widetilde{\Lambda}}  (\varpi^{\oplus s^2})^*F\otimes_{k[\A]}\varsigma^*\mathcal{H}\;,\\
&(\varpi^{\oplus s^2})^*F\otimes_{k[\A]}\varsigma^*\mathcal{H}\xrightarrow[]{\widetilde{\Theta_k^\dag}}F_{\varpi,\varsigma}^\dag\otimes_{k[\Vect_k]}\mathcal{H}\xrightarrow[]{(\dag)} F_{\pi,\sigma}^\dag\otimes_{k[\Vect_k]}\mathcal{H}\;,\\
&F_{\pi,\sigma}^\dag\otimes_{k[\Vect_k]}\mathcal{H}\xrightarrow[]{\id\otimes f} (F_{\pi,\sigma}^\dag)^{(n)(-n)}\otimes_{k[\Vect_k]}{\mathcal{H}'}^{(-n)}\xrightarrow[]{\widetilde{\Psi_k}} (F_{\pi,\sigma}^\dag)^{(n)}\otimes_{\Pp_k}\mathcal{H}'\;.
\end{align*} 
Here, the simplicial morphism $(\dag)$ is induced by the simplicial morphisms $\varpi\to \pi$ and $\varsigma\to \sigma$ and $\widetilde{\Lambda}$, $\widetilde{\Theta'_k}$ and $\widetilde{\Psi_k}$ are degreewise equal to the degree zero component of $\Lambda$, $\Theta'_k$ and $\Psi_k$, e.g. the component of $\widetilde{\Lambda_k}$ in degree $j$  is equal to 
$$\Tor_0^{k[\A]}(\varpi^*_jF,\varsigma^*_j\mathcal{H}_j)\xrightarrow[]{\Lambda} \Tor_0^{k[\A]}((\varpi_j^{\oplus s^2})^*F,\varsigma^*_j\mathcal{H}_j)\;.$$
(That these simplicial morphisms induce our maps $\Lambda$, $\Theta_k^\dag$ and $\Psi_k$ is obvious for the first one and the last one. For $\Theta^\dag_k$, this follows by the same reasoning as in the proof of theorem \ref{thm-Theta}). 
We deduce that the comparison map $\Psi_k\circ\Theta^\dag_k\circ\Lambda$ equals the map induced on homotopy groups by the composition of the simplicial morphism:
\begin{align}
\left(\widetilde{\Psi_k}\circ (\id\otimes f)\circ \widetilde{\Theta^\dag_k}\circ \widetilde{\Lambda}\right)\;: \;\varpi^*F\otimes_{k[\A]}\varsigma^*\mathcal{H}\to (F_{\varpi,\varsigma}^\dag)^{(n)}\otimes_{\Pp_k}\mathcal{H}'\label{eqn-morph1}
\end{align}
followed by the simplicial morphism induced by $\varpi\to \pi$ and $\varsigma\to \sigma$:
\begin{align}
(F_{\varpi,\varsigma}^\dag)^{(n)}\otimes_{\Pp_k}\mathcal{H}'\to (F_{\pi,\sigma}^\dag)^{(n)}\otimes_{\Pp_k}\mathcal{H}'
\label{eqn-morph2}
\end{align}

Now we are going to show that the simplicial morphisms \eqref{eqn-morph1} and \eqref{eqn-morph2} are $e$-connected. 
Let us regard the source and the target of \eqref{eqn-morph1} as the diagonal of bisimplicial objects $\varpi_i^*F\otimes_{k[\A]}\varsigma^*_i\mathcal{H}_j$ and 
$(F_{\varpi_i,\varsigma_i}^\dag)^{(n)}\otimes_{\Pp_k}\mathcal{H}_j'$ with bisimplicial degrees $(i,j)$. Then the simplicial morphism \eqref{eqn-morph1} actually comes from a bisimplicial morphism. Spectral sequences of bisimplicial $k$-modules as in \cite[IV section 2.2]{GoerssJardine} yield two spectral sequences:
\begin{align*}
& \mathrm{I}^1_{ij}=\Tor_j^{k[\A]}(\varpi_i^*F,\varsigma^*_i H)\Longrightarrow \pi_{i+j}\left(\varpi^*F\otimes_{k[\A]}\varsigma^*\mathcal{H} \right)\;,\\
& \mathrm{II}^1_{ij}=\Tor_j^{\Pp_k}(F^\dag_{\varpi_i,\varsigma_i},H)\Longrightarrow \pi_{i+j}\left((F_{\varpi,\varsigma}^\dag)^{(n)}\otimes_{\Pp_k}\mathcal{H}'\right)\;,
\end{align*}
And there is a morphism of spectral sequences $\mathrm{I}\to \mathrm{II}$ which coincides with the morphism \eqref{eqn-morph1} on the abutment, and with the map $\Psi_k\circ\Theta^\dag_k\circ\Lambda$ on the first page. 
By proposition \ref{prop-base-case}, the map $\Phi_k\circ\Theta^\dag_k\circ\Lambda$ is an isomorphism when $\pi$ and $\rho$ are direct sums of standard projectives, hence the morphism of spectral sequences is an isomorphism on the first page. Hence the simplicial morphism \eqref{eqn-morph1} is an isomorphism on the level of homotopy groups (hence $e$-connected).

Thus, it remains to prove that the simplicial morphism \eqref{eqn-morph2} is $e$-connected. The $\Tor$-vanishing hypothesis of theorem \ref{thm-gen-comp} implies that the maps $({}^{(-ri)}\varpi)\otimes_{k[\A]}\varsigma \to ({}^{(-ri)}\pi)\otimes_{k[\A]}\sigma$ are $e$-connected for all integers $i$. Hence the map $F^\dag_{\varpi,\varsigma}\to F^\dag_{\pi,\sigma}$ is $e$-connected by lemma \ref{lm-Whitehead-thm}, hence the simplicial morphism \eqref{eqn-morph2} is $e$-connected by the usual bisimplicial spectral sequence argument.
\end{proof}

\section{Special cases of the generalized comparison theorem}\label{sec-special-cases}
In this section we spell out explicitly some special cases and direct consequences of the generalized comparison theorem.

\subsection{Linearity over a big subfield $\FF$} The statement of the generalized comparison theorem can be simplified when the functors $\pi$ and $\rho$ are linear over a big subfield $\FF$. 
For all strict polynomial functors $F$ and $G$ and all additive functors $\pi$ and $\rho$ we consider the composition:
\begin{align}\Tor^{k[\A]}_i(\pi^*F, \rho^*G)\xrightarrow[]{\Theta_k}  \Tor_i^{k[\Vect_k]}(D_{\pi,\rho}^*F,G) \xrightarrow[]{\Phi_k}
\Tor_i^{\gen}(D_{\pi,\rho}^*F,G)\;,\label{eqn-comp-simplified}
\end{align}
where $\Theta_k$ is the auxiliary comparison map of section \ref{sec-gen-prelim-comparison}, and $\Phi_k$ is the strong comparison map for $\FF=k$ (note that when $\FF=k$, the base change functor $t$ is the identity functor, hence it can be removed from the notation).

\begin{thm}\label{thm-F-lin-case}
Let $k$ be an infinite perfect field of positive characteristic, containing a subfield $\FF$ and let $\A$ be an additive $\FF$-linear category. Let $\pi$ and $\rho$ be two $\FF$-linear functors from $\A$ to $k$-vector spaces, respectively contravariant and covariant, and let $F$ and $G$ be two strict polynomial functors with degrees less than the cardinal of $\FF$.
Assume furthermore that 
\[\Tor_j^{k\otimes_{\mathbb{Z}}\A}(\pi,\rho)=0\quad \text{for $0<j<e$.}\]
Then the map \eqref{eqn-comp-simplified} is an isomorphism if $j< e$, and it is surjective if $j=e$.
\end{thm}
\begin{proof}
If $\FF$ is a finite field, we can apply theorem \ref{thm-gen-comp} with $s=1$. Then the generalized comparison map \eqref{eqn-gencompmap} is equal to the map \eqref{eqn-comp-simplified}, which proves the result.

Now assume that $\FF$ is infinite. Then $\Tor_*^{k\otimes_\mathbb{Z}\A}({}^{(i)}\pi,{}^{(j)}\rho)=0$ for all $i\ne j$ by lemma \ref{lm-homogeneity-vanish} below. Moreover, composition with the equivalence of categories ${}^{(j)}:k\Md\to k\Md$ induces an isomorphism $\Tor_*^{k\otimes_\mathbb{Z}\A}(\pi,\rho)\simeq \Tor_*^{k\otimes_\mathbb{Z}\A}({}^{(j)}\pi,{}^{(j)}\rho)$ for all $j$. Thus we may apply theorem \ref{thm-gen-comp} for $r=1$ ($\Fq$ is the prime field) and $s$ big enough, in order that $p^s$ is greater or equal to the degrees of $F$ and $G$. In this situation, the expression of $F^\dag$ simplifies because $({}^{(i)}\pi\otimes_{k[\A]}{}^{(j)}\rho=0$ for all $i\ne j$. Namely if we let $\alpha={}^{(1-s)}\pi$ and $\beta={}^{(1-s)}\rho$ then we have: 
\[F^\dag=\overline{F}^{(s-1)}\circ D_{\alpha,\beta}\;.\]
Moreover one readily checks that the composite map 
\[\Theta^\dag_k\circ \Lambda_k\;:\; \Tor_*^{k[\A]}(\pi^*F,\rho^*G)\to \Tor_*^{k[\Vect_k]}(D_{\alpha,\beta}^*(F^{(s-1)}), G^{(s-1)})\]
is equal to the map $\Theta_k$ relative to $\alpha={}^{(1-s)}\pi$ and $\beta={}^{(1-s)}\rho$. 

Now we consider the following diagram, in which the composition of functors is omitted, $T_*$ stands for $\Tor_*$, and we use the following notations $x:=s-1$, $D=D_{\pi,\rho}$, $D'=D_{\alpha,\beta}$. The Frobenius twist functor $^{(x)}-:k\Md\to k\Md$ is isomorphic to the extension of scalars along the morphism of fields $k\to k$, $\lambda\mapsto \lambda^{p^x}$, hence we have a canonical isomorphism $\mathrm{can}:{}^{(x)}D'\simeq D^{(x)}$, and the isomorphisms $(*)$ in this diagram are induced by this canonical isomorphism.
\[
\begin{tikzcd}
T_*^{k[\A]}(\overline{F}\pi,\overline{G}\rho)\ar{r}{\Theta_k} \ar{d}{\Theta_k}
& T_*^{k[\Vect_k]}(\overline{F}^{(x)}D',G^{(x)})\ar{r}{\Phi_k}\ar{d}{\res^{^{(x)}-}}[swap]{\simeq}
& T_*^{\gen}(\overline{F}{}^{(x)}D',G^{(x)})\ar{d}{\res^{^{(x)}-}}[swap]{\simeq}\\
T_*^{k[\Vect_k]}(\overline{F}D,G) \ar{r}{\simeq}[swap]{(*)}\ar{d}{\Phi_k}
& T_*^{k[\Vect_k]}(\overline{F}{}^{(x)}D'{}^{(-x)},G)\ar{r}{\Phi_k}
& T_*^{\gen}(\overline{F}{}^{(x)}D'{}^{(-x)},G)\\
T_*^{\gen}(\overline{F}D,G)\ar{rru}{\simeq}[swap]{(*)}& &
\end{tikzcd}
\]
The diagram is commutative: the upper left square commutes by the base change property of proposition \ref{prop-chgt-base}, the upper right square and the triangle commute by naturality of $\Phi_k$. As explained above, the composite map corresponding to the upper row is $e$-connected by theorem \ref{thm-gen-comp}. Therefore, the composite given by the first column is also $e$-connected. But this composite is nothing but the map \eqref{eqn-comp-simplified}. This finishes the proof of the theorem.
\end{proof}

The next vanishing lemma is used in the proof of theorem \ref{thm-F-lin-case}. In this lemma, $\A$ is a $\FF$-linear category, and we say that an additive functor $\alpha:\A\to k\Md$ is \emph{$d$-homogeneous} if $\alpha(\lambda f)=\lambda^d\alpha(f)$ for all morphisms $f$ in $\A$ and all $\lambda\in\FF$. 

\begin{lm}\label{lm-homogeneity-vanish}
Let $\FF$ be a subfield of $k$, and let $d\ne e$ be two non-negative integers less than the cardinal of $\FF$. Let $\pi$ and $\rho$ are two additive functors from $\A$ to $k\Md$ which are respectively contravariant and covariant. If $\pi$ is $d$-homogeneous and $\rho$ is $e$-homogeneous, then 
$\Tor_*^{k\otimes_\mathbb{Z}\A}(\pi,\rho)=0$.
\end{lm}
\begin{proof}
Let $D_k\pi$ denote the dual of $\pi$, i.e. $D_k\pi(v)=\Hom_k(\pi(v),k)$.
Then the $k$-linear dual of $\Tor_j^{k\otimes_\mathbb{Z}\A}(\pi,\rho)$ is isomorphic to the $\Ext^j_{k\otimes_{\mathbb{Z}}\A}(\rho,D_k\pi)$, hence it suffices to prove that the latter is zero for all $j$.

Let $\alpha$ and $\beta$ be two arbitrary objects of $k\otimes_\mathbb{Z}\A\Md$.
Since $\A$ is $\FF$-linear, every $\lambda\in\FF$ yields a natural transformation $\lambda_\alpha\in \End_{k\otimes_{\mathbb{Z}}\A}(\alpha)$ whose component at $x$ equals $\alpha(\lambda \id_x)$. Thus $\Ext^*_{k\otimes_{\mathbb{Z}}\A}(\alpha,\beta)$ has an $\FF$-$\FF$-bimodule structure given by $\lambda\cdot [\xi]\cdot \mu = [\mu_\beta\circ \xi\circ \lambda_\alpha]$, where $-\circ\lambda_\alpha$ is the pullback of an extension along $\lambda_\alpha$ and $\mu_\beta\circ-$ is the pushout of an extension along $\mu_\beta$. Moreover, for all morphisms $f:\gamma\to \delta$ in $k\otimes_\mathbb{Z}\A\Md$ we have $f\circ \lambda_\gamma=\mu_\delta\circ f$, which implies that the two $\FF$-module structures on $\Ext^*_{k\otimes_{\mathbb{Z}}\A}(\alpha,\beta)$ coincide: $\lambda\cdot [\xi]=[\xi]\cdot \lambda$.
Assume now that $\alpha$ is $d$-homogeneous and $\beta$ is $e$-homogeneous. Then $\lambda_\alpha=\lambda^d\id_\alpha$ and $\lambda_\beta=\lambda^e\id_\beta$. Thus for all extensions $[\xi]$ we have $\lambda^d[\xi]=\lambda_\alpha\cdot [\xi]=[\xi]\cdot\lambda_\beta=\lambda^e[\xi]$.
Since the cardinal of $\FF$ is greater than $d$ and $e$, this implies that $[\xi]=0$ for all $[\xi]\in\Ext^*_{k\otimes_{\mathbb{Z}}}(\alpha,\beta)$.
The result now follows by taking $\alpha=\rho$ and $\beta=D_k\pi$.
\end{proof}

Recall that each strict polynomial functor $F$ has an associated contravariant strict polynomial functor $F^{\vee}(-)=F(\Hom_k(-,k))$.  

\begin{cor}\label{cor-infinite-perfect-field}
Let $k$ be an infinite perfect field of positive characteristic, the strong comparison map 
\begin{align*}
&\Phi_k:\Tor_*^{k[\Vect_k]}(F^\vee,G)\to \Tor^\gen_*(F^\vee,G)
\end{align*}
is an isomorphism for all strict polynomial functors $F$ and $G$.
\end{cor}
\begin{proof}
We wish to apply theorem \ref{thm-F-lin-case} with $\pi(v)={}^{\vee}v$, $\rho(v)=v$ and $\FF=k$. 
As explained in the proof of \cite[Lm 6.1]{DT-add}, the Eilenberg-Watts theorem yields an isomorphism with the Hochschild homology of $k$:
\[\Tor_*^{k\otimes_{\mathbb{Z}}\Vect_k}(\pi,\rho)\simeq \Tor_*^{k\otimes_\mathbb{Z}k}(k,k)=\mathrm{HH}_*(k)\;,\]
and the Hochschild-Kostant-Rosenberg theorem together with the perfectness of $k$ imply that this Hochschild homology is zero in positive degrees. Hence $\Phi_k\circ \Theta_k$ is an isomorphism by theorem \ref{thm-F-lin-case}. 
Moreover, $\pi$ and $\rho$ are standard projectives of $\Mdd k\otimes_{k}\Vect_k$ and $k\otimes_{k}\Vect_k\Md$ respectively, hence $\Theta_k$ is an isomorphism by proposition \ref{pr-iso-rep-Theta}. Thus, $\Phi_k: \Tor_*^{k[\Vect_k]}(D_{\pi,\rho}^*F,G)\to \Tor^\gen_*(D_{\pi,\rho}^*F,G)$ is an isomorphism. 
To finish the proof, it suffices to observe that $\mathrm{HH}_0(k)=k$, so that $D_{\pi,\rho}\simeq {}^\vee-$.
\end{proof}

\subsection{The generalized comparison theorem for $\Ext$}\label{subsec-Ext-comp}

We now dualize the generalized comparison theorem and its special cases. We start with the simplest case of corollary~\ref{cor-infinite-perfect-field}. Recall that the strong comparison map $\Phi_k$ from definition \ref{defi-strong-comp-map}.
\begin{thm}\label{thm-infinite-perfect-field-Ext}
Let $k$ be an infinite perfect field of positive characteristic. For all strict polynomial functors $F$ and $G$ the strong comparison map 
$$\Phi_k:\Ext^i_{\gen}(F,G)\to 
\Ext^i_{k[\Vect_k]}(F,G)$$
is a graded isomorphism.
\end{thm}
\begin{proof}
By a standard spectral sequence argument, the proof reduces to the case where $G$ is a standard injective, hence when $G=\Hom_k(E,k)$, where $E$ is a standard projective in $\Pp_k$ precomposed by duality $^{\vee}-$. Then $\Phi_k$ is an isomorphism as a consequence of corollary \ref{cor-infinite-perfect-field} and the duality diagram \eqref{cd-phiphi} (page~\pageref{cd-phiphi}).
\end{proof}

Similarly, one can dualize theorem \ref{thm-F-lin-case}. To be more specific, given two additive functors $\rho,\sigma:\A\to k\Md$ and a $k$-vector space $v$, we let 
$$T_{\rho,\sigma}(v):= \Hom_{k[\A]}(\rho,\sigma)\otimes v\;.$$
Then for all strict polynomial functors $G$ and $K$ we have a map 
$$\Theta_k:\Ext^*_{k[\Vect_k]}(F,T^*_{\rho,\sigma}G)\to \Ext^*_{k[\A]}(\rho^*F,\sigma^*G)$$
induced by restriction along $\rho$ and by the canonical evaluation morphism $\mathrm{ev}:\Hom_{k[\A]}(\rho,\sigma)\otimes \rho\to \sigma$. 
\begin{thm}\label{thm-thm-F-lin}
Let $k$ be an infinite perfect field of positive characteristic, containing a subfield $\FF$ and let $\A$ be a svelte additive $\FF$-linear category. Let $\rho,\sigma:\A\to \Vect_k$ be two $\FF$-linear functors such that $\sigma$ takes finite-dimensional values and $\Hom_{k[\A]}(\rho,\sigma)$ is finite-dimensional. 
Let $e$ be a positive integer such that  
\[\Ext^j_{k\otimes_{\mathbb{Z}}\A}(\rho,\sigma)=0\quad \text{for $0<j<e$.}\]
Then for all strict polynomial functors $F$ and $G$ with degrees less than the cardinal of $\FF$, the map
$$\Theta_k\circ\Phi_k: \Ext^j_{\gen}(F,T^*_{\rho,\sigma}G)\to \Ext^j_{k[\A]}(\rho^*F,\sigma^*G)$$
is an isomorphism if $j<e$ and it is injective if $j=e$.
\end{thm}
\begin{proof}
In this proof, we let $^\vee-=\Hom_k(-,k)$ and we omit the composition operator for functors, e.g. if $K$ is a strict polynomial functor, $^\vee K^\vee$ stands for $({}^\vee-)\circ K\circ ({}^\vee-)$. 

We first prove the result when $G={}^\vee K^\vee$ for some $K$ in $\Pp_k$. We let $\pi:={}^\vee\sigma$ and we let  
$\xi: T_{\sigma,\rho}\to {}^\vee D_{\pi,\rho}$ be
the morphism of functors whose component at a vector space $v$ is given by the composition
$$\xi\;:\;v\otimes\Hom_{k[\A]}(\rho,{}^\vee\pi)\xrightarrow[]{\simeq} v\otimes {}^{\vee}(\pi\otimes_{k[\A]}\rho)\to {}^\vee\Hom_k(v, \pi\otimes_{k[\A]}\rho)$$
where the first map is the usual adjunction isomorphism and the second map is the canonical map $\mathrm{can}:v\otimes {}^\vee w\to {}^\vee\Hom_k(v,w)$ such that $\mathrm{can}(x\otimes f)(\phi)=f(\phi(x))$, and which is an isomorphism if $v$ is finite dimensional.
One readily checks that the composition 
$$ T_{\rho,\sigma}\rho\xrightarrow[]{\xi\rho} {}^\vee D_{\pi,\rho}\rho\xrightarrow[]{{}^\vee\theta_\FF} {}^{\vee}\pi=\sigma$$
equals the canonical evaluation map $\mathrm{ev}$. The finite dimensionality hypotheses on the values of $\sigma$ and on $\Hom_{k[\A]}(\rho,\sigma)$ imply that: 
\begin{enumerate}
\item[i)] 
$\xi:T_{\sigma,\rho}\to {}^\vee D_{\pi,\rho}$ is an isomorphism,
\item[ii)] $\sigma^*(G^\vee)=G^{\vee}\sigma=G\pi=\pi^*G$,
\item[iii)]$^{\vee}{}^\vee D_{\pi,\rho}$ identifies with $D_{\pi,\rho}$.
\end{enumerate}
Hence we have a commutative diagram
\[
\begin{tikzcd}
\Ext^*_{k[\Vect_k]}(F, {}^\vee K^\vee T_{\rho,\sigma})\ar{r}{\res^\rho}&
\Ext^*_{k[\A]}(F\rho, {}^\vee K^\vee T_{\rho,\sigma}\rho)\ar{r}{{}^\vee K^\vee(\mathrm{ev})}&
\Ext^*_{k[\A]}(F\rho, {}^\vee K^\vee \sigma)\\
{}^\vee\Tor_*^{k[\Vect_k]}(K^\vee T_{\rho,\sigma},F)\ar{u}[swap]{}{\simeq}\ar{r}{{}^\vee\res_\rho}\ar{d}{K^\vee(\xi)}[swap]{\simeq}&
{}^\vee\Tor_*^{k[\A]}(K^\vee T_{\rho,\sigma}\rho,F\rho)
\ar{u}[swap]{}{\simeq}\ar{r}{K^\vee(\mathrm{ev})}\ar{d}{K^\vee(\xi\rho)}[swap]{\simeq}
&{}^\vee\Tor_*^{k[\A]}(K^\vee\sigma,F\rho)\ar{u}[swap]{}{\simeq}
\\
{}^\vee\Tor_*^{k[\Vect_k]}(KD_{\pi,\rho},F)\ar{r}{{}^\vee\res_\rho}&
{}^\vee\Tor_*^{k[\A]}(KD_{\pi,\rho}\rho,F\rho) \ar{r}{K(\theta_\FF)}
& {}^\vee\Tor_*^{k[\A]}(K\pi,F\rho)\ar[equal]{u}
\end{tikzcd}
\]
from which we deduce that the graded map $\Theta_k\circ\Phi_k$ fits into a commutative square 
\begin{equation}
\begin{tikzcd}[column sep=large]
\Ext^*_{\gen}(F, T_{\rho,\sigma}^*({}^\vee K^\vee) )\ar{r}{\Theta_k\circ\Phi_k}\ar{d}{\simeq}&
\Ext^*_{k[\A]}(\rho^*F, \sigma^*({}^\vee K^\vee))\ar{d}{\simeq}\\
{}^\vee\Tor_*^{\gen}(D_{\pi,\rho}^*K,F)\ar{r}{{}^\vee(\Theta_k\circ\Phi_k)}&{}^\vee\Tor_*^{k[\A]}(\pi^*K,\rho^*F)
\end{tikzcd}\label{cd-dual}
\end{equation}
where the bottom arrow is dual to the map \eqref{eqn-comp-simplified} of theorem \ref{thm-F-lin-case}. Since $\Ext^*_{k\otimes_\mathbb{Z}\A}(\rho,\sigma)$ is isomorphic to ${}^\vee\Tor_*^{k\otimes_\mathbb{Z}\A}(\pi,\rho)$, this bottom arrow is an isomorphism by theorem \ref{thm-F-lin-case}, hence theorem \ref{thm-thm-F-lin} holds when $G={}^\vee K^\vee$.

This proves in particular that theorem \ref{thm-thm-F-lin} holds for all strict polynomial functors $G$ with finite dimensional values, and hence for the standard injectives. For an arbitrary $G$ one can consider an injective resolution and the result follows by a standard spectral sequence argument.
\end{proof}

With the same strategy, one can also dualize theorem \ref{thm-gen-comp}. Given a strict polynomial functor $G$, we denote by $G^\ddag$ the strict polynomial functor such that 
$$G^\ddag(v)=\overline{G}\left(\bigoplus_{0\le i<s^2} {}^{(ri)}\big(v\otimes \Hom_{k[\A]}({}^{(r-rs)}\rho, {}^{(-ri)}\sigma)\big)\right)\;.$$
One defines a comparison map in the same fashion as the map of theorem \ref{thm-thm-F-lin}:
\begin{align}
\Ext^j_\gen(F^{(rs-r)},G^\ddag)\to \Ext^j_{k[\A]}(\rho^*F,\sigma^*G)\;.
\label{eq-gen-comp-dual}
\end{align}
The proof of the following corollary is similar to the proof of theorem \ref{thm-thm-F-lin} and is left to the reader.
\begin{thm}\label{thm-thm-gencomp}
Let $k$ be an infinite perfect field of characteristic $p$, containing a finite field $\Fq$ of cardinal $q=p^r$. 
Let $\A$ be a svelte additive $\Fq$-linear category, let $\rho,\sigma:\A\to \Vect_k$ be two $\Fq$-linear functors such that $\sigma$ has finite-dimensional values and $\Hom_{k[\A]}(\rho,\sigma)$ is finite-dimensional. Assume that $s$ and $e$ are positive integers such that 
\[\Ext^j_{k\otimes_\mathbb{Z}\A} \left(   {}^{(r-rs)}\rho,  {}^{(ri)}\sigma  \right)=0\]
for $0< j<e$ and $0\le i<s^2$. Then for all strict polynomial functors $F$ and $G$ of degrees less than $q^s$, the map \eqref{eq-gen-comp-dual} is an isomorphism if $j<e$ and injective if $j=e$.
\end{thm}

\begin{rk}\label{rk-limit-pb}
The finite dimensionality hypotheses on the values of $\sigma$ and on $\Hom_{k[\A]}(\rho,\sigma)$ are necessary in the proof of theorem \ref{thm-thm-F-lin} in order that i), ii) and iii) are satisfied. Without them, we would not obtain a commutative square \eqref{cd-dual} with vertical \emph{isomorphisms}. Similarly, the finite dimensionality hypotheses are needed for the proof of theorem \ref{thm-thm-gencomp}. Instead of dualizing, one could try to prove theorems \ref{thm-thm-F-lin} and \ref{thm-thm-gencomp} by a direct approach, following the same strategy as the proofs of theorems \ref{thm-gen-comp} and \ref{thm-F-lin-case}. However, such a direct approach seems to raise inextricable problems with (co)limits.
\end{rk}

\section{Some concrete calculations}
In this section, we apply the generalized comparison theorem (or the special cases given in section \ref{sec-special-cases}) to obtain new explicit homological computations.
\subsection{A sample of new functor homology calculations over $k[\Vect_k]$}\label{subsec-Pk}
Many concrete computations of generic $\Ext$ can be found in the literature (or quickly deduced from existing results). By contrast, no\footnote{Appart from the recent results of \cite{DT-add}, which are only valid for $\Ext$ between direct summands of tensor powers.} $\Ext$-computation in $k[\Vect_k]\Md$ is known when $k$ is an infinite perfect field of positive characteristic $p$. 
Thus, theorem \ref{thm-infinite-perfect-field-Ext} can be viewed as an efficient tool to obtain concrete $\Ext$-computations in $k[\Vect_k]\Md$ over an infinite field. We illustrate this idea here.

\begin{nota}
We will write $\Ext^*$ indifferently for $\Ext_\gen^*$ or for $\Ext^*_{k[\Vect_k]}$, since the two are isomorphic. If one interprets the statements of the section with $\Ext^*_\gen$, then these statements simply gather some well-known computations. But if one interprets the statements of the section with $\Ext^*_{k[\Vect_k]}$, then these statements give new homological computations which were previously out of reach. 
\end{nota}

We first generalize  the computations of \cite{FFSS} of the $\Ext$-algebras between symmetric, exterior and divided powers to an infinite perfect field $k$. To be more specific, let $C^*$ be a graded coalgebra in $\Pp_k$ and let $A^*$ be a graded algebra in $\Pp_k$. We consider the trigraded vector space
$$\mathrm{E}^*(C^*,A^*):=\bigoplus_{i,d,e\ge 0}\Ext^i(C^d,A^e)$$
equipped with the algebra structure given by convolution: 
$$\mathrm{E}^i(C^d,A^e)\otimes \mathrm{E}^j(C^f,A^g)\xrightarrow[]{\cup}\mathrm{E}^{i+j}(C^d\otimes C^f,A^e\otimes A^g)\to \mathrm{E}^{i+j}(C^{d+f},A^{e+g})\;.$$
The symmetric powers $S^*$, the exterior powers $\Lambda^*$ and the divided powers $\Gamma^*$ all have a canonical Hopf algebra structure, hence $\Ext$ between them have an algebra structure. 

For all vector spaces $V$ and all strict polynomial  functors $G$ over $k$, we denote by $G_V$ the 'functor with parameter $V$' defined by $G_V(-):=G(V\otimes_k-)$. If $A^*$ is a graded algebra, then $A^*_V$ is also a graded algebra.

\begin{thm}\label{thm-FFSS-infinite}
Let $k$ be an infinite perfect field of positive characteristic $p$, let $V$ be a $k$-vector space of finite dimension, and let $r$ be a nonnegative integer. Let 
\[V_{s,r}=\bigoplus_{i\ge 0}{}^{(r)}V[2ip^r+s(p^r-1),1,p^r]\]
be the trigraded $k$-vector space where each $^{(r)}V[\ell,m,n]$ denotes a copy of $^{(r)}V$ placed in tridegree $(\ell,m,n)$. Then we have isomorphisms of trigraded algebras, natural with respect to $V$:
\begin{align*}
&\mathrm{E}^*(\Gamma^{*(r)},S^*_V)\simeq S(V_{0,r})\;,
&&
\mathrm{E}^*(\Gamma^{*(r)},\Lambda^*_V)\simeq \Lambda(V_{1,r})
\;, \\
&\mathrm{E}^*(\Lambda^{*(r)},S^*_V)\simeq \Lambda(V_{0,r})\;,
&&
\mathrm{E}^*(\Lambda^{*(r)},\Lambda^*_V)\simeq \Gamma(V_{1,r})
\;, \\
&\mathrm{E}^*(S^{*(r)},S^*_V)\simeq \Gamma(V_{0,r})\;,
&&
\mathrm{E}^*(\Gamma^{*(r)},\Gamma^*_V)\simeq \Gamma(V_{2,r})
\;. 
\end{align*}
\end{thm}
\begin{proof}
The result is obtained by letting $s\to\infty$ in \cite[Thm 15.1]{TouzeBar}. 
\end{proof}
\begin{rk}
The pairs $(\Lambda^{*\,(r)},\Gamma^*_V)$, $(S^{*\,(r)},\Gamma^*_V)$ and $(S^{*\,(r)},\Lambda^*_V)$ do not appear in theorem \ref{thm-FFSS-infinite}. The reader may obtain the corresponding generic $\Ext$ by letting $s\to \infty$ in \cite[Thm 15.2 and 15.3]{TouzeBar}; we don't reproduce them here to save space. Additional computations of the same flavour, involving more general Hopf algebras, can be retrieved from \cite[Thm 7.1]{TouzeIMRN}.
\end{rk}

There is a general formula computing extensions between twisted strict polynomial functors, see \cite{Chalupnik}, \cite{TouzeUnivNew} and \cite{Touze-Survey}. Namely, if $V$ is a finite-dimensional graded $k$-vector space and $G$ is a strict polynomial functor, the functor with parameter $G_V(-):=G(V\otimes_k -)$ inherits a grading. It is the unique grading natural with respect to $G$ and $V$, which coincides with the usual grading on symmetric powers of a graded vector space if $G$ is a standard injective, see \cite[Section 2.5]{TouzeENS} and \cite[Section 4.2]{Touze-Survey}. Let $\mathbb{E}_r$ be the graded vector space
which equals $k$ in degrees $2i$ for $0\le i<p^r$ and which is zero in the other degrees. Then we have a graded isomorphism, where the degree on the right-hand side is obtained by totalizing the $\Ext$-degree with the degree of the functor $G_{\mathbb{E}_r}$ (that is, if $G_{\mathbb{E}_r}^j$ denotes the component of degree $j$ then the summand $\Ext^i_{\Pp_k}(F,G_{\mathbb{E}_r}^j)$ is placed in degree $i+j$):
\[\Ext^*_{\Pp_k}(F^{(r)},G^{(r)})\simeq \Ext^*_{\Pp_k}(F,G_{\mathbb{E}_r})\;.\]

We extend the parametrization of $G$ to infinite-dimensional graded vector spaces $v$ by letting 
$G_V:=\colim G_U$, where the colimit is taken over the poset of all finite-dimensional graded vector spaces $U\subset V$. By taking the colimit over $r$ in the previous isomorphism, and by using theorem \ref{thm-infinite-perfect-field-Ext} we obtain the following result (in which no Frobenius twist appear in the $\Ext$ of the right-hand side!).
\begin{thm}\label{thm-calcul-kProjk}
Let $k$ be an infinite perfect field of positive characteristic. Let $\mathbb{E}_\infty$ be the graded vector space equal to $k$ in even degrees and to $0$ in odd degrees. There is a graded isomorphism, natural with respect to the strict polynomial functors $F$ and $G$, and where the degree on the right-hand side is computed by totalizing the $\Ext$-degree with the degree of the functor $G_{\mathbb{E}_\infty}$:
\[\Ext^*(F,G)\simeq \Ext^*_{\Pp_k}(F,G_{\mathbb{E}_\infty})\;.\]
\end{thm}
\begin{rk}
If $G$ is $d$-homogeneous and $F=\Gamma^{d}$, the right-hand side in theorem \ref{thm-calcul-kProjk} is zero in positive $\Ext$-degrees, and in $\Ext$-degree zero it can be computed by \cite[Lm 2.10]{TouzeENS}. One obtains a graded isomorphism:
\[\Ext^*(\Gamma^{d},G)\simeq G_{\mathbb{E}_\infty}(k)=G({\mathbb{E}_\infty})\;.\]
\end{rk}

Duality between $\Ext$ and $\Tor$ allows to convert the results of this section into $\Tor$ computations. To be more specific, let $\Tor_*$ denote $\Tor_*^\gen$ or $\Tor_*^{k[\Vect_k]}$. If $F^*$ and $G^*$ are two graded strict polynomial functors over $k$, we denote by $\mathrm{T}_*(F^*,G^*)$ the trigraded vector space defined by 
\[\mathrm{T}_*(F^*,G^*)=\bigoplus_{i,d,e\ge 0}\Tor_i(F^{d\, \vee},G^e)\;,\]
where $F^{d\,\vee}$ is the contravariant strict polynomial functor associated to $F^{d\,\vee}$. 
\begin{cor}\label{cor-FFSS-infinite}
Let $k$ be an infinite perfect field of positive characteristic $p$, let $V$ be a $k$-vector space, and let $r$ be a nonnegative integer. Let
\[V_{s,r}=\bigoplus_{i\ge 0}{}^{(r)}V[2ip^r+s(p^r-1),1,p^r]\]
be the trigraded $k$-vector space where each $^{(r)}V[\ell,m,n]$ denotes a copy of $^{(r)}V$ placed in tridegree $(\ell,m,n)$. There are isomorphisms of trigraded vector spaces, natural with respect to $V$:

\begin{align*}
&\mathrm{T}_*(\Gamma^*_V,\Gamma^{*(r)})\simeq \Gamma(V_{0,r})\;,
&&
\mathrm{T}_*(\Lambda^*_V,\Gamma^{*(r)})\simeq \Lambda(V_{1,r})
\;, \\
&\mathrm{T}_*(\Gamma^*_V,\Lambda^{*(r)})\simeq \Lambda(V_{0,r})\;,
&&
\mathrm{T}_*(\Lambda^*_V,\Lambda^{*(r)})\simeq S(V_{1,r})
\;, \\
&\mathrm{T}_*(\Gamma^*_V,S^{*(r)})\simeq S(V_{0,r})\;,
&&
\mathrm{T}_*(S^*_V,\Gamma^{*(r)})\simeq S(V_{2,r})
\;.
\end{align*}
\end{cor}
\begin{proof}
If $V$ has finite dimension, the result is deduced from theorem \ref{thm-FFSS-infinite} by duality between $\Ext$ and $\Tor$. Each of the obtained isomorphisms extends to the case of infinite-dimensional $V$ by taking filtered colimits over the finite dimensional vector subspaces of $V$. 
\end{proof}

\begin{cor}\label{cor-calcul-kProjk}
Let $k$ be an infinite perfect field of positive characteristic. Let $\mathbb{T}_\infty$ be the graded vector space equal to $k$ in even degrees and to $0$ in odd degrees. There is a graded isomorphism, natural with respect to the strict polynomial functors $F$ and $G$, and where the degree on the right-hand side is computed by totalizing the $\Tor$-degree with the degree of the functor $F_{\mathbb{T}_\infty}^\vee$:
\[\Tor_*(F^\vee,G)\simeq \Tor_*^{\Pp_k}((F_{\mathbb{T}_\infty})^\vee,G)\;.\]
\end{cor}


\subsection{Calculations over $k[\Proj_R]$ and homology of $\GL_n(R)$}\label{subsec-ex-GL}
The next result, combined with the explicit computations of generic $\Tor$ from the previous section, yields many  explicit computations of $\Tor$ over $k[\Proj_R]$.
\begin{thm}\label{thm-calculs-kProjR}
Assume that $R$ is an algebra over an infinite perfect field $k$ of positive characteristic. Let $\ell$ and $m$ be two positive integers, and let $\pi$ and $\rho$ be the additive functors:
\[
\begin{array}{cccc}
\pi: &\Proj_R^\op & \to & k\Md\\
& v & \mapsto & \Hom_R(v,R^\ell)
\end{array}\;,\qquad
\begin{array}{cccc}
\rho: &\Proj_R & \to & k\Md\\
& v & \mapsto & \Hom_R(R^m,v)
\end{array}\;.
\]
There is a graded isomorphism, natural with respect to $F$, $G$, $R^\ell$ and $R^m$:
\[ \Tor_*^{k[\Proj_R]}(\pi^*F,\rho^*G)\simeq\Tor_*^{\gen}((F_{\Hom_R(R^m,R^\ell)})^\vee,G)\;.
\]
\end{thm}
\begin{proof}
The result will follow from theorem \ref{thm-F-lin-case}, once we have proved that the graded vector space $\Tor_*^{k\otimes_\mathbb{Z}\Proj_R}(\pi,\rho)$ is equal to zero in positive degrees, and to $\Hom_R(R^m,R^\ell)$ in degree zero.
By additivity of $\Tor$, it suffices to prove this for $\ell=m=1$. Then, the Eilenberg-Watts theorem yields an isomorphism:
\[\Tor_*^{k\otimes_\mathbb{Z}\Proj_R}(\pi,\rho)\simeq \Tor_*^{k\otimes_\mathbb{Z}R}(R,R)\;.\]
If $P\to k$ is a free resolution in the category of $(k,k)$-bimodules, then by applying base change along the flat morphism $k\to R$ on the left action of $k$ we obtain a free resolution $R\otimes_k P \to R\otimes_k k=R$ in the category of $(R,k)$-bimodules. Thus $\Tor_*^{k\otimes_\mathbb{Z}R}(R,R)$ is the homology of the complex:
\[ R\otimes_{R\otimes_\mathbb{Z}k}(R\otimes_k P)\simeq  R\otimes_k(k\otimes_{k\otimes_\mathbb{Z}k}P)\;.\]
Hence $\Tor_*^{k\otimes_\mathbb{Z}R}(R,R)=R\otimes_k\mathrm{HH}_*(k)$ by the universal coefficient theorem, where $\mathrm{HH}_*(k)$ is the Hochschild homology of $k$. The latter is equal to $k$ in degree zero and it is zero in higher degrees since $k$ is a perfect field (see e.g. the proof of  \cite[Lm 6.1]{DT-add} for a detailed argument). Whence the result.
\end{proof}

The infinite general linear group is defined as $\GL_\infty(R)=\bigcup_{n\ge 1}\GL_n(R)$, where each group $\GL_n(R)$ is viewed as the subgroup of $\GL_{n+1}(R)$ of block matrices of the form $\begin{bsmallmatrix} M & 0\\ 0& 1\end{bsmallmatrix}$.

Every functor $T:\Proj_R\to k\Md$ yields a left $k$-linear representation $T_\infty$ of $\GL_\infty(R)$. To be more specific, as a $k$-vector space $T_\infty= \bigcup_{n\ge 1}T(R^n)$, where each $T(R^n)$ is seen as a subspace of $T(R^n)$ via the map $T(\iota_n)$ where $\iota_n$ is the standard inclusion of $R^n$ as the first $n$ coordinates of $R^{n+1}$. The action of $\GL_\infty(R)$ is the unique action such that each $g\in \GL_n(R)$ acts as $T(g)$ on $T(R^n)$. Similarly, every $T':\Proj_R^\op\to k\Md$ yields a left $k$-linear representation $T'_\infty$ of $\GL_\infty(R)$, which equals the increasing union of the vector spaces $T'(R^n)$ (where each $T'(R^n)$ is now seen as a subspace of $T'(R^{n+1})$ via $T'(p_n)$ with $p_n$ the canonical projection of $R^{n+1}$ onto its first $n$ coordinates), with action of $\GL_\infty(R)$ such that each $g\in \GL_n(R)$ acts as $T'(g^{-1})$ on $T'(R^n)$.

If we take $\pi$ and $\rho$ as in theorem \ref{thm-calculs-kProjR}, then $\rho_\infty$ is the direct sum of $m$ copies of the $R$-module with infinite countable basis $R^\infty$ (viewed as a $k$-vector space by restriction along $k\to R$), each of these copies endowed with the action of $\GL_\infty(R)$ by matrix multiplication: $g\cdot v := gv$, while $\pi_\infty$ is the direct sum of $\ell$ copies of $R^\infty$ with action of $\GL_\infty(R)$ given by multiplication by inverse transpose: $g\cdot v:= (g^{-1})^{\mathrm{T}}v$. Moreover, for all strict polynomial functors $F$ and $G$ the representations $(\pi^*F)_\infty$ and $(\rho^*G)_\infty$ are respectively equal to $\overline{F}(\pi_\infty)$ and $\overline{G}(\rho_\infty)$. The next result follows from \cite[Thm 5.6]{DjaR} (due to Scorichenko, unpublished) and theorem \ref{thm-calculs-kProjR}.
\begin{cor}\label{cor-HGL}
Under the hypotheses of theorem \ref{thm-calculs-kProjR}, there is a graded isomorphism, natural with respect to $F$, $G$, $R^\ell$ and $R^m$:
\begin{align*}
\rmH_*(\GL_\infty(R);\overline{F}(\pi_\infty)\otimes_k \overline{G}(\rho_\infty))\simeq \rmH_*(\GL_\infty(R);k)\otimes_k \Tor_*^{\gen}((F_{\Hom_R(R^m,R^\ell)})^\vee,G)\;.
\end{align*}
\end{cor}

Concrete computations of the generic $\Tor$ appearing in the right-hand side of corollary \ref{cor-HGL} are given in section \ref{subsec-Pk}. For example, the decomposition $\Lambda^d(\pi_\infty\oplus\rho_\infty)=\bigoplus_{i+j=d}\Lambda^i(\pi_\infty)\otimes_k\Lambda^j(\rho_\infty)$, together with corollary \ref{cor-HGL} and the $\Tor$-computation of corollary \ref{cor-FFSS-infinite} give the computation of example \ref{ex-intro-1} from the introduction.

\section{Rational versus discrete cohomology}\label{sec-rat-disc}
Let $G$ be an algebraic group over an infinite perfect field $k$ of positive characteristic $p$, let $\mathbf{Mod}_G$ denote the category of all $k$-linear representations of the discrete group $G$, and let $\mathbf{Rat}_G$ denote the full subcategory of $\mathbf{Mod}_G$ on the rational representations as in \cite{Jantzen}. 
Extensions between two rational representations $V$ and $W$ can be computed in $\mathbf{Rat}_G$ or $\mathbf{Mod}_G$. In the sequel, we let 
\begin{align*}
&\EExt_G^*(V,W):=\Ext^*_{\mathbf{Rat}_G}(V,W)\;,&&&\HH^*(G;W):=\EExt_G^*(k,W)\;,\\
&\Ext_G(V,W):=\Ext^*_{\mathbf{Mod}_G}(V,W)\;,&&&\rmH^*(G;W):=\Ext_G^*(k,W)\;.
\end{align*}
There is a canonical morphism:
\begin{align*}
\EExt^*_G(V,W)\to \Ext_G^*(V,W) 
\end{align*}
which is far from being an isomorphism in general. 
An important difference between the source and the target of the canonical morphism is the behaviour of Frobenius morphisms. Namely assume that $G$ is one of the classical groups $\GL_n(k)$, $\Sp_{2n}(k)$ or $\orth_{n,n}(k)$ (and if $k$ has characteristic $\ne 2$ in the latter case), let $\phi:G\to G$, $[a_{ij}]\mapsto [a_{ij}^p]$, denote the morphism of algebraic groups induced by the Frobenius endomorphism of $k$, and let $V^{[r]}$ denote the restriction of $V$ along $\phi^r$. We have a commutative ladder whose horizontal arrows are induced by restriction along $\phi$:
\[
\begin{tikzcd}[column sep=small]
\EExt^i_G(V,W)\ar{r}\ar{d}&\cdots \ar{r}&\EExt^i_G(V^{[r]},W^{[r]})\ar{d}\ar{r}&\EExt^i_G(V^{[r+1]},W^{[r+1]})\ar{r}\ar{d}&\cdots\\
\Ext^i_G(V,W)\ar{r}&\cdots \ar{r}&\Ext^i_G(V^{[r]},W^{[r]})\ar{r}&\Ext^i_G(V^{[r+1]},W^{[r+1]})\ar{r}&\cdots\;.
\end{tikzcd}
\]
Since $k$ is perfect, $\phi$ has an inverse $\phi^{-1}([a_{ij}])=[a_{ij}^{-p}]$ so that the morphisms in the bottom row are all isomorphisms. However $\phi$ has no inverse in the sense of algebraic groups; instead it is known \cite[II 10.14]{Jantzen} that the morphisms in the top row are all injective (and that they are isomorphisms for $r\gg 0$ only). Hence, the ladder yields canonical maps:
\begin{align}
\Phi_{k,G}:\EExt^*_{G}(V^{[r]},W^{[r]})\to \Ext^*_G(V,W)\;.
\label{eqn-compare-gen}
\end{align}

Embed $G=\GL_n(k)$, $\Sp_{2n}(k)$ or $\orth_{n,n}(k)$ into the multiplicative monoid of matrices $\M_m(k)$ in the usual way (here $m=n$ for $\GL_n(k)$ and $m=2n$ for the symplectic or orthogonal groups). Then a finite-dimensional representation $V$ of $G$ is called \emph{polynomial of degree less or equal to $d$} if it is the restriction to $G$ of a representation of $\M_m(k)$, such that the coordinate maps of the action morphism
$$\rho_V: \M_m(k)\to \End_k(V)\simeq \M_{\dim V}(k)$$
are polynomials of degree less or equal to $d$ of the $m^2$ entries of $[a_{ij}]\in \M_m(k)$.  The first main result of the section is the following theorem.
\begin{thm}\label{thm-comp-GL}
Assume that the perfect field $k$ is infinite, and that $V$ and $W$ are two polynomial representations of $G=\GL_n(k)$ of degrees less or equal to $d$. Let $r$ be a nonnegative integer such that $n\ge \max\{dp^r,4p^r+2d+1\}$. Then the map \eqref{eqn-compare-gen} is an isomorphism in degrees $i<2p^r$ and it is injective in degree $i=2p^r$.
\end{thm}
The second main result of the section is an analogue of theorem \ref{thm-comp-GL} for orthogonal and symplectic groups. Here we take $V=k$ hence $V^{[r]}=k$ and the comparison map \eqref{eqn-compare-gen} can be rewritten as a map
\begin{align}
\Phi_{k,G}:\HH^*(G;W^{[r]})\to \rmH^*(G;W)\;.
\label{eqn-compare-osp}
\end{align}
\begin{thm}\label{thm-comp-OSp}
Let $k$ be an infinite perfect field of odd characteristic $p$, let $G=\Sp_{2n}(k)$ or $\orth_{n,n}(k)$ and let $W$ be a polynomial representation of degree less or equal to $d$. Assume that $2n\ge \max\{dp^r, 8p^r+4+2d\}$.
Then the comparison map \eqref{eqn-compare-osp} is an isomorphism in degrees $i<2p^r$ and it is injective in degree $i=2p^r$.
\end{thm}

The remainder of the section is devoted to the proof of theorems \ref{thm-comp-GL} and \ref{thm-comp-OSp}. 

\subsection{Proof of theorem \ref{thm-comp-GL}}
Let us say that a map between cohomologically graded vector spaces is \emph{$e$-connected} if it is an isomorphism in degrees less than $e$ and if it is injective in degree $e$. Thus, we have to show that the canonical map 
\[\EExt^*_{\GL_n(k)}(V^{[r]},W^{[r]})\to \Ext_{\GL_n(k)}^*(V^{[r]},W^{[r]})\]
is $2p^r$-connected. 
Since $n>dp^r$, we know from \cite[Lm 3.4]{FS} that there are uniquely determined strict polynomial functors $F$ and $G$ with degrees less or equal to $d$ and with finite-dimensional values, such that $F^{(r)}(k^n)\simeq V^{[r]}$ and $G^{(r)}(k^n)=W^{[r]}$. 
We consider the commutative diagram, in which the vertical arrows are induced by evaluation on $k^n$ and the horizontal arrows are the canonical maps (induced by forgetting the `scheme structure'):
\[
\begin{tikzcd}
\Ext^*_{\Pp_k}(F^{(r)},G^{(r)})\ar{d}{(a)}\ar{r}{(b)}& \Ext^*_{k[\Vect_k]}(F^{(r)},G^{(r)})\ar{d}{(c)}\\
\EExt^*_{\GL_n(k)}(F^{(r)}(k^n),G^{(r)}(k^n))\ar{r}&\Ext_{\GL_n(k)}^*(F^{(r)}(k^n),G^{(r)}(k^n))
\end{tikzcd}\;.
\]
The map $(a)$ is a graded isomorphism by \cite[Cor 4.3.1]{FS} since $n\ge dp^r$. The map $(b)$ is $2p^r$-connected by theorem \ref{thm-infinite-perfect-field-Ext} and by the bound of stabilization for generic $\Ext$ given in proposition-definition \ref{pdef-gen-Ext}. Thus, the following proposition finishes the proof of theorem \ref{thm-comp-GL}.
\begin{pr}\label{pr-FH-GL}
For all strict polynomial functors $F$ and $G$ of degree less or equal to $d$ with finite-dimensional values, if $n\ge 2e+2d+1$, then the map $(c)$ is $e$-connected.
\end{pr}
\begin{proof}
Let $T$ and $T'$ be two strict polynomial functors, respectively covariant and contravariant. Then $T$ defines left $k[\GL_n(k)]$-modules $T(k^n)$ (where $g\in \GL_n(k)$ acts as $T(g)$ on $T(k^n)$) for all positive $n$, and the colimit of these representations yields a left $k[\GL_\infty]$-module $T_\infty$ (as explained in section \ref{subsec-ex-GL}). Similarly, $T'$ defines right $k[\GL_n(k)]$-modules $T'(k^n)$ and a right $k[\GL_\infty]$-module $T'_\infty$. 
We consider the commutative triangle
\[
\begin{tikzcd}
{\begin{array}{c}
\Tor_*^{k[\GL_\infty(R)]}(T'_\infty,T_\infty)=\\
\colim_{n}\Tor_*^{k[\GL_n(R)]}(T'(k^n),T(k^n))
\end{array}
}\ar{rd}{(**)}\\
\Tor_*^{k[\GL_n(R)]}(T'(k^n),T(k^n))\ar{r}\ar{u}{(*)}&\Tor_*^{k[\Vect_k]}(T',T)&
\end{tikzcd}
\] 
in which the vertical arrow $(*)$ is the canonical arrow of the colimit, the horizontal arrow is induced by evaluation on $k^n$ and the last arrow $(**)$ is the one given by the universal property of colimits.


The homology $\rmH_*(\GL_\infty(k);k)$ is zero in positive degrees by lemma \ref{lm-vanish-HGL}, hence it follows from \cite[Thm 5.6]{DjaR} that the map $(**)$ is a graded isomorphism. 

Now strict polynomial functors yield split coefficient systems of finite degree in the sense of \cite[Def 4.10]{RWW}, thus it follows from \cite[Thm 5.11]{RWW} that the vertical map $(*)$ is an isomorphism in degrees less than $e$ and an epimorphism in degree $e$ provided $n\ge 2e+\delta +1$, where $\delta$ is the degree of the coefficient system $T'(k^n)\otimes_k T(k^n)$, $n\ge 1$. 
Using duality between $\Ext$ and $\Tor$, we deduce that the map induced by evaluation on $k^n$:
\[\Ext^*_{k[\Vect_k]}(T,D_kT')\to \Ext^*_{\GL_n(k)}(T(k^n),D_kT'(k^n))\]
is $e$-connected, provided $n\ge 2e+\delta +1$. 

Now take $T=F^{(r)}$ and $T'=D_kG^{(r)}$. Then $G^{(r)}=D_kT'$ since $G$ has finite dimensional values, and the proposition follows from the fact that if $F$ and $G$ are both strict polynomial functors of degree less or equal to $d$, then the degree of the coefficient system $T'(k^n)\otimes_k T(k^n)$, $n\ge 1$, is less or equal to $2d$. 
\end{proof}

\begin{lm}\label{lm-vanish-HGL}
Let $k$ be an infinite perfect field of positive characteristic $p$. The mod $p$ homology of $\GL_\infty(k)$ is zero in positive degrees.
\end{lm}
\begin{proof}
The result follows from the $p$-local Hurewicz theorem \cite[Thm 1.8.1]{Neisendorfer} and the unique $p$-divisibility of the homotopy groups of  $B\GL(k)^+$ over an infinite perfect field of characteristic $p$ \cite[Lm 5.2 and Cor 5.5]{Kratzer}. 
\end{proof}

\subsection{Proof of theorem \ref{thm-comp-OSp}}
The proof of theorem \ref{thm-comp-OSp} follows the same idea as the $\GL_n(k)$ case. Since $2n>dp^r$, we know from \cite[Lm 3.4]{FS} that there is a uniquely determined strict polynomial functor $F$ with degree less or equal to $d$ and with finite-dimensional values such that $F^{(r)}(k^{2n})=W^{[r]}$. Thus we have to show that the canonical map
\[\HH^*(G;F^{(r)}(k^{2n}))\to  \rmH^*(G;F^{(r)}(k^{2n}))\]
is $2p^r$-connected. Since every strict polynomial functor splits as a direct sum of homogeneous strict polynomial functors, we may assume further that $F$ is homogeneous of degree $d$.

Assume that $G=\orth_{n,n}(k)$ or $G=\Sp_{2n}(k)$. 
We associate to $G$ a `characteristic functor' $X:\Vect_k\to \Vect_k$, namely $X=S^2$ in the orthogonal case  $X=\Lambda^2$ in the symplectic case. Let $k[X]$ denote the functor which sends a vector space $v$ to the vector space with basis $X(v)$ and let $\Gamma^{d}X$ denote the functor such that $\Gamma^{d}X(v)=\Gamma^{d}(X(v))$. 
We set:
\begin{align*}
&\rmH^*_X(k[\Vect_k];F^{(r)}):=\Ext^*_{k[\Vect_k]}(k[X],F^{(r)}), \\
&\rmH^*_X(\Pp_k;F^{(r)}):=\begin{cases}
\Ext^*_{\Pp_k}(\Gamma^{p^{r}d/2}X,F^{(r)}) & \text{ if $d$ is even,}\\
0 &\text{ if $d$ is odd.}
\end{cases}
\end{align*}
We consider the commutative square
\[
\begin{tikzcd}
\rmH^*_{X}(\Pp_k;F^{(r)})\ar{d}{(a)}\ar{r}{(b)}& \rmH^*_{X}(k[\Vect_k];F^{(r)})\ar{d}{(c)}\\
\HH^*(G;F^{(r)}(k^{2n}))\ar{r}{\Phi_{k,G}}&\rmH^*(G;F^{(r)}(k^{2n}))
\end{tikzcd}
\]
in which maps $(a)$, $(b)$ and $(c)$ are defined as follows. The bilinear form defining $G$ yields an element $\omega\in X(k^{2n})$ invariant under the action of $G$, hence a $G$-equivariant morphism 
$\lambda:k\to k[X(k^{2n})]$ such that $\lambda(1)=\omega$. The map $(c)$ is the composition
\[\Ext^*_{k[\Vect_k]}\left(k[X],F^{(r)}\right)\to \Ext^*_{G}\left(k[X(k^{2n})],F^{(r)}(k^{2n})\right)\to\Ext^*_{G}\left(k,F^{(r)}(k^{2n})\right) \]
where the first map is given by evaluation on $k^{2n}$, and the second one is pullback along $\lambda$.
Similarly, the map $(a)$ is zero if $d$ is odd, and if $d$ is even, it is induced by evaluation on $k^{2n}$ and pullback along the $G$-equivariant morphism $\underline{\lambda}:k\to \Gamma^{d/2}(X(k^{2n}))$ such that $\underline{\lambda}(1)=\omega^{\otimes d/2}$.
Finally the map $(b)$ is zero if $d$ is odd, and if $d$ is even, it equals the composition
\[\rmH^*_{X}(\Pp_k;F^{(r)})\to \Ext^*_{k[\Vect_k]}(\Gamma^{p^rd/2}X,F^{(r)})\to  \rmH^*_{X}(k[\Vect_k];F^{(r)}) \]
where the first map is induced by the forgetful functor $\Pp_k\to k[\Vect_k]\Md$ and the second one is induced by pullback along the natural transformation $k[X]\to \Gamma^{p^rd/2}X$ which sends every element $x\in X(v)$ to $x^{\otimes p^rd/2}\in\Gamma^{p^rd/2}(X(v))$. 
 
Since $2n\ge dp^r$, the map $(a)$ is an isomorphism by 
\cite[Thm 3.17]{TouzeClassical} or \cite[Thm 3.24]{TouzeClassical}. The map $(b)$ is $2p^r$-connected by proposition \ref{pr-PF-OSp} below, and the map $(c)$ is $2p^r$-connected by proposition \ref{pr-FH-OSp} below. Thus $\Phi_{k,G}$ is $2p^r$-connected, which finishes the proof of theorem \ref{thm-comp-OSp}.

\begin{pr}\label{pr-PF-OSp}
Let $k$ be a field of odd characteristic. For all strict polynomial functors $F$ with finite-dimensional values, the map $(b)$ is $2p^r$-connected.
\end{pr}
\begin{proof}
Assume first that $d$ is even. Let $\otimes^2$ be the second tensor power, i.e. $\otimes^2(v)=v^{\otimes 2}$. We consider the map 
\begin{align}
\Ext^*_{\Pp_k}(\Gamma^{p^rd/2}\otimes^2,F^{(r)})
\to \Ext^*_{k[\Vect_k]}(k[\otimes^2],F^{(r)})\label{eqn-HT2}
\end{align}
induced by the forgetful functor $\Pp_k\to k[\Vect_k]\Md$ and by pullback along the natural transformation $k[\otimes^2]\to \Gamma^{p^rd/2}\otimes^2$ which sends every element $x\in X(v)$ to $x^{\otimes  p^rd/2}\in\Gamma^{p^rd/2}(X(v))$. Consider the action of the symmetric group $\Si_2$ on $\otimes^2$ such that $(\otimes^2)^{\Si_2}=X$. Then  the action of $\Si_2$ on $\otimes^2$ induces an action of $\Si_2$ on the source and the target of \eqref{eqn-HT2}. Since $p$ is odd, $\Si_2$-fixed points are  direct summands, hence the map $(b)$ is a retract of the map \eqref{eqn-HT2}. Thus it suffices to show that the map \eqref{eqn-HT2} is $2p^r$-connected.

Let $\Pp_k(2)$ denote the category of strict polynomial bifunctors with two covariant variables. Let $\boxtimes^2$ be the bifunctor such that $\boxtimes^2(v,w)=v\otimes_k w$, let $B$ be the strict polynomial bifunctor such that $B(v,w)=F(v\oplus w)$, and let $B^{(r)}$ denote the bifunctor obtained by precomposing each variable of $B$ by the Frobenius twist $I^{(r)}$, that is $B^{(r)}(v,w)=B({}^{(r)}v,{}^{(r)}w)$.
By using adjunction between sum and diagonal as in \cite[(1.7.1) p. 672]{FFSS}, the map \eqref{eqn-HT2} identifies with the map
\begin{align}
\Ext^*_{\Pp_k(2)}(\Gamma^{p^rd/2}\boxtimes^2,B^{(r)})
\to \Ext^*_{k[\Vect_k\times\Vect_k]}(k[\boxtimes^2],B^{(r)})\label{eqn-HT2bis}
\end{align} 
induced by the forgetful functor $\Pp_k(2)\to k[\Vect_k\times\Vect_k]\Md$ and by pullback along the natural transformation $k[\boxtimes^2]\to \Gamma^{p^rd/2}\boxtimes^2$ wich sends every basis element $x\in v\otimes_k w$ to $x^{\otimes p^rd/2}\in\Gamma^{p^rd/2}(v\otimes_k w)$. 
Thus, it suffices to prove that the map \eqref{eqn-HT2bis} is $2p^r$-connected for all homogeneous strict polynomial bifunctors $B$ of degree $d$, with finite dimensional values.

Every such $B$ has a coresolution by strict polynomial bifunctors of separable type, that is, of the form $(F_1\boxtimes F_2)(v,w)=F_1(v)\otimes_kF_2(w)$, where $F_1$ and $F_2$ are homogeneous strict polynomial functors of degrees $d_1$ and $d_2$ with $d_1+d_2=d$ and with finite dimensional values. Thus by a standard spectral sequence argument, it suffices to prove that the map \eqref{eqn-HT2bis} is $2p^r$-connected when $B=F_1\boxtimes F_2$. 
In this case, by dualizing the first variable and by using \cite[thm 1.5]{FF} and its equivalent for ordinary functors, The map \eqref{eqn-HT2bis} identifies with the map 
\begin{align}\Ext^*_{\Pp_k}(F_1^{\sharp \,(r)},F_2^{(r)})\to \Ext^*_{k[\Vect_k]}(F_1^{\sharp \,(r)},F_2^{(r)})
\label{eqn-HT2ter}
\end{align}
induced by the forgetful map $\Pp_k\to k[\Vect_k]\Md$, and where $F_1^\sharp$ refers to the Kuhn dual of $F_1$ \cite[Prop 2.6]{FS}. 
Finally, the map \eqref{eqn-HT2ter} is $2p^r$-connected as a consequence of theorem \ref{thm-infinite-perfect-field-Ext} and of the bound of stabilization for generic $\Ext$ given in proposition-definition \ref{pdef-gen-Ext}, which finishes the proof of proposition \ref{pr-PF-OSp} for $d$ even. 

Assume now that $d$ is odd. Then proposition \ref{pr-PF-OSp} is equivalent to the vanishing of $\rmH^*(k[X];F^{(r)})$. This vanishing is proved with the same reduction steps as in the case $d$ even. Namely, $\rmH^*(k[X];F^{(r)})$ is a direct summand of $\Ext^*_{k[\Vect_k]}(k[\otimes^2],F^{(r)})$, which is isomorphic to $\Ext^*_{k[\Vect_k\times\Vect_k]}(k[\boxtimes^2],B^{(r)})$ with $B(v,w)=F(v\oplus w)$. Thus, it suffices to prove that the latter is zero for all homogeneous strict polynomial bifunctors $B$ of degree $d$. The case  $B=F_1\boxtimes F_2$ is sufficient, where $F_1$ and $F_2$ are homogeneous strict polynomial functors of degrees $d_1$ and $d_2$ with $d_1+d_2=d$ and with finite dimensional values, and in that case, $\Ext^*_{k[\Vect_k\times\Vect_k]}(k[\boxtimes^2],B^{(r)})$ is isomorphic to $\Ext^*_{k[\Vect_k]}(F_1^{\sharp (r)},F_2^{(r)})$. We observe that $d_1\ne d_2$ because $d$ is odd. Thus these extensions are zero by the same argument as in the proof of lemma \ref{lm-homogeneity-vanish}.
\end{proof}

\begin{pr}\label{pr-FH-OSp}
Let $k$ be a perfect field of odd characteristic.
For all strict polynomial functors $F$ of degree less or equal to $d$, the map $(c)$ is $e$-connected provided $2n> 4e + 2d+ 4$.
\end{pr}
\begin{proof}
We proceed exactly in the same way as for the proof of proposition \ref{pr-FH-GL}, using lemma \ref{lm-vanish-HOSp} below for the vanishing of $\rmH_*(G;k)$, the stable homology computations of \cite[Cor 5.4]{DjaR} and the homological stabilization result of \cite[Thm 5.15]{RWW}.
\end{proof}
\begin{lm}\label{lm-vanish-HOSp}
Let $k$ be a perfect field of odd characteristic $p$. Then the mod $p$ homology of the groups $\Sp_\infty(k)$ and $\orth_{\infty,\infty}(k)$ is zero in positive degrees.
\end{lm}
\begin{proof}
Let $G=\Sp_\infty(k)$ or $\orth_{\infty,\infty}(k)$. By the universal coefficient theorem, it is equivalent to prove that $\rmH_i(G;\mathbb{Z})$ is uniquely $p$-divisible for $i>0$. 
If $A$ is an abelian group, we let $A[1/2]$ denote the tensor product $A\otimes_{\mathbb{Z}}\mathbb{Z}[1/2]$. Since $p$ is odd, $A$ is uniquely $p$-divisible if and only if $A[1/2]$ is uniquely $p$-divisible. And since $\mathbb{Z}[1/2]$ is flat we have $\rmH_*(G;\mathbb{Z})[1/2]=\rmH_*(G;\mathbb{Z}[1/2])$. Thus the statement of the lemma is equivalent to $\rmH_i(G;\mathbb{Z}[1/2])$ being uniquely $p$-divisible for $i>0$.

Since $k$ is a field of odd characteristic, the Witt groups $W(k)$ are $2$-primary torsion groups \cite[Chap. 2, Thm 6.4]{Scharlau}. Thus by \cite[Thm 3.18]{Karoubi}
$\rmH_*(G;\mathbb{Z}[1/2])$ is equal to $T_*(k)$, that is, to the homology of a space $\mathcal{C}(k)$ which is a retract of the localized classifying space $({B\mathcal{P}_k'}^+)_{(2)}$, see \cite[p. 253]{Karoubi} for the latter point. Since the integral homology of $({B\mathcal{P}_k'}^+)_{(2)}$ is equal to $\rmH_*({B\mathcal{P}_k'}^+;\mathbb{Z})[1/2]$, the lemma will be proved if we can prove that ${B\mathcal{P}_k'}^+$ has uniquely $p$-divisible positive integral homology groups.

But ${B\mathcal{P}_k'}^+$ has the weak homotopy type of $K_{0}(k)\times B\GL_\infty(k)^+$, which is uniquely $p$-divisible by lemma \ref{lm-vanish-HGL} and by the universal coefficient theorem.
\end{proof}

\begin{rk}[B. Calm\`es]
Instead of relying on \cite{Karoubi}, one could prove lemma \ref{lm-vanish-HOSp} by using the formula of \cite[Rk 7.8]{Schlichting}, which says that after tensoring by $\mathbb{Z}[1/2]$, the hermitian $K$-theory (hence \cite[App A]{Schlichting} the homotopy groups of $BG^+$ for $G=\Sp_\infty(k)$ or $\orth_{\infty,\infty}(k)$) is the direct sum of a term computed from the $K$-theory of $k$ and a term given by Balmer's Witt groups of $k$ tensored with $\mathbb{Z}[1/2]$. 
\end{rk}

\bibliographystyle{plain}

\bibliography{biblio-DT-Homology-part3.bib}

\end{document}